\documentclass[a4paper,11pt,reqno]{amsart} 
\usepackage{amssymb,amsthm,amsmath,amsxtra,dsfont,bbm,stix,color}
\usepackage{hyperref,cleveref} 

\usepackage{environ} 
\usepackage{pgfkeys} 

\usepackage{mathtools} 
\usepackage{stmaryrd} 
\usepackage{braket} 
\usepackage{enumitem}
\usepackage{mathrsfs}

\allowdisplaybreaks

\numberwithin{equation}{section}
\makeatletter

\makeatletter
\def\@tocline#1#2#3#4#5#6#7{\relax
  \ifnum #1>\c@tocdepth 
  \else
    \par \addpenalty\@secpenalty\addvspace{#2}%
    \begingroup \hyphenpenalty\@M
    \@ifempty{#4}{%
      \@tempdima\csname r@tocindent\number#1\endcsname\relax
    }{%
      \@tempdima#4\relax
    }%
    \parindent\z@ \leftskip#3\relax \advance\leftskip\@tempdima\relax
    \rightskip\@pnumwidth plus4em \parfillskip-\@pnumwidth
    #5\leavevmode\hskip-\@tempdima
      \ifcase #1
       \or\or \hskip 1em \or \hskip 2em \else \hskip 3em \fi%
      #6\nobreak\relax
      \dotfill
      \hbox to\@pnumwidth{\@tocpagenum{#7}}
    \par
    \nobreak
    \endgroup
  \fi}
\makeatother

\theoremstyle{plain}
\newtheorem{theorem}{Theorem}[section]
\newtheorem{lemma}[theorem]{Lemma}
\newtheorem{corollary}[theorem]{Corollary}
\newtheorem{proposition}[theorem]{Proposition}

\theoremstyle{definition}
\newtheorem{definition}[theorem]{Definition}

\newtheorem{notation}[theorem]{Notation}

\theoremstyle{remark}
\newtheorem{remark}[theorem]{Remark}

\newcommand{\dd}{\mathrm{d}}
\newcommand{\bx}{\mathbf{x}}
\newcommand{\br}{\mathbf{r}}
\newcommand{\bR}{\mathbf{R}}
\newcommand{\bp}{\mathbf{p}}
\newcommand{\by}{\mathbf{y}}
\newcommand{\bz}{\mathbf{z}}
\newcommand{\bX}{\mathbf{X}}
\newcommand{\bA}{\mathbf{A}}

\newcommand{\1}{\mathbb{1}}
\newcommand{\R}{\mathbb{R}}
\newcommand{\C}{\mathbb{C}}
\newcommand{\N}{\mathbb{N}}
\newcommand\norm[1]{\left\lVert#1\right\rVert}
\newcommand\abs[1]{\left|#1\right|}
\newcommand\prth[1]{\left(#1\right)}
\newcommand\sbra[1]{\left[#1\right]}
\newcommand\matrx[1]{\begin{matrix}#1\end{matrix}}
\newcommand\intint[1]{\left\llbracket #1\right\rrbracket}

\newcommand{\iintr}{\iint\displaylimits}

\newcommand\sett[1]{\left\{#1\right\}}
\DeclareMathOperator{\limit}{\to}

\DeclareMathOperator{\eqvl}{\sim}

\newcommand\syst[1]{\begin{cases}#1\end{cases}}
\newcommand{\floor}[1]{\left\lfloor #1 \right\rfloor}

\newcommand{\bk}[2]{\left< #1 \middle| #2 \right>}
\renewcommand{\bra}[1]{\left< #1 \right|}
\renewcommand{\ket}[1]{\left| #1 \right>}
\newcommand{\Tr}[1]{\mathrm{Tr}\sbra{#1}}
\newcommand{\tr}{\mathrm{Tr}}

\crefformat{equation}{#2(#1)#3} 

\newcounter{sProof}[section] 
\newcounter{ssProof}[section] 

\begin{document}

    \title{Gyrokinetic limit of the 2D Hartree equation in a large magnetic field}

\author[D. P\'erice]{Denis P\'erice}
\address{Constructor university Bremen }
\email{dperice@constructor.university}

\author[N. Rougerie]{Nicolas Rougerie}
\address{Ecole Normale Sup\'erieure de Lyon \& CNRS,  UMPA (UMR 5669)}
\email{nicolas.rougerie@ens-lyon.fr}

\date{December, 2025}

\begin{abstract}
We study the dynamics of  two-dimensional interacting fermions submitted to a homogeneous transverse magnetic field. We consider a large magnetic field regime, with the gap between Landau levels set to the same order as that of potential energy contributions. Within the mean-field approximation, i.e. starting from Hartree's equation for the first reduced density matrix, we derive a drift equation for the particle density. We use vortex coherent states and the associated Husimi function to define a semi-classical density almost satisfying the limiting equation. We then deduce convergence of the density of the true Hartree solution by a Dobrushin-type stability estimate.
\end{abstract}

\maketitle

    \tableofcontents
    
    \newpage
    
\section{Introduction}

Motivated in particular by the physical context of the quantum Hall effect~\cite{Jain,Goerbig-09} we study the dynamics of many interacting 2D fermions in a large perpendicular magnetic field. At the many-body level the set-up would be that of the $N$-body Schr\"odinger equation 
\begin{align}\label{eq:many}
i \hbar \partial_t \Psi_N &= H_N \Psi_N \nonumber\\
H_N &= \sum_{j=1} ^N \left\{\left( i \hbar \nabla_j + \frac{b}{2} \bx_j^\perp\right)^2 + N V(\bx_j)\right\} + \sum_{j<k} w(\bx_j -\bx_k)
\end{align}
for $\Psi_N (t) \in L^2_{\rm asym} (\R^{2N})$ an antisymmetric many-body wave-function. The second and third terms of the Hamiltonian were chosen to each formally weigh $O(N^2)$ in the large $N$ limit. The first term, because of the Pauli exclusion principle encoded in the wave-function's antisymmetry and the nature of the spectrum of the magnetic Laplacian, will weigh $\simeq \max (N^2,bN)$ (say for fixed $\hbar$). What we mean by a large magnetic field limit is a scaling where $b\gtrsim N \to \infty$ with fixed $\hbar$, with time possibly rescaled appropriately. This will result in a combined mean-field and semi-classical\footnote{Observe the respective roles of $\hbar$ and $b$ in the kinetic energy operator $\left( i \hbar \nabla + \frac{b}{2} \bx^\perp\right)^2$.} limit, as usual for many-fermions systems. Indeed the Pauli principle will impose the occupancy of a large number of single-particle quantum states, the hallmark of semi-classical regimes. The most crucial feature of the large magnetic field regime is that the appropriate classical phase-space is \emph{not} the position/momentum $(\bx,\bp)$ space. This being perhaps the most novel aspect of the problem, we shall for now bypass the justification of the mean-field approximation to focus on semi-classics.

Hence we start from the mean-field approximation of the above. This means replacing $\Psi_N$ e.g. by a Slater determinant of $N$ orthogonal one-body wave-functions and consider the time evolution of the projector on the subspace thus spanned. More generally, and with the above scaling conventions, this leads to Hartree's equation 
\begin{equation}\label{eq:intro Hartree}
i \hbar \partial_t \gamma = \left[ H_\gamma, \gamma\right] 
\end{equation}
with $0 \leq \gamma (t) \leq \1$ a trace-class operator on $L^2 (\R^2)$, that one should think of as being related to $\Psi_N$ by a partial trace
$$ 
\gamma = \tr_{2\to N} |\Psi_N \rangle \langle \Psi_N |. 
$$
The mean-field Hamiltonian $H_\gamma$ is given as 
\begin{equation}\label{eq:MF hamil}
H_\gamma = \left( i \hbar \nabla + \frac{b}{2} \bx^\perp\right)^2 + N V +  w\star \rho_\gamma
\end{equation}
where 
$$
\rho_\gamma (\bx) = \gamma (\bx,\bx)
$$
is the density of $\gamma$, defined in terms of its' operator kernel. The Pauli principle (antisymmetry of $\Psi_N$) translates into the operator constraint~\cite[Chapter~3]{LieSei-09} 
$$0 \leq \gamma (t) \leq \1$$
while the number of particles is set as 
$$ \tr \gamma = N.$$

The limiting dynamics can be guessed by studying that of a classical particle of charge $-1$ in a transverse magnetic field of amplitude $b$ and a force field $F$. Newton's fundamental equation of dynamics gives
\begin{equation}\label{eq:Newton}
    Z''(t) = F(t, Z(t)) + b  Z'(t)^\perp 
\end{equation}
For a constant and homogeneous force field, the motion is split into a cyclotron orbit and a drift (of the orbit's center) term
\begin{equation}\label{eq:classical decomposition}
    Z(t) = \underbrace{\frac{\abs{Z_c'(0)}}{b}\prth{\matrx{\cos(bt) \\ \sin(bt)}}}_{\eqcolon Z_c(t)} + \underbrace{\dfrac{F^\perp}{b}t}_{\eqcolon Z_d(t)} 
\end{equation}
where we assumed
\begin{align*}
    &Z_d(0) = (0,0) \\
    &Z_c(0) = \frac{\abs{Z_c'(0)}}{b}(1,0)
\end{align*}
The characteristic time for the cyclotron orbit is $b^{-1}$, that for the drift is of order $b$. This suggests, for $b\to \infty$, to observe the motion over a time scale of order $b$, and that the cyclotron motion will be averaged over in the limit equation, its' radius being given by
$$r_c \coloneq \dfrac{\abs{Z_c'(0)}}{b}.$$
Then, if we assume a more general force field $F$, slowly varying on the scale of the cyclotron orbit, we should expect to leading order an effective equation 
$$ 
Z'_d (t) = \frac{F^\perp}{b}
$$
for the motion of the orbit center. Following~\eqref{eq:MF hamil} we should set
$$ F = \nabla \left(NV + w \star \rho\right)$$
with $\rho$ the density of particles. This leads, via the method of characteristics, to a transport-type equation 
\begin{equation}\label{eq:intro drift}
 \partial_t \rho + \nabla^\perp \left( N V + w \star \rho \right) \cdot \nabla \rho = 0.
\end{equation}
Our goal is to rigorously connect~\eqref{eq:intro Hartree} to~\eqref{eq:intro drift} in the limit $b\sim N \to \infty.$ At the classical level this is a gyrokinetic limit.

At the quantum level, the cyclotron radius gets quantized in multiples of $\sqrt{n}$ where $n\in \N$ corresponds to the Landau level index labelling the eigenstates of the magnetic Laplacian. I.e. we write the spectral decomposition of the latter as  
$$\left( i \nabla + \frac{b}{2} \bx^\perp\right)^2 = \sum_{n\in N} 2b \left(n+\frac{1}{2}\right) \Pi_n$$
where $\Pi_n$ is an orthogonal projection (see below for more details). The main difference between the large magnetic field limit we consider here and more common analysis over the $(\bx,\bp)$ semi-classical phase-space is that the gap $2b$ between Landau levels will be of the same order as other energetic contributions. As a consequence our phase-space will be parameterised by $\bz, n\in \R^2 \times \N$, corresponding to the center $\bz$ of the cyclotron orbit's (that we will identify with the complex number $z$) and the Landau level index/quantized cyclotron motion parameter $n$.

Corresponding static problems have been considered at the level of energy ground states in the series of works~\cite{LieSolYng-94,LieSolYng-94b,Yngvason-91} for large atoms (see also~\cite{FouMad-19}) and, more related to our context, for quantum dots in~\cite{LieSolYng-95} (see also~\cite{Perice22}). In particular it is found that for $b\ll N$, the problem reduces to leading order to a $(\bx,\bp)$ semi-classical one, similar to the weak magnetic field situation of~\cite{LieSim-77b,Thirring-81,FouLewSol-15}. The limit energy is the usual Vlasov/Thomas-Fermi functional. By contrast, for $b\gtrsim N$, one finds a magnetic Thomas-Fermi theory set on the $(\bz, n)$ phase-space. This is the case under consideration here, seemingly for the first time at the dynamical level. 

Indeed, dynamical studies of large fermionic systems we are aware of, even for non-zero magnetic fields, all proceed on the $(\bx,\bp)$ phase-space.  The classical counterpart of this work, i.e. the gyrokinetic limit of the Vlasov equation, has been well studied \cite{GS,Brenier-07,Frnod2000LONGTB,Bostan2009TheVS,HanKwan2008TheTF,Ghendrih2009DerivationOA,Miot2016OnTG,PMM,SaintRaymond-00,SaintRaymond-02,BosVu23}. Some results start from Newton's dynamics~\cite{KIM}. In the quantum literature it is known that the Hartree equation can be obtained by a mean field limit from the $N$-body Schr\"{o}dinger dynamics \cite{Benedikter2014MeanFieldEO,Golse2016OnTD,GolMouPau-16,Petrat2016ANM}. It is also known that the Vlasov equation can be derived in a semi-classical limit from the Hartree equation \cite{AmoKhoNou-13,AmoKhoNou-13b,Benedikter2015FromTH,Lafleche2020StrongSL,Lafleche-21,Saffirio-20}. The mean field and semi-classical limits can be coupled to obtain directly the Vlasov equation from the $N$-body Schr\"{o}dinger dynamics \cite{CheLeeLie-21,CheLeeLie-22,CheLeeLie-23}. More recent results have been dealing with singular potentials \cite{Saffirio2017MeanfieldEO,Porta2016MeanFE,ChoLafSaf-21,ChoLafSaf-23}. 

Closer to our setting, we mention the recent~\cite{Porat} where Euler's equation in vorticity form is obtained from the $N$-body Schr\"{o}dinger dynamics with large magnetic field and repulsive 2D Coulomb interaction $w= -\log |\,.\,|$. This corresponds to the drift equation~\eqref{eq:intro drift} in this context. The crucial difference between~\cite{Porat} and the present contribution is that the former is set in a regime where the gap between Landau levels is small compared to the interactions. The classical phase-space is consequently again the position/momentum one. However,~\cite{Porat} can deal with a much more singular interaction potential, leveraging its' coercivity.  

As regards the approach to the semi-classical limit, in the $(\bx,\bp)$ phase-space, the use of the Wigner function is often the privileged angle of attack. We do not see that such a tool is available for the $(\bz, n)$ phase-space we have to consider here. Hence we rely on appropriate magnetic coherent states $\psi_{z,n}\in L^2 (\R^2)$ with Landau level index $n$, approximately localized around a guiding center position~$\bz$.  We associate to $\gamma$ a Husimi function 
$$ m_\gamma (z, n) \propto \left\langle \psi_{z,n}, \gamma \,  \psi_{z,n}\right\rangle$$
and study the dynamics thereof. A related approach, based on $(\bx,\bp)$ coherent states, has been implemented previously in~\cite{CheLeeLie-21,CheLeeLie-22,CheLeeLie-23}. As it now stands, our method is rather demanding in terms of regularity of the potentials $V$ and $w$. This can be improved if the boundedness of some moments of the magnetic kinetic energy are propagated in time (see Remark~\ref{rem:choices} below). We cannot prove this at present: we are in effect dealing with long-time asymptotics, for which it is difficult to keep moments under control, even at the level of the classical Vlasov equation.


%
%
%
    \bigskip
    
    \noindent \textbf{Acknowledgments.} Work financially supported by the European Union's Horizon 2020 Research and Innovation Programme (Grant agreement CORFRONMAT No. 758620). The present contribution is an expanded version of a chapter of the first author's PhD thesis~\cite{Perice-these}. We are grateful to the members of the thesis committee (Thierry Champel, Mathieu Lewin, S\"{o}ren Petrat, Nicolas Raymond, Laure Saint-Raymond, Chiara Saffirio) for their careful reading of, and helpful suggestions on, the first version of this text. We also thank Laurent Lafl\`eche for interesting discussions  and an anonymous referee for spotting mistakes in a previous version of the text.
    
\section{Main results}

\subsection{Model and scaling}

    It will be convenient to work in a slightly different scaling than sketched in the introduction. We set this up first.  

    \begin{notation}[\textbf{Model}]\label{model dynamic}\mbox{}\\
        We work on $\mathbb{R}^2$. The one body kinetic energy operator is the magnetic Laplacian
        \begin{equation*}
            \mathscr{L}_b \coloneq \left(i\hbar \nabla + b \bA\right)^2
        \end{equation*}
        With
        \begin{equation*}
            \text{Dom}\prth{\mathscr{L}_b} \coloneq \left\{\psi \in L^2\prth{\mathbb{R}^2} | \mathscr{L}_b \psi \in L^2\prth{\mathbb{R}^2}\right\}
        \end{equation*}
        We work in symmetric gauge, namely the vector potential is
        \begin{equation}\label{eq:A symmetric}
            \bA = \dfrac{1}{2} \bx^\perp. 
        \end{equation}
        We denote by $b$ the magnetic field amplitude, associated to the magnetic length 
        \begin{equation*}
            l_b \coloneq\sqrt{\dfrac{\hbar}{b}}
        \end{equation*}
        Let $V$ be the external potential and $w$ the interaction potential, assumed to be radial. 
        
        We study a solution $\gamma \in L^\infty \prth{\mathbb{R}_+, \mathcal{L}^1\prth{L^2\prth{\mathbb{R}^2}}}$ to the Hartree equation 
        \begin{equation}
        \boxed{i\hbar \partial_t \gamma = \sbra{\mathscr{L}_b + V + w\star \rho_\gamma, \gamma}} \label{eq:HF}
    \end{equation}
        where $\mathcal{L}^p$ is the $p$-th Schatten class, 
        \begin{equation}\label{eq:space time density}
            \rho_\gamma(t, \bx) \coloneq \gamma(t)(\bx, \bx) 
        \end{equation}
        the density associated to $\gamma$ that we identify with its integral kernel. We will denote
        \begin{equation}\label{eq:Hb}
            H_b(t) \coloneq \mathscr{L}_b + V + \dfrac{1}{2} w \star \rho_{\gamma_b(t)} .
        \end{equation}
        \hfill$\diamond$
    \end{notation}

    Our goal is to obtain from the Hartree equation~\eqref{eq:HF} the following drift equation for a density $\rho:\mathbb{R}_+\times \mathbb{R}^2 \to \mathbb{R}_+$,
    \begin{equation}
        \partial_t \rho + \nabla^\perp (V + w\star \rho)) \cdot \nabla \rho = 0. \label{eq:drift}
    \end{equation}
  We will denote
        \begin{equation*}
            \mathrm{DRIFT}_\rho(\mu)(t, \bz) \coloneq \partial_t \mu(t, \bz) + \nabla^\perp \prth{V + w\star \rho(t)}(\bz)\cdot \nabla \mu(t, \bz)
        \end{equation*}
        so that our target equation~\eqref{eq:drift} takes the form 
        $$\mathrm{DRIFT}_\rho(\rho) = 0.$$
        
%
    

Our plan is to examine a truly large magnetic field regime where all the terms in the Hamiltonian~\eqref{eq:Hb} are of order $1$. As recalled in Section~\ref{sec:preli} below, the order of magnitude of the kinetic energy is $\hbar b$, which we will henceforth fix to unity. As discussed in the introduction, the time-scale we work on is of order $b$. Since we consider fermionic particles, constraints on the density matrix are imposed to enforce the Pauli exclusion principle. We summarize these conventions below:

    \begin{notation}[\textbf{Scaling}]\label{scaling dynamic}\mbox{}\\
        We work in a large magnetic field/semi-classical limit
        \begin{align*}
            b\limit+\infty, \quad \hbar \limit\displaylimits_{b\to\infty} 0
        \end{align*}
        such that the magnetic kinetic energy is of order $1$:
        \begin{equation}\label{eq:energy scaling}
            \hbar b \limit\displaylimits_{b \to\infty} 1. 
        \end{equation}
        Let $\gamma \in L^\infty\prth{\mathbb{R}_+, \mathcal{L}^1\prth{L^2(\mathbb{R}^2}}$, such that
        \begin{equation}\label{eq:gamma fermionic}
            \Tr{\gamma(0)} = 1 \quad \text{ and } \quad 0 \le \gamma(0) \le 2\pi l_b^2 = 2\pi \frac{\hbar}{b}
        \end{equation}
        and define the time rescaled density matrix
        \begin{equation}\label{eq:time gamme re}
            \forall t\in \mathbb{R}_+, \quad \gamma_b(t) \coloneq \gamma(bt) 
        \end{equation}
        \hfill $\diamond$
    \end{notation}

    If $\gamma$ satisfies~\eqref{eq:HF}, the equation for the time-rescaled density matrix is
    \begin{equation}\label{eq:retime HF}
        \partial_t \gamma_b 
            = \dfrac{b}{i \hbar} \sbra{\mathscr{L}_b + V + w\star \rho_{\gamma_b}, \gamma_b}
            = \dfrac{1}{i l_b^2} \sbra{\mathscr{L}_b + V + w\star \rho_{\gamma_b}, \gamma_b} 
    \end{equation}
The Pauli principle $\gamma_b \le 2\pi l_b^2$  (which propagates in time, see Lemma~\ref{lem:fermions conserv} below) guarantees that the system occupies a volume of order $1$ in the limit
    \begin{align*}
        l_b \limit\displaylimits_{b\to\infty} 0
    \end{align*}
    Indeed it is known, \cite[Chapter~3]{Jain} or \cite[Subsection I.4]{Perice22}, that the degeneracy per area inside a Landau level is of order $l_b^{-2}$. A typical fermionic state satisfying \eqref{eq:gamma fermionic} is a projection onto a $N$-body Slater determinant of $N$ orthonormal one body wave-functions with
    \begin{align*}
        N \coloneq \mathcal{O}\prth{\dfrac{1}{2\pi l_b^2}}
    \end{align*}
    Such a $N$-particles state occupies a volume of order
    \begin{align*}
        \dfrac{N}{l_b^{-2}} = \mathcal{O}(1)
    \end{align*}
    Hence with \eqref{eq:energy scaling} this confirms that all the terms in the Hamiltonian $\mathscr{L}_b + V + w\star \rho_\gamma$ are of order $1$. 
%
    As a remark, we give an equivalent formulation of this scaling, making connections to the introductory section. If one takes exactly $\hbar = 1/b$, then \eqref{eq:retime HF} is equivalent to
    \begin{align*}
        i \partial_t \gamma
            = \sbra{\prth{i\nabla + b^2 \bA} + b^2(V + w\star \rho_{\gamma_b}), \gamma_b}
    \end{align*}
    In other words with the new scaling
    \begin{align*}
        & \widetilde{b} \coloneq b^2 \\
        &\widetilde{\gamma_b} \coloneq \dfrac{b^2}{2\pi} \gamma 
    \end{align*}
    we have
    \begin{align*}
        &\Tr{\widetilde{\gamma_b}} = \dfrac{\widetilde{b}}{2\pi}, \quad \widetilde{\gamma_b} \le 1 \\
        &i \partial_t \widetilde{\gamma_b} = \sbra{\prth{i\nabla + \widetilde{b} \bA}^2 + \widetilde{b}V + w\star \rho_{\widetilde{\gamma_b}}, \widetilde{\gamma_b}}
    \end{align*}
    where all the terms in the Hamiltonian 
    $$\prth{i\nabla + \widetilde{b} \bA}^2 + \widetilde{b}V + w\star \rho_{\widetilde{\gamma_b}}$$
    are of order $\widetilde{b}$. The above is the equivalent of the scaling ``particle number proportional to magnetic field'', $\widetilde{b} \propto N$ originally studied in~\cite{LieSolYng-95} at the level of ground states, where a magnetic Thomas-Fermi theory emerges as the relevant effective description.

\subsection{Results}

We may now state our main results. We first prove that the density of the Hartree solution almost satisfies the weak form of the drift equation:

    \begin{theorem}[\textbf{Dynamics of the Hartree solution}]\label{th:V of HF}\mbox{}\\
        Let $V, w \in  W^{4, \infty}\prth{\mathbb{R}^2}$ and $\gamma_b \in L^\infty \prth{\mathbb{R}_+, \mathcal{L}^1\prth{L^2\prth{\mathbb{R}^2}}}$ be a solution of \eqref{eq:retime HF} with initial datum satisfying
        \begin{equation*}
            \Tr{\gamma_b(0)} = 1, \quad 0\le \gamma_b(0) \le 2\pi l_b^2, \quad \Tr{\gamma_b(0) H_b(0)} < C 
        \end{equation*}
        with $C$ independent of $b$. Let the associated drift operator be as in~\eqref{eq:drift}. Then, $\forall \varphi \in C^\infty \prth{\mathbb{R}_+ \times \mathbb{R}^2}$ with compact support
        \begin{align*}
             \left|\int_{\mathbb{R}_+ \times \mathbb{R}^2}\rho_{\gamma_b}(t, \bz)\mathrm{DRIFT}_{\rho_{\gamma_b}}(\varphi)(t, \bz) \dd t \dd \bz
             - \int_{\mathbb{R}^2} \varphi(0, \bz) \rho_{\gamma_b}(0, \bz) \dd \bz \right|
                \le C(\varphi, V, w)l_b^{\frac{2}{7}}
        \end{align*}
        for some fixed constant $C(\varphi, V, w)$.
    \end{theorem}
    
    The proof of the above consists of two main parts:
    
    \smallskip
    
     \noindent$\bullet$ We define a semi-classical measure on phase-space approximating the exact quantum density. Using the vortex coherent state $\psi_{z,n}$ (approximately) localized around the position $z\in \C \leftrightarrow \bz\in \R^2$ and (exactly) localized in the $n$-th Landau level of the magnetic kinetic energy (see Section~\ref{sec:coherent} below for more details) we form the Husimi function 
     \[
   (z,n) \mapsto \frac{1}{2\pi l_b^2} \left\langle \psi_{z,n}, \gamma_b \psi_{z,n} \right\rangle \eqcolon m_{\gamma_b} (z,n).  
     \]
     Selecting a suitably large cut-off $M\gg 1$ for the Landau level index we mimic the true density by summing the above for $n\leq M$ 
     \begin{equation}\label{eq:semi density 1}
      \rho_{\gamma_b}^{sc, \le M} \coloneq \sum_{n\leq M} m_{\gamma_b} (z,n) \simeq \rho_{\gamma_b}
     \end{equation}
under suitable, mild assumptions. 

\smallskip
    
     \noindent$\bullet$ Combining the Hartree equation~\eqref{eq:HF} with the algebraic properties of vortex coherent states we find that  $\rho_{\gamma_b}^{sc, \le M}$ approximately solves (for suitably large $b$ and tuned $M\gg 1$) the drift equation~\eqref{eq:drift}. The control of the implied error terms is the core analytical part of the proof.
    
\smallskip

    Next, using appropriate stability estimates for solutions of the limiting drift equation, we can lift the above theorem to an estimate between $\rho_{\gamma_b}$ and the classical solution.  We denote $W_1$ the Monge-Kantorovitch-Wasserstein (MKW) metric
    \begin{equation}\label{eq:MKW}
        W_1(\mu, \nu) 
                \coloneq \inf_{\pi \in \Gamma(\mu, \nu)} \iint_{\mathbb{R}^2 \times \mathbb{R}^2} \abs{x-y} \dd \pi (\bx,\by) 
                = \sup_{\norm{\nabla \varphi}_{L^\infty\prth{\R^2}}\le 1}\abs{\int_{\R^2}\varphi \, \dd \prth{\mu - \nu}}
    \end{equation}
    where $\Gamma(\mu, \nu)$ is the set of couplings between $\mu, \nu \in \mathcal{P}\prth{\R^2}$, namely $\mathcal{P}\prth{\R^4 }\ni \pi \in \Gamma(\mu, \nu)$ if 
    \begin{equation}\label{eq:couplings}
    \int_{\R^2} \pi (\bx,\by) \dd \by = \mu (\bx), \quad \int_{\R^2} \pi (\bx,\by) \dd \bx = \nu (\by).
    \end{equation}
We then have 
    
    \begin{theorem}[\textbf{Convergence of densities}] \label{th:conv}\mbox{}\\
        We make the same assumptions as in Theorem~\ref{th:V of HF}, with in addition 
        \begin{align*}
            \nabla w \in L^1\prth{\R^2}, \quad w\in H^2\prth{\R^2}, \quad V, w \in W^{9, \infty}\prth{\R^2}, \quad d^9V, d^9w \in L^1\prth{\R^2}
        \end{align*}
        and, with the position operator denoted by $\bX$, 
        \begin{equation}\label{eq:trapped}
        \Tr{\gamma_b(0) \abs{X}^p} < C 
            \end{equation}
            independently of $b$, for some $p>7$.
        
        Let  $\rho \in L^\infty\prth{\R_+, L^1\prth{\R^2}}$ solve the drift equation~\eqref{eq:drift}. Then $\forall t\in\R_+$ and every test function $\varphi$ over $\R^2$
        \begin{equation}\label{eq:CV dens}
            \abs{\int_{\R^2}\varphi\prth{\rho_{\gamma_b}(t) - \rho(t)}}
                \le \widetilde{C}(p, t, V, w)\prth{\norm{\varphi}_{W^{1, \infty}} + \norm{\nabla\varphi}_{L^2}}\prth{W_1\prth{\rho_{\gamma_b}(0), \rho(0)} + l_b^{\min\prth{2\frac{p-7}{4p-7}, \frac{2}{7}}}}
        \end{equation}
        with
        \begin{align*}
            &\widetilde{C}(p, t, V, w) \coloneq \abs{\Tr{\gamma_b(0) H_b(0)}} + \norm{V}_{L^\infty} + \norm{w}_{L^\infty} \\&+ e^{2\prth{\norm{w}_{W^{2,\infty}} +  \norm{V}_{W^{2,\infty}}}t} \bigg( 1  + C(p) \prth{1 + \Tr{\gamma_b(0) \abs{X}^p}} \prth{\abs{\Tr{\gamma_b(0) H_b(0)}} + \norm{V}_{L^\infty} + \norm{w}_{L^\infty}} \\&+  Ct^2\prth{\norm{w}_{H^2}+\norm{V}_{W^{9,\infty}} + \norm{d^9V}_{L^1} + \norm{w}_{W^{9,\infty}} + \norm{d^9w}_{L^1}} \\
                &\prth{\norm{\nabla w}_{L^1} + \norm{w}_{W^{2,\infty}}e^{\prth{\norm{V}_{W^{2,\infty}} +  \norm{w}_{W^{2,\infty}}} t}}
            \prth{\abs{\Tr{\gamma_b(0) H_b(0)}} + \norm{V}_{L^\infty} + \norm{w}_{L^\infty}} \bigg)
        \end{align*}
        
    \end{theorem}
    
    \medskip
    
Theorem~\ref{th:conv} follows from a stability estimate \`a la Dobrushin~\cite{Dobrushin-79} for classical Hamiltonian equations. We treat the error between the quantum and classical equations obtained from Theorem~\ref{th:V of HF} as a source term for the limiting equation, and show that the stability theory of the latter (as reviewed e.g. in~\cite[Section~1.4]{Golse-13}) survives this addition. When considering the limit of the semi-classical density~\eqref{eq:semi density 1}, an estimate is obtained directly in Wasserstein-1 metric, see Proposition~\ref{prop:sc theorem} below. The additional initial trapping assumption~\eqref{eq:trapped}, and the slight modification of the norm in which the convergence is measured~\eqref{eq:CV dens}, arise when vindicating the approximation in~\eqref{eq:semi density 1}. As regards the additional assumption~\eqref{eq:trapped}, we note that it is fairly natural for typical initial data. For example, equilibria of systems with an additional trapping external potential included in the Hamiltonian usually decay exponentially in space, in which case one may think that $p= \infty$ formally.


\subsection{Organisation of the paper}

    Section~\ref{sec:preli} covers preliminaries, focusing on the magnetic Laplacian and the conserved properties of the dynamics. In Section~\ref{sec:densities}, we introduce Husimi functions, the associated semi-classical densities and prove that they approximate the physical density. Then we study the dynamics of these semi-classical densities in Section~\ref{sec:sc EDP} and prove they approximately follow the drift equation. The conclusion of the proof of Theorem~\ref{th:V of HF} is given in Section~\ref{ssec:proof} by combining this with the results of Section~\ref{sec:densities}. We study perturbed classical flows associated with~\eqref{eq:drift} in Section~\ref{sec:stab class}. This leads to the Dobrushin-like stability estimate allowing to conclude the proof of Theorem~\ref{th:conv} in Section~\ref{sec:dobrushin}.
    
\section{Landau quantization and the magnetic Hartree equation} \label{sec:preli}

    We here recall the  usual formalism for describing the magnetic Laplacian in terms of annihilation and creation operators. Further details about these operators and the properties of Landau levels are reviewed e.g. in~\cite{RouYng-19} or~\cite{Perice22} and references therein. This formalism provides a basis~\eqref{eq:magnetic Laplacian R2} of eigenstates indexed by two quantum numbers $n$ and $m$, with $n$ denoting the index describing the Landau level and $m$ representing ``angular momentum minus Landau level index''. To obtain a projector on a point in phase space, we use coherent states for a fixed $n$. In two dimensions, the complex parameter in the definition of coherent states can be identified with a position. Consequently, following e.g.~\cite{ChaFlo-07,ChaFloCan-08,ChaFlo-09}, we construct a one-particle state localized at a specific point in space, see Definition~ \ref{def:coherentstates}. Then, we provide some properties of the associated projector. We conclude this section with a brief recap of the conserved quantities of the Hartree equation~\eqref{eq:HF}.

\subsection{Landau quantization} \label{sec:LL quantization}

    \begin{notation}[\textbf{Magnetic momentum and kinetic energy}]\label{not:mag mom}\mbox{}\\
    We denote by $p_1, p_2$ the coordinates of the magnetic momentum
        \begin{align*}
            \mathscr{P}_{\hbar, b} 
                \eqcolon \prth{\matrx{p_1 \\ p_2 }}
                \eqcolon - \prth{\matrx{i\hbar\partial_1 + bA_1 \\ i \hbar \partial_2 + bA_2 }}
        \end{align*}
        with $\bA = \frac{\bx^\perp}{2}$. 
        Define then the annihilation, creation and number operators respectively as
        \begin{align*}
            a \coloneq \dfrac{p_1 + i p_2}{\sqrt{2 \hbar b}}, \quad a^\dagger \coloneq \dfrac{p_1 - i p_2}{\sqrt{2 \hbar b}}, \quad \mathcal{N} \coloneq a^\dagger a
        \end{align*}
        \hfill$\diamond$
    \end{notation}
    
    We have the commutation relations:
    \begin{align*}
        &\sbra{p_1, p_2} = i\hbar b \\
        &\sbra{a, a^\dagger} = \mathbb{1} \text{ (canonical commutation relation)}
    \end{align*}
    and may express the magnetic Laplacian as 
    \begin{align*}
        \mathscr{L}_{b} = 2\hbar b\prth{\mathcal{N} +  \dfrac{\text{Id}}{2}}
    \end{align*}

    \begin{notation}[\textbf{Landau levels}]\label{not:Landau levels}\mbox{}\\
        We define the $n^{th}$ Landau level as the eigenspace associated to $n \in \mathbb{N}$: 
        \begin{align*}
            \text{nLL} \coloneq \sett{\psi \in \text{Dom}\prth{\mathscr{L}_{\hbar, b}} \text{ such that } \mathcal{N} \psi =  n\psi}
        \end{align*}
        The ground level, denoted $\text{LLL}$ for \textit{Lowest Landau Level} has energy $E_0 = \hbar b$.\hfill$\diamond$
    \end{notation}
        
    The Landau levels are isomorphic, and the operator $a^\dagger/\sqrt{n+1}$ is a unitary mapping from $\text{nLL}$ to $\text{(n+1)LL}$ of inverse $a/\sqrt{n+1}$. Therefore we may, using $a^\dagger$, extend a basis of LLL to higher Landau levels. The Lowest Landau level consists of holomorphic functions pondered by a Gaussian factor, see e.g.~\cite{RouYng-19}.

    The Landau level quantization of the kinetic energy corresponds to the quantization of the cyclotron orbit. To complete this aspect, we associate an operator to the motion of the guiding center of the orbit.
    
    \begin{notation}[\textbf{Guiding center oscillator}]\label{not:guiding}\mbox{}\\
        For the rest of the text, we will identify a vector
        \begin{align*}
            \bx \coloneq \prth{\matrx{x_1 \\ x_2}} \in \mathbb{R}^2
        \end{align*}
        with the corresponding complex number 
        $$x \coloneq x_1 + i x_2.$$
        It will be convenient to use an upper case $\bX$ for the position operator (multiplication by $\bx$) in $\mathbb{R}^2$ and $X$ the position operator in $\mathbb{C}$ (multiplication by $x$).
        
        We introduce the following ``cyclotron motion'' and ``orbit center'' operators~\cite{RouYng-19}
        \begin{align*}
            &\br \coloneq \prth{\matrx{r_1 \\ r_2}} \coloneq \dfrac{\mathscr{P}_{\hbar, b}^\perp}{b}  = \dfrac{1}{b}\prth{\matrx{-p_2 \\ p_1}} \\
            &\bR \coloneq \bX - \br
        \end{align*}
        and associate to $\bR$ the creation and annihilation operators
        \begin{align*}
            c = \frac{R_1 - i R_2}{\sqrt{2}l_b} \quad c^\dagger = \frac{R_1 + i R_2}{\sqrt{2}l_b}.
        \end{align*}
        \hfill$\diamond$
    \end{notation}
    
    The operator $\br$ represents the position of a particle in the center of orbit frame. The classical physics meaning of this definition is that, during cyclotron motion, the momentum is perpendicular to the position relative to the center to the orbit. Electrons are describing clockwise orbits, thus the momentum rotated of $\pi$/2 gives us $\br $. Moreover, $\br $ is related to the quantization of the cyclotron pulsation of the orbit because
    \[
        a = \frac{p_1 + i p_2}{\sqrt{2\hbar b}} = \frac{r_2 - ir_1}{\sqrt{2} l_b} = \dfrac{-i r}{\sqrt{2}l_b}, \quad a^\dagger = \dfrac{i\overline{r}}{\sqrt{2}l_b}.
    \]
    From the definition of $\br $, the position $\bR $ of the orbit center is indeed
    \begin{align*}
        \textbf{X} = \textbf{R} + \textbf{r}
    \end{align*}
    and related to the second harmonic oscillator
    \begin{align}
        c = \dfrac{\overline{R}}{\sqrt{2}l_b}, \quad c^\dagger = \dfrac{R}{\sqrt{2}l_b}. \label{eq: R and c}
    \end{align}
    The components of $\br $, $\bR $ and $\bX $ commute with one another. Moreover
    \begin{align}\label{eq:commutators}
            &\sbra{r_1, r_2} = i l_b^2 \nonumber\\
            &\sbra{R_1, R_2} = -il_b^2\nonumber\\
            &\sbra{c, c^\dagger} = \text{Id} \nonumber\\
            &\sbra{a, c} = \sbra{a, c^\dagger} = \sbra{a^\dagger, c} = \sbra{a^\dagger, c^\dagger} = 0
    \end{align}
    We therefore have two independent harmonic oscillators. By successively applying the creation operators $a^\dagger$ and $c^\dagger$ we obtain the desired eigenbasis of the magnetic Laplacian.
    
    It is useful to rewrite the creation and annihilation operators using complex notation:
    \begin{align}
        &a = -i \prth{\frac{X}{2\sqrt{2}l_b} + \sqrt{2}l_b \partial_{\overline{X}}},\quad
        a^\dagger = i \prth{\frac{\overline{X}}{2\sqrt{2}l_b} - \sqrt{2}l_b \partial_{X}} \nonumber\\
        &c = \frac{\overline{X}}{2\sqrt{2}l_b} + \sqrt{2}l_b \partial_{X},\quad 
        c^\dagger = \frac{X}{2\sqrt{2}l_b} - \sqrt{2}l_b \partial_{\overline{X}} \label{eq:complex ops}
    \end{align}
    In symmetric gauge~\eqref{eq:A symmetric}, the family defined by
    \begin{equation}\label{eq:magnetic Laplacian R2}
        \varphi_{n, m} \coloneq \frac{\prth{a^\dagger}^n \prth{c^\dagger}^m}{\sqrt{n!m!}} \varphi_{0,0} 
    \end{equation}
    with
    \[
        \varphi_{0,0}(\bx) = \dfrac{1}{\sqrt{2\pi l_b^2}} e^{\frac{-\abs{x}^2}{4l_b^2}}
    \]
    is an orthonormal Hilbert basis of $L^2\prth{\mathbb{R}^2}$. The full expression, obtained from \cref{eq:complex ops} or see \cite{ChaFlo-07,ChaFloCan-08,ChaFlo-09,RouYng-19}, is
    \begin{equation}\label{eq:LL basis}
        \varphi_{n, m}(\bx) = \dfrac{\prth{i\frac{\overline{X}}{2} - 2il_b^2 \partial_{X}}^n X^m}{\sqrt{2\pi l_b^2 n!m!}\prth{\sqrt{2}l_b}^{n+m}} e^{\frac{-\abs{x}^2}{4l_b^2}} 
    \end{equation}
    The orthogonal projector on the $n$-th Landau level from Notation~\ref{not:Landau levels} is recovered as 
    \begin{align*}
        \Pi_n \coloneq \sum_{m \in \mathbb{N}} \ket{\varphi_{n, m}} \bra{\varphi_{n, m}}.
    \end{align*}

\subsection{Coherent states}\label{sec:coherent}

    We next define coherent states in order to have wave functions localized at a precise point in the phase-space ``position of the orbit center $\times$ Landau level index  ''. 
    
    \begin{definition}[\textbf{Vortex coherent states}]\label{def:coherentstates}\mbox{}\\
        Let $z\in \C \leftrightarrow \R^2, n\in \N$. We define the associated coherent state
        \begin{align}
            \psi_{z, n} 
                \coloneq e^{\frac{\overline{ z } c^\dagger -  z  c}{\sqrt{2}l_b}} \varphi_{n, 0}
                = e^{-\frac{\abs{z}^2}{4 l_b^2} + \frac{\overline{ z } c^\dagger}{\sqrt{2}l_b}} \varphi_{n, 0}
                = e^{-\frac{\abs{z}^2}{4 l_b^2}} \sum_{m\in\mathbb{N}} \frac{1}{\sqrt{m!}}\prth{\frac{\overline{z}}{\sqrt{2}l_b}}^m \varphi_{n, m} \label{eq:cor states def}
        \end{align}
        and the associated projectors
        \begin{align}\label{eq:Pizn}
            \Pi_{z, n} &\coloneq \ket{\psi_{z, n}} \bra{\psi_{z, n}}\nonumber\\
             \Pi_z &= \sum_{n\in\mathbb{N}} \Pi_{z, n}.
        \end{align}
        Finally we define for $M\in\N$ or $N_1 \le N_2$ the truncated projectors
        \begin{align*}
            \Pi_{N_1:N_2} \coloneq \sum_{n = N_1}^{N_2} \Pi_n, \quad
            \Pi_{\le M} \coloneq \Pi_{0:M}, \quad
            \Pi_{z,\le M} \coloneq \sum_{n = 0}^{N}\Pi_{z, n}
        \end{align*}
        with similar definitions for $\Pi_{>M}$ and $\Pi_{z, > M}$. By convention, if $N_1 < 0$ we set
        $$\Pi_{N_1:N_2} \coloneq \Pi_{\le N_2}.$$
        \hfill$\diamond$
    \end{definition}

    By construction we have that
    \begin{equation}\label{eq:ladder coh}
    a^\dagger \psi_{z,n} = \sqrt{n+1} \psi_{z,n+1}, \quad a \psi_{z,n} = \sqrt{n} \psi_{z,n-1}.
    \end{equation}
    Also, 
    \begin{align}
        c \psi_{z, n} = \frac{\overline{ z }}{\sqrt{2}l_b}\psi_{z, n} \label{eq:coherent eigen}
    \end{align}
    and thus, recalling \cref{eq: R and c} we note that $\psi_{z, n}$ is localised around $z$ in the sense that 
    \begin{align*}
        \overline{R} \psi_{z, n} = \overline{ z }\psi_{z, n}
    \end{align*}
    with fluctuations of order $l_b$ (as can be seen from the next lemma). The following explicit expressions will be important for us:

    \begin{lemma}[\textbf{Expression of vortex coherent states}]\label{lem:coherent}\mbox{}\\
        We have
        \begin{align}
            &\psi_{z, n}(\bx) =  \dfrac{i^n}{\sqrt{2\pi n!}l_b} \prth{\dfrac{\overline{ x - z  }}{\sqrt{2}l_b}}^n e^{-\frac{\abs{x-z}^2 - 2i \bz ^\perp \cdot \bx}{4l_b^2}} \label{eq:psi nR} \\
            &\Pi_{z, n}(\bx,\by) = \dfrac{1}{2\pi n! l_b^2} \prth{\dfrac{\prth{\overline{ x - z  }}\prth{ y - z }}{2l_b^2}}^n e^{-\frac{\abs{x-z}^2 + \abs{y-z}^2 - 2i\bz ^\perp\cdot \left(\bx - \by\right)}{4l_b^2}} \nonumber \\
            &\Pi_z(\bx,\by) = \dfrac{1}{2\pi l_b^2} e^{-\frac{\abs{x-y}^2 - 2i\prth{\bx^\perp \cdot \by + 2\bz ^\perp\cdot\left(\bx - \by\right)} }{4l_b^2}}. \label{eq:Pi_z kernel}
        \end{align}
        Consequently 
        \begin{equation}\label{eq:gradperp Pi 1}
        \nabla_\bz ^\perp \Pi_z(\bx,\by) = \dfrac{\by - \bx}{i l_b^2} \Pi_z(\bx,\by)
    \end{equation}
    or, as an operator identity,
    \begin{equation}
        \nabla_\bz ^\perp \Pi_z = \dfrac{1}{i l_b^2} \sbra{\Pi_z, \bX} \label{eq:gradperp Pi}
    \end{equation}
    \end{lemma}

    We refer to~\cite{ChaFlo-07} again for the derivation of the above exact expressions, from which~\eqref{eq:gradperp Pi 1} immediately follows. We will rely heavily on the following closure relations~\cite{ChaFlo-07,RouYng-19}:
    
    \begin{lemma}[\textbf{Coherent states partition of unity}]\label{lem:resol}\mbox{}\\
        With Definition~ \ref{def:coherentstates},
        \begin{align}
             \dfrac{1}{2\pi l_b^2} \int_{\mathbb{R}^2} \Pi_{z, n} \dd \bz  &= \Pi_n  \label{eq:res ID dz}\\
            \dfrac{1}{2\pi l_b^2}\sum_{n\in\mathbb{N}} \int_{\mathbb{R}^2} \Pi_{z, n} \dd \bz  &= \1_{L^2 (\R^2)} \label{eq:res ID}
        \end{align}
    \end{lemma}
    \begin{proof}
        We prove \cref{eq:res ID dz} with a direct computation 
        \begin{align*}
            &\dfrac{1}{2\pi l_b^2} \int_{\mathbb{R}^2} \ket{\psi_{z, n}} \bra{\psi_{z, n}} \dd \bz  \\
                =&  \sum_{m_1, m_2 \in \N} \dfrac{1}{2\pi l_b^2\sqrt{m_1! m_2!}}\int_{\R^2} \prth{\dfrac{\overline{ z }}{\sqrt{2}l_b}}^{m_1} \prth{\dfrac{ z }{\sqrt{2}l_b}}^{m_2}e^{-\frac{\abs{z}^2}{2l_b^2}} \dd \bz  \ket{\varphi_{n, m_1}} \bra{\varphi_{n, m_2}} \\
                =&  \sum_{m\in\mathbb{N}}\dfrac{1}{2\pi l_b^2 m!}\int_{\R^2} \prth{\dfrac{\abs{z}^2}{2 l_b^2}}^m e^{-\frac{\abs{z}^2}{2l_b^2}} \dd \bz  \ket{\varphi_{n, m}} \bra{\varphi_{n, m}}
                = \sum_{m\in\mathbb{N}}\ket{\varphi_{n, m}} \bra{\varphi_{n, m}} 
                = \Pi_n.
        \end{align*}
        In the second line we have used polar coordinates to observes that the integrals over $z$ vanish when $m_1 \neq m_2$. Summing over $n$, \cref{eq:res ID} follows.
    \end{proof}
    
    Formula~\eqref{eq:gradperp Pi} will play a key role in the computation of the spacial derivative of the density in Section~\ref{sec:sc EDP} below. We will also rely heavily on an approximation thereof applying to the truncated projector.
     
    \begin{lemma}[\textbf{Spatial derivatives of coherent state projectors}]\label{lem:derivative}\mbox{}\\
    For the localized projector~\eqref{eq:Pizn} we have 
    \begin{multline}\label{eq:d1 pi le N Z bis}
     \nabla^\perp_{\bz}  \Pi_{z, n} = \dfrac{1}{il_b^2} \sbra{\Pi_{z, n}, \bX}
                     - \dfrac{\sqrt{n+1}}{\sqrt{2}l_b}\prth{\matrx{1 & 1 \\ -i & i}} \prth{\matrx{ \ket{\psi_{z, n}} \bra{\psi_{z, n+1}}\\ \ket{\psi_{z, n+1}} \bra{\psi_{z, n}}}} \\
                     + \dfrac{\sqrt{n}}{\sqrt{2}l_b}\prth{\matrx{1 & 1 \\ -i & i}} \prth{\matrx{ \ket{\psi_{z, n-1}} \bra{\psi_{z, n}}\\ \ket{\psi_{z, n}} \bra{\psi_{z, n-1}}}}
    \end{multline}
and, summing over $n\leq M$,
        \begin{equation}\label{eq:d1 pi le N z}
            \nabla^\perp_{\bz}  \Pi_{z, \le M}
                = \dfrac{1}{il_b^2} \sbra{\Pi_{z, \le M}, \bX}
                    - \dfrac{\sqrt{M+1}}{\sqrt{2}l_b}\prth{\matrx{1 & 1 \\ -i & i}} \prth{\matrx{
                        \ket{\psi_{z, M}} \bra{\psi_{z, M+1}}\\
                        \ket{\psi_{z, M+1}} \bra{\psi_{z, M}}
                     }}.  
                     \end{equation}
    \end{lemma}
    
    \begin{proof}    
        Starting from~\eqref{eq:psi nR} we have
        \begin{equation}
            \psi_{z, n}(\bx) = \dfrac{-i}{\sqrt{n}} \dfrac{\overline{ z - x }}{\sqrt{2} l_b} \psi_{z, n-1}(\bx), \quad \overline{\psi_{z, n}(\by)} = \dfrac{i}{\sqrt{n}}\dfrac{ z - y }{\sqrt{2}l_b} \overline{\psi_{z, n-1}(\by)} \label{eq:psi n to n-1}.
        \end{equation}
        Hence 
        \begin{align*}
            \Pi_{z, n}(\bx,\by) 
                =& \psi_{z, n}(\bx) \overline{\psi_{z, n}(\by)}
                = \dfrac{1}{2\pi n! l_b^2} \prth{\dfrac{\prth{\overline{ x - z  }}\prth{ y - z }}{2l_b^2}}^n e^{-\frac{\abs{x-z}^2 + \abs{y - z}^2 - 2i\bz ^\perp\cdot\prth{ \bx - \by }}{4l_b^2}}\\
                =& \dfrac{1}{2\pi n! l_b^2} \prth{\dfrac{\prth{\overline{ x - z  }}\prth{ y - z }}{2l_b^2}}^n e^{-\frac{2\abs{z}^2 + \abs{x}^2 + \abs{y}^2 - 2(\overline{ z }x +  z \overline{y})}{4l_b^2}} 
        \end{align*}
        and we deduce
        \begin{align*}
            \partial_ z \Pi_{z, n}(\bx,\by)
                =& \dfrac{\overline{ y - z }}{2l_b^2} \Pi_{z, n}(\bx,\by) + \dfrac{\overline{ z - x }}{2l_b^2} \Pi_{z, n-1}(\bx,\by) \\
                =& \dfrac{\overline{ y - x }}{2l_b^2} \psi_{z, n}(\bx) \overline{\psi_{z, n}(\by)}
                    + \dfrac{\overline{ x - z  }}{2l_b^2} \psi_{z, n}(\bx) \overline{\psi_{z, n}(\by)}
                    + \dfrac{\overline{ z - x }}{2l_b^2} \psi_{z, n-1}(\bx) \overline{\psi_{z, n-1}(\by)} \\
                =& \dfrac{\overline{ y - x }}{2l_b^2} \psi_{z, n}(\bx) \overline{\psi_{z, n}(\by)}
                     - i\dfrac{\sqrt{n + 1}}{\sqrt{2}l_b} \psi_{z, n+1}(\bx) \overline{\psi_{z, n}(\by)}
                    + i\dfrac{\sqrt{n}}{\sqrt{2}l_b} \psi_{z, n}(\bx) \overline{\psi_{z, n-1}(\by)}
        \end{align*}
        together with 
        \begin{align*}
            \partial_{\overline{ z }} \Pi_{z, n}(\bx,\by)
                =& \dfrac{ x - z  }{2l_b^2} \Pi_{z, n}(\bx,\by) + \dfrac{ z - y }{2l_b^2} \Pi_{z, n-1}(\bx,\by) \\
                =& \dfrac{ x - y }{2l_b^2} \psi_{z, n}(\bx) \overline{\psi_{z, n}(\by)}
                    + \dfrac{ y - z }{2l_b^2} \psi_{z, n}(\bx) \overline{\psi_{z, n}(\by)}
                    + \dfrac{ z - y }{2l_b^2} \psi_{z, n-1}(\bx) \overline{\psi_{z, n-1}(\by)} \\
                =& \dfrac{ x - y }{2l_b^2} \psi_{z, n}(\bx) \overline{\psi_{z, n}(\by)}
                    +i \dfrac{\sqrt{n + 1}}{\sqrt{2}l_b} \psi_{z, n}(\bx) \overline{\psi_{z, n+1}(\by)}
                    - i\dfrac{\sqrt{n}}{\sqrt{2}l_b} \psi_{z, n-1}(\bx) \overline{\psi_{z, n}(\by)}.
        \end{align*}
        This leads to
        \begin{align*}
            \partial_{z_1} \Pi_{z, n}(\bx,\by) 
                =& \prth{\partial_ z + \partial_{\overline{ z }}} \Pi_{z, n}(\bx,\by) \\
                =&  i\dfrac{x_2 - y_2}{l_b^2}\Pi_{z, n}(\bx,\by)
                    +i \dfrac{\sqrt{n + 1}}{\sqrt{2}l_b}\prth{\psi_{z, n}(\bx) \overline{\psi_{z, n+1}(\by)} - \psi_{z, n+1}(\bx) \overline{\psi_{z, n}(\by)}} \\
                    &- i\dfrac{\sqrt{n}}{\sqrt{2}l_b}\prth{\psi_{z, n-1}(\bx) \overline{\psi_{z, n}(\by)} - \psi_{z, n}(\bx) \overline{\psi_{z, n-1}(\by)}}
        \end{align*}
        and
        \begin{align*}
            \partial_{z_2} \Pi_{z, n}(\bx,\by)
                =& i \prth{\partial_ z - \partial_{\overline{ z }}} \Pi_{z, n}(\bx,\by) \\
                =& i\dfrac{y_1 - x_1}{l_b^2}\Pi_{z, n}(\bx,\by)
                    + \dfrac{\sqrt{n + 1}}{\sqrt{2}l_b}\prth{\psi_{z, n}(\bx) \overline{\psi_{z, n+1}(\by)} + \psi_{z, n+1}(\bx) \overline{\psi_{z, n}(\by)}} \\
                    &- \dfrac{\sqrt{n}}{\sqrt{2}l_b}\prth{\psi_{z, n-1}(\bx) \overline{\psi_{z, n}(\by)} + \psi_{z, n}(\bx) \overline{\psi_{z, n-1}(\by)}}
        \end{align*}
        It follows that
        \begin{align*}
            \nabla^\perp_{\bz}  \Pi_{z, n}(\bx,\by) 
                =& i\dfrac{x - y}{l_b^2} \Pi_{z, n}(\bx,\by)
                     - \dfrac{\sqrt{n+1}}{\sqrt{2}l_b}\prth{\matrx{1 & 1 \\ -i & i}} \prth{\matrx{\psi_{z, n}(\bx) \overline{\psi_{z, n+1}(\by)}\\ \psi_{z, n+1}(\bx) \overline{\psi_{z, n}(\by)}}} \\
                     &+ \dfrac{\sqrt{n}}{\sqrt{2}l_b}\prth{\matrx{1 & 1 \\ -i & i}} \prth{\matrx{\psi_{z, n-1}(\bx) \overline{\psi_{z, n}(\by)}\\ \psi_{z, n}(\bx) \overline{\psi_{z, n-1}(\by)}}},
        \end{align*}
        which is~\eqref{eq:d1 pi le N Z bis}. The summation over $n$ cancels telescopic terms, leading to
        \begin{align*}
            \nabla^\perp_{\bz}  \Pi_{z, \le M}(\bx,\by) = 
                i \dfrac{x - y}{l_b^2} \Pi_{z, \le M}(\bx,\by) - \dfrac{\sqrt{M+1}}{\sqrt{2}l_b}\prth{\matrx{1 & 1 \\ -i & i}} \prth{\matrx{\psi_{z, M}(\bx) \overline{\psi_{z, M+1}(\by)}\\ \psi_{z, M+1}(\bx) \overline{\psi_{z, M}(\by)}}}.
        \end{align*}
    \end{proof}

\subsection{Conservation properties}\label{ssec:conserved}

    We next state some basic properties of the Hartree dynamics~\eqref{eq:HF}.

    \begin{lemma}[\textbf{Conservation of mass and Pauli principle}]\label{lem:fermions conserv}\mbox{}\\
        Assume $\gamma_b \in L^\infty \prth{\mathbb{R}_+, \mathcal{L}^1\prth{L^2\prth{\mathbb{R}^2}}}$ solves
        \begin{align*}
            \partial_t \gamma_b(t) = \dfrac{1}{il_b^2} \sbra{\mathscr{L}_b + V + w \star \rho_{\gamma_b(t)}, \gamma_b}
        \end{align*}
        and satisfies
        \begin{align*}
            \Tr{\gamma_b(0)} = 1, \quad 0 \le \gamma_b(0) \le 2\pi l_b^2
        \end{align*}
        then $\forall t\in \mathbb{R}_+$,
        \begin{align*}
            \Tr{\gamma_b(t)} = 1, \quad  0 \le \gamma_b(t) \le 2\pi l_b^2
        \end{align*}
    \end{lemma}
    \begin{proof}
        This follows from the fact that the dynamics is Hamiltonian.
    \end{proof}

We also have

    \begin{lemma}[\textbf{Energy conservation}]\label{lem:E conserv}\mbox{}\\
        Assume $\gamma_b \in L^\infty \prth{\mathbb{R}_+, \mathcal{L}^1\prth{L^2\prth{\mathbb{R}^2}}}$ solves
        \begin{align*}
            \partial_t \gamma_b(t) = \dfrac{1}{il_b^2} \sbra{\mathscr{L}_b + V + w \star \rho_{\gamma_b(t)}, \gamma_b}
        \end{align*}
        with initial datum satisfying
        \begin{align*}
            \Tr{\gamma_b(0)} = 1, \quad \Tr{\gamma_b(0) H_b(0)} < C
        \end{align*}
        then
        \begin{align*}
            \dfrac{d}{dt} \Tr{\gamma_b(t) H_b(t)} = 0.
        \end{align*}
        Moreover, $\forall W \in L^\infty\prth{\mathbb{R}^2}$,
        \begin{align*}
            \Tr{\gamma_b \mathscr{L}_b} \le \abs{\Tr{\gamma_b\prth{\mathscr{L}_b + W}}} + \norm{W}_{L^\infty}
        \end{align*}
    \end{lemma}
    
    \begin{proof}
        The equation being Hamiltonian, the total energy is certainly conserved. Then, the kinetic energy is bounded by
        \begin{align*}
            \Tr{\gamma_b \mathscr{L}_b} 
                = \Tr{\gamma_b\prth{\mathscr{L}_b + W}} - \Tr{\gamma_b W} 
                \le \abs{\Tr{\gamma_b\prth{\mathscr{L}_b + W}}} + \norm{W}_{L^\infty}
        \end{align*}
    \end{proof}
%
%
%
%

\section{Husimi functions and semi-classical densities}  \label{sec:densities}

    In this section, we provide the first set of tools mentioned below Theorem~\ref{th:V of HF}, namely the construction and properties of Husimi functions and associated semi-classical densities.

\subsection{Semi-classical density} 
    
    Given a coherent state basis over a Hilbert space and a trace-class operator acting on the latter, the notion of Husimi function/lower symbol is fairly standard (see e.g.~\cite[Section~3.3]{Rougerie-EMS} and references therein): 
    
    \begin{definition}[\textbf{Husimi function and associated density}]\label{def:Husimi}\mbox{}\\
        For $\gamma \in  \mathcal{L}^1\prth{L^2\prth{\mathbb{R}^2}}$, let
        \begin{equation}\label{eq:sc density}
            m_\gamma(z, n) \coloneq  \dfrac{1}{2\pi l_b^2} \bk{\psi_{z, n}}{\gamma\psi_{z, n}} =\dfrac{1}{2\pi l_b^2} \Tr{\Pi_{z, n} \gamma} 
        \end{equation}
        with the associated semi-classical density
        $$\rho_\gamma^{sc} (z) \coloneq \sum_{n= 0} ^\infty m_\gamma(z, n). $$
        For $M\in \N$ such that $1\ll M \ll l_b^{-2}$, we define the truncated version thereof
        \begin{equation}\label{eq:LowLL sc density}
            \rho^{sc, \le M}_\gamma(z) \coloneq \dfrac{1}{2\pi l_b^2} \Tr{\gamma \Pi_{z, \le M}} 
        \end{equation}
        \hfill$\diamond$
    \end{definition}
    
    The parameter $M$ represents the number of Landau levels we take into account for the semi-classical approximations. It will be important to have $M\gg 1$ when $b\to \infty$ to recover the true quantum density $\rho_\gamma$ of a general $\gamma$ (with reasonable magnetic kinetic energy). On the other hand, the larger $n$ the less the coherent state $\psi_{z, n}$ is localized around $z$, making the approximation less efficient. For this reason we will mostly use the truncated~\eqref{eq:LowLL sc density} for a suitable $1 \ll M \ll l_b^{-2}$, and use moments of the kinetic energy to discard the contribution to the density of Landau levels with index $n>M$.
    
    The main estimate we will rely on is as follows:
    
    \begin{proposition}[\textbf{Convergence of the truncated semi-classical density}], \label{prop:rho sc cvg}\mbox{}\\
        Let $k\ge 0, \gamma_b \in  \mathcal{L}^1\prth{L^2\prth{\mathbb{R}^2}}$ and assume 
        \begin{equation}\label{eq:pauli sc conv}
            \Tr{\gamma_b} = 1, \quad 0\le \gamma_b \le 2\pi l_b^2, \quad \Tr{\gamma_b \mathscr{L}_b^k} < \infty
        \end{equation}
        then $\forall \varphi \in L^\infty\prth{\R^2} \cap H^1\prth{\R^2}$,
        \begin{align}\label{eq:GYRO esti 2}
            \abs{\int_{\mathbb{R}^2}\varphi\prth{\rho_{\gamma_b} - \rho_{\gamma_b}^{sc, \le M}}}
                \le& \norm{\varphi}_{L^\infty} M^{-\frac{k}{2}} \sqrt{\Tr{\gamma_b\Pi_{>M}\mathscr{L}_b^k}} \nonumber \\
                    &+ C \norm{\nabla\varphi}_{L^2} \sqrt{\Tr{\gamma_b \mathscr{L}_b^k}}\cdot \syst{\matrx{
                        M^{1-\frac{k}{2}} l_b &\text{ if } k < 2 \\
                        \sqrt{\ln(M)} l_b &\text{ if } k = 2 \\
                        l_b &\text{ if } k > 2
                    }} 
        \end{align}
    \end{proposition}

    The first term on the right-hand side of ~\eqref{eq:GYRO esti 2} corresponds to the contribution of high Landau levels. In our approach to the semi-classical limit we will rely solely on one moment of the magnetic kinetic energy being bounded uniformly in time, $k=1$. Observe that then the error terms will be small if $\sqrt{M} l_b \ll 1$. This constraint has a physical meaning. Specifically, from the expression of the coherent state~\eqref{eq:psi nR}, we can infer that the characteristic localization length for particles in nLL is $\sqrt{n} l_b$. Therefore, $\rho_{\gamma_b}^{sc, \le M}$ with $\sqrt{M} l_b \ll 1$ represents the semi-classical density of particles localized with precision better than what is aimed at in the classical equation. The fluctuations of the position operator will be small in the limit for this contribution. 
    
We will have to test $\gamma_b$ against multiplication operators by nice functions $\varphi$. Using the resolution of the identity from Lemma~\ref{lem:resol}, the gist of the estimate consists in vindicating that 
$$ \frac{1}{2\pi l_b^2} \sum_{n\geq 0} \int_{\R^2} \iint_{\R^4} \left(\varphi (\bx) - \varphi( \bz )\right)\gamma_b (\bx ,\by ) \Pi_{z,n} (\by , \bx ) \dd \bx   \dd \by  \dd \bz  \underset{b\to \infty}{\simeq} 0 $$
using that the coherent projector's kernel
$$ \Pi_{z,n} (\bx , \by ) = \psi_{z,n} (\bx) \overline{\psi_{z,n} (\by)}$$
is strongly localized, for moderate values of $n$, around $\bz=\by=\bx$. We start with a lemma that will deal with the part of the sum bearing on low Landau levels, $A \left( \bz \right)$ below playing the role of the multiplication operator by $\varphi (\bullet) - \varphi( \bz )$.

    \begin{lemma}[\textbf{Bounds on expectations of truncated operators}]\label{lem:moment estimate}\mbox{}\\
        Let $\forall z \in \mathbb{R}^2, \prth{A _n\left( \bz \right)}_{n\in\mathbb{N}}$ be a sequence of operators on $L^2\prth{\mathbb{R}^2}$, $k\le 0, \gamma_b \in  \mathcal{L}^1\prth{L^2\prth{\mathbb{R}^2}}$ and assume 
        \begin{equation}\label{eq:pauli in lemma}
            \Tr{\gamma_b} = 1, \quad 0\le \gamma_b \le 2\pi l_b^2 .
        \end{equation}
        Then
        \begin{equation}\label{eq:moment lemma 1}
            \dfrac{1}{2\pi l_b^2}\sum_{n=0}^M\int_{\mathbb{R}^2} \abs{\Tr{A _n\left( \bz \right) \gamma_b \Pi_{z, n}}}\dd \bz  
                \le \dfrac{1}{2} \sqrt{\Tr{\gamma_b \mathscr{L}_b^k}} \prth{\sum_{n=0}^M \dfrac{1}{(n+1)^k}\int_{\mathbb{R}^2} \Tr{\abs{A _n\left( \bz \right)}^2 \Pi_{z, n}}\dd \bz }^{\frac{1}{2}} 
        \end{equation}
        and $\forall \varphi \in L^1\prth{\mathbb{R}^2} \cap L^\infty\prth{\mathbb{R}^2}$,
        \begin{multline}\label{eq:moment lemma 2}
            \dfrac{1}{2\pi l_b^2}\sum_{n=0}^M\int_{\mathbb{R}^2} \abs{\varphi( \bz ) \Tr{A _n\left( \bz \right) \gamma_b \Pi_{z, n}}}\dd \bz 
                \\ \le \dfrac{\sqrt{\norm{\varphi}_{L^1}\norm{\varphi}_{L^\infty}}}{2}\sqrt{\Tr{\gamma_b \mathscr{L}_b^k}}
                \prth{\sum_{n=0}^M (n+1)^{-k} \sup\limits_{z\in\mathbb{R}^2} \norm{A _n\left( \bz \right)}_{\mathcal{L}^2}^2}^{\frac{1}{2}}
        \end{multline}
    \end{lemma}
    \begin{proof}
        The main idea is to exploit the sum over Landau levels to introduce moments of the kinetic energy. From \cref{eq:energy scaling}, we notice that $\forall n \in \mathbb{N}$,
        \begin{align}
            (n+1)^k \Tr{\gamma_b \Pi_n} = \prth{\frac{n+1}{2\hbar b\prth{n + \dfrac{1}{2}}}}^k \Tr{\gamma_b \Pi_n \mathscr{L}_b^k} 
                \le \Tr{\gamma_b \Pi_n \mathscr{L}_b^k} \label{eq:power to L}
        \end{align}
        so
        \begin{align}\label{eq:kth sum to moment}
            \sum_{n=0}^M (n+1)^k \Tr{\gamma_b \Pi_n } 
                \le \Tr{\gamma_b \mathscr{L}_b^k} 
        \end{align}
        Applying the Cauchy-Schwarz inequality, using $\Pi_{z, n}^2 = \Pi_{z, n}$ and~\eqref{eq:pauli in lemma} followed by Young's inequality we find
        \begin{align}\label{eq:lem moment esti}
            \sum_{n=0}^M \abs{ \Tr{A _n\left( \bz \right) \gamma_b \Pi_{z, n}}}
                \le& \sum_{n=0}^M \sqrt{\Tr{\abs{A _n\left( \bz \right)}^2 \Pi_{z, n}}} \sqrt{\Tr{\gamma_b^2 \Pi_{z, n}}} \nonumber\\
                \le& \sqrt{2\pi l_b^2} \sum_{n=0}^M \sqrt{\Tr{\abs{A _n\left( \bz \right)}^2 \Pi_{z, n}}} \sqrt{\Tr{\gamma_b \Pi_{z, n}}} \nonumber\\
                \le& \sqrt{\dfrac{\pi}{2}}l_b \sum_{n=0}^M \prth{\dfrac{1}{\epsilon_n}\Tr{\abs{A _n\left( \bz \right)}^2 \Pi_{z, n}} + \epsilon_n\Tr{\gamma_b \Pi_{z, n}}} 
        \end{align}
        where we will choose $\epsilon_n \coloneq \epsilon(n+1)^k$ for some $\epsilon>0$. Integrating in $z$, using \eqref{eq:res ID dz} and inserting \eqref{eq:kth sum to moment} gives
        \begin{align*}
            &\dfrac{1}{2\pi l_b^2}\sum_{n=0}^M \int_{\mathbb{R}^2}\abs{\Tr{A _n\left( \bz \right) \gamma_b \Pi_{z, n}}}\dd \bz  \\
                \le& \sqrt{\dfrac{\pi}{2}}l_b \sum_{n=0}^M \prth{ \dfrac{1}{2\pi l_b^2\epsilon_n}\int_{\mathbb{R}^2}
                \Tr{\abs{A _n\left( \bz \right)}^2 \Pi_{z, n}}\dd \bz  + \epsilon_n \Tr{\gamma_b \Pi_{n}}} \\
                \le& \sqrt{\dfrac{\pi}{2}}l_b \prth{ \sum_{n=0}^M \dfrac{1}{2\pi l_b^2\epsilon (n+1)^k}\int_{\mathbb{R}^2} \Tr{\abs{A _n\left( \bz \right)}^2 \Pi_{z, n}}\dd \bz  + \epsilon \Tr{\gamma_b \mathscr{L}_b^k}}
        \end{align*}
        Choosing now
        \begin{align*}
            \epsilon \coloneq \prth{\dfrac{1}{2\pi l_b^2}\sum_{n=0}^M \dfrac{1}{(n+1)^k}\int_{\mathbb{R}^2} \Tr{\abs{A _n\left( \bz \right)}^2 \Pi_{z, n}}\dd \bz \frac{1}{\Tr{\gamma_b \mathscr{L}_b^k}}}^{\frac{1}{2}}
        \end{align*}
        leads to~\eqref{eq:moment lemma 1}:
        \begin{align*}
            &\dfrac{1}{2\pi l_b^2}\sum_{n=0}^M \int_{\mathbb{R}^2}\abs{\Tr{A _n\left( \bz \right) \gamma_b \Pi_{z, n}}}\dd \bz  \\
                \le& \sqrt{\dfrac{\pi}{2}}l_b
            \prth{\dfrac{1}{2\pi l_b^2}\sum_{n=0}^M \dfrac{1}{(n+1)^k}\int_{\mathbb{R}^2} \Tr{\abs{A _n\left( \bz \right)}^2 \Pi_{z, n}}\dd \bz }^{\frac{1}{2}} \sqrt{\Tr{\gamma_b \mathscr{L}_b^k}} \\
                =& \dfrac{1}{2} \sqrt{\Tr{\gamma_b \mathscr{L}_b^k}} \prth{\sum_{n=0}^M \dfrac{1}{(n+1)^k}\int_{\mathbb{R}^2} \Tr{\abs{A _n\left( \bz \right)}^2 \Pi_{z, n}}\dd \bz }^{\frac{1}{2}}
        \end{align*}
        As regards~\eqref{eq:moment lemma 2}, we start again from \eqref{eq:lem moment esti} to find
        \begin{align*}
            &\dfrac{1}{2\pi l_b^2}\sum_{n=0}^M\int_{\mathbb{R}^2} \abs{\varphi( \bz ) \Tr{A _n\left( \bz \right) \gamma_b \Pi_{z, n}}}\dd \bz  \\
                \le& \sqrt{\dfrac{\pi}{2}}l_b \prth{ \sum_{n=0}^M \dfrac{1}{2\pi l_b^2\epsilon (n+1)^k}\int_{\mathbb{R}^2} \abs{\varphi( \bz )}\Tr{\abs{A _n\left( \bz \right)}^2 \Pi_{z, n}}\dd \bz  + \epsilon \norm{\varphi}_{L^\infty}\Tr{\gamma_b \mathscr{L}_b^k}} \\
                \le& \sqrt{\dfrac{\pi}{2}}l_b \prth{ \sum_{n=0}^M \dfrac{1}{2\pi l_b^2\epsilon (n+1)^k} \norm{\varphi}_{L^1}\sup\limits_{z\in\mathbb{R}^2} \norm{A _n\left( \bz \right)}_{\mathcal{L}^2}^2 + \epsilon \norm{\varphi}_{L^\infty}\Tr{\gamma_b \mathscr{L}_b^k}}.
        \end{align*}
        We conclude the proof by choosing 
        \begin{align*}
            \epsilon \coloneq \sqrt{\dfrac{\norm{\varphi}_{L^1}}{2\pi l_b^2 \norm{\varphi}_{L^\infty}\Tr{\gamma_b \mathscr{L}_b^k}}\sum_{n=0}^M \dfrac{1}{(n+1)^k } \sup\limits_{z\in\mathbb{R}^2}\norm{A _n\left( \bz \right)}_{\mathcal{L}^2}^2}
        \end{align*}
    \end{proof}

    We may next proceed to the 
    
    \begin{proof}[Proof of Proposition~\ref{prop:rho sc cvg}]
    For $\varphi \in C^\infty_c\prth{\mathbb{R}^2}$, we write
        \begin{align}\label{eq:rho cgv 0}
            \int_{\mathbb{R}^2}\varphi\prth{\rho_{\gamma_b} - \rho_{\gamma_b}^{sc, \le M}} 
                =& \Tr{\varphi \gamma_b} - \dfrac{1}{2\pi l_b^2} \int_{\mathbb{R}^2} \varphi( \bz ) \Tr{\Pi_{z, \le M} \gamma_b}\dd \bz  \nonumber\\
                =& \dfrac{1}{2\pi l_b^2} \int_{\mathbb{R}^2}  \Tr{\prth{\varphi - \varphi( \bz )}\Pi_{z, \le M} \gamma_b}\dd \bz  + \Tr{\varphi \Pi_{> M} \gamma_b} 
        \end{align}

        \medskip
        
    \noindent\textbf{Step 1, low Landau levels.}
        Using the change of variables
        \begin{align*}
            \dfrac{\bx - \bz }{\sqrt{2}l_b} \mapsto \bx
        \end{align*}
        and Taylor's theorem, we get
        \begin{align}\label{eq:Taylor LLL}
            \int_{\mathbb{R}^2} \Tr{\abs{\varphi - \varphi( \bz )}^2\Pi_{z, n}}\dd \bz 
                =& \iintr_{\R^2 \times \R^2}\abs{\varphi(\bx) - \varphi( \bz )}^2 \Pi_{z, n}(\bx ,  \bx )\dd \bx   \dd \bz \nonumber \\
                =& \dfrac{1}{2\pi n! l_b^2} \iintr_{\R^2 \times \R^2} \abs{\varphi(\bx) - \varphi( \bz )}^2 \abs{\dfrac{x - z }{\sqrt{2}l_b}}^{2n} e^{-\frac{\abs{x-z}^2}{2l_b^2}} \dd \bx   \dd \bz  \nonumber\\
                =& \dfrac{1}{\pi n!} \iintr_{\R^2 \times \R^2} \abs{\varphi\prth{\bz + \sqrt{2}l_b \bx} - \varphi( \bz )}^2 \abs{x}^{2n} e^{-\abs{x}^2} \dd \bx \dd \bz \nonumber \\
                =& \dfrac{1}{\pi n!} \iintr_{\R^2 \times \R^2} \abs{\int_0^1 \nabla\varphi \left( \bz + \sqrt{2}l_b s \bx \right)\cdot \sqrt{2}l_b \bx \dd s}^2 \abs{x}^{2n} e^{-\abs{x}^2} \dd \bx \dd \bz \nonumber\\
                \le& \dfrac{2l_b^2}{\pi n!} \iintr_{\R^2 \times \R^2} \int_0^1 \abs{\nabla\varphi \left( \bz + \sqrt{2}l_b s \bx \right)}^2 \abs{x}^{2(n+1)} e^{-\abs{x}^2} \dd s \dd \bx \dd \bz  \nonumber\\
                \le& \dfrac{2l_b^2}{\pi n!} \norm{\nabla\varphi}_{L^2}^2 \int_{\R^2} \abs{x}^{2(n+1)} e^{-\abs{x}^2} \dd \bx 
                = 2\norm{\nabla\varphi}_{L^2}^2 (n+1)l_b^2
        \end{align}        
        Introducing the notation
        \begin{equation}\label{eq:power+log}
            p_\lambda(n) \coloneq n^{-\lambda} \mathbb{1}_{\lambda < 0} + \ln(n) \mathbb{1}_{\lambda = 0} +  \mathbb{1}_{\lambda > 0} 
        \end{equation}
        we have the asymptotics 
        \begin{align*}
            \sum_{n=0}^M \dfrac{1}{(n+1)^{k-1}} 
                = \syst{\matrx{
                    \mathcal{O}\prth{M^{2-k}} &\text{ if } k < 2 \\
                    \mathcal{O}\prth{\ln(M)} &\text{ if } k = 2 \\
                    \mathcal{O}\prth{1} &\text{ if } k > 2
                }}
                = \mathcal{O}\prth{p_{k-2}(M)}
        \end{align*}
        Hence ~\eqref{eq:Taylor LLL} gives
        \begin{align*}
            \sum_{n=0}^M \dfrac{1}{(n+1)^k}\int_{\mathbb{R}^2} \Tr{\abs{\varphi - \varphi( \bz )}^2 \Pi_{z, n}}\dd \bz 
                \le&  2 l_b^2 \norm{\nabla \varphi}_{L^2}^2 \sum_{n=0}^M (n+1)^{1-k}
                = C l_b^2 \norm{\nabla \varphi}_{L^2}^2  p_{k-2}(M)
        \end{align*}
        We next apply ~\eqref{eq:moment lemma 1} to $A _n\left( \bz \right) \coloneq \varphi - \varphi( \bz )$ and obtain
        \begin{align}\label{eq:rho cvg LowLL}
            &\abs{\dfrac{1}{2\pi l_b^2}\int_{\mathbb{R}^2} \Tr{\prth{\varphi - \varphi( \bz )}\gamma_b \Pi_{z, \le M}}\dd \bz } \nonumber\\
            \le& \dfrac{1}{2} \sqrt{\Tr{\gamma_b \mathscr{L}_b^k}} \prth{\sum_{n=0}^M \dfrac{1}{(n+1)^k}\int_{\mathbb{R}^2} \Tr{\abs{\varphi - \varphi( \bz )}^2 \Pi_{z, n}}\dd \bz }^{\frac{1}{2}}  \nonumber\\
            \le& C  \norm{\nabla \varphi}_{L^2} \sqrt{\Tr{\gamma_b \mathscr{L}_b^k}} l_b\sqrt{p_{k-2}(M)}  
        \end{align}
        
        \medskip
        
    \noindent\textbf{Step 2, high Landau levels.} We remark that
        \begin{align*}
                \abs{\Tr{\varphi\Pi_{>M} \gamma_b}}
                \le \sqrt{\Tr{\gamma_b\abs{\varphi}^2} \Tr{\gamma_b\Pi_{>M}}}
                \le \norm{\varphi}_{L^\infty} \sqrt{\Tr{\gamma_b\Pi_{>M}}}
        \end{align*}
        and
        \begin{equation}\label{eq:end of LL series}
            \Tr{ \gamma_b\Pi_{>M}} 
                \le \sum_{n>M} \dfrac{n^k}{M^k}\Tr{\gamma_b\Pi_n}
                \le \dfrac{1}{M^k} \sum_{n>M} \Tr{\gamma_b\Pi_n \mathscr{L}_b^k}
                = \dfrac{1}{M^k} \Tr{\gamma_b\Pi_{>M}\mathscr{L}_b^k} .
        \end{equation}
        Hence 
        \begin{equation}\label{eq:rho cvg HLL}
            \abs{\Tr{\varphi\Pi_{>M} \gamma_b}} \le \norm{\varphi}_{L^\infty} M^{-\frac{k}{2}} \sqrt{\Tr{\gamma_b\Pi_{>M}\mathscr{L}_b^k}} 
        \end{equation}
        We conclude by combining~\eqref{eq:rho cvg LowLL},~\eqref{eq:rho cvg HLL} and \eqref{eq:rho cgv 0}.
 \end{proof}

 \subsection{Improved convergence with confinement}   
    
    To complete the proof of Theorem~\ref{th:conv} we need to improve~\eqref{eq:GYRO esti 2} to an estimate in Wasserstein-1 distance. We may do this at the price of an extra confining assumption:
    
    \begin{proposition}[\textbf{Convergence of the truncated semi-classical density with confinement}]\label{prop:cvg t0}\mbox{}\\
        Let $\beta > 0, p > 3$. We make the same assumptions as in~Proposition~\ref{prop:rho sc cvg}, with in addition 
        \begin{equation}\label{eq:R LB}
            \Tr{\gamma_b \abs{X}^p} < \infty 
        \end{equation}
        If we denote
        \begin{align}
            \rho_b \coloneq \frac{\rho_{\gamma_b}^{sc, \le M}}{\Tr{\gamma_b \Pi_{\le M}}} \label{eq:final rho}
        \end{align}
        and assume $4\sqrt{M}l_b \le l_b^{-\beta}$,
        \begin{align}\label{eq:W1withconf}
            & W_1(\rho_{\gamma_b}, \rho_b)
            \le C(p) \prth{1 + \Tr{\gamma_b \abs{X}^p}} \left( \Tr{\gamma_b \Pi_{> M} \mathscr{L}_b}M^{-k}\right.\\
                &\left. + \sqrt{\Tr{\gamma_b \mathscr{L}_b^k}} l_b^{1-\beta}   \syst{\matrx{
                        M^{1-\frac{k}{2}}  &\text{ if } k < 2 \\
                        \sqrt{\ln(M)}  &\text{ if } k = 2 \\
                        1 &\text{ if } k > 2
                    }} 
                +  M^{-\frac{k}{2}} \sqrt{\Tr{\gamma_b\Pi_{>M}\mathscr{L}_b^k}}
                +  l_b^{\beta(p-1)}
                +  M l_b^{\beta(p-3) -2} \right)
        \end{align}
    \end{proposition}

    We need a technical lemma containing some basic estimates:
    
    \begin{lemma}[\textbf{Technical integration results}]\label{lem:integ results}\mbox{}\\
        Recalling the definition~\eqref{eq:psi nR},
        \begin{equation}\label{eq:psinz L1}
            \norm{\psi_{z, n}}_{L^1} \le C \prth{n+1}^{\frac{1}{4}} l_b. 
        \end{equation}
        Let $n\in\mathbb{N}, a > 0$ and
        \begin{align*}
            I_n(a) \coloneq \int_a^\infty t^n e^{-\frac{t^2}{2}}\dd t
        \end{align*}
        then
        \begin{align*}
            I_{2n+1}(a) = 2^n n! e^{-\frac{a^2}{2}} \sum_{i=0}^n \dfrac{1}{i!}\prth{\frac{a^2}{2}}^{i}.
        \end{align*}
    \end{lemma}
    \begin{proof}
        Using Stirling's formula
        \begin{align*}
            &\norm{\psi_{z, n}}_{L^1} \\
                =& \dfrac{1}{\sqrt{2\pi n!}l_b} \int_{\R^2} \abs{\dfrac{x-z}{\sqrt{2l_b}}}^n e^{-\frac{\abs{x-z}^2}{4l_b^2}}\dd \bx   
                = \sqrt{\frac{2}{\pi n!}}l_b \int_{\R^2} \abs{x}^n e^{-\frac{\abs{x}^2}{2}} \dd \bx   
                = \sqrt{\frac{2\pi}{n!}}l_b \int_{\R_+} 2 t^{n+1} e^{-\frac{t^2}{2}} \dd t \\
                =& \sqrt{\frac{2\pi}{n!}}l_b \int_{\R_+}  t^{\frac{n}{2}} e^{-\frac{t}{2}} \dd t
                =  \sqrt{\frac{2\pi}{n!}}l_b 2^{\frac{n}{2} + 1} \int_{\R_+}  t^{\frac{n}{2}} e^{-t} \dd t
                =  2^{\frac{n+3}{2}}\sqrt{\pi}l_b \dfrac{\Gamma\prth{\dfrac{n}{2}+1}}{\sqrt{n!}} \\
                \eqvl\displaylimits_{n \to \infty}& 2^{\frac{n+3}{2}}\sqrt{\pi} l_b \dfrac{\sqrt{2\pi \frac{n}{2}} \prth{\frac{n}{2e}}^{\frac{n}{2}}}{\prth{2\pi n}^{\frac{1}{4}}\prth{\frac{n}{e}}^{\frac{n}{2}}}
                = 2^{\frac{5}{4}}\pi^{\frac{3}{4}} n^{\frac{1}{4}} l_b
                \eqvl\displaylimits_{n \to \infty} 2^{\frac{5}{4}}\pi^{\frac{3}{4}} \prth{n+1}^{\frac{1}{4}} l_b.
        \end{align*}
        As for the second claim in the lemma, an integration by parts shows
        \begin{align*}
            I_{n+2}(a) = \prth{n+1}I_n(a) + a^{n+1} e^{-\frac{a^2}{2}} 
        \end{align*}
        so for odd integers
        \begin{align*}
            I_{2(n+1)+1}(a) = I_{2n+1+2}(a) = 2\prth{n+1}I_{2n+1}(a) + a^{2(n+1)} e^{-\frac{a^2}{2}} 
        \end{align*}
        hence by induction,
        \begin{align*}
            I_{2n+1}(a) = 2^n n!\prth{I_1(a) + e^{-\frac{a^2}{2}} \sum_{i=1}^n \dfrac{1}{i!}\prth{\frac{a^2}{2}}^{i}}
        \end{align*}
        and
        \begin{align*}
            I_1(a) = e^{-\frac{a^2}{2}}
        \end{align*}
        corresponds to the index $i=0$ in the sum.
    \end{proof}
    
    We now turn to the
    
    \begin{proof}[Proof of Proposition~\ref{prop:cvg t0}]
        Using the dual representation of the Wasserstein-1 metric, we aim at estimating
        \begin{align*}
            \abs{\int_{\mathbb{R}^2}\varphi\prth{\rho_{\gamma_b} - \rho_b}}
        \end{align*}
        for a Lipschitz function $\varphi.$ Since $\rho_{\gamma_b}$ and $\rho_b$ have the same integral we assume without loss of generality that $\varphi(0) = 0$, and hence
        \begin{equation}\label{eq:lip varphi}
            \abs{\varphi(\bx)} \le \norm{\nabla \varphi}_{L^\infty} \abs{x}. 
        \end{equation}
        Let 
        $$R\coloneq l_b^{-\beta}$$
        and define a partition of unity $,\chi_R,\eta_{R}\in C^\infty\prth{\R^2, \mathbb{R_+}}$ such that
        \begin{align}\label{eq:eta def}
            &\chi_R + \eta_R = 1 \nonumber \\
            &\abs{z} \le R \implies \chi_R = 1, \eta_R = 0 \nonumber\\
            &\abs{z} \ge 2R \implies \chi_R = 0, \eta_R = 1 \nonumber\\
            & \norm{\nabla \eta_R}_{L^\infty} \le \dfrac{2}{R}
        \end{align}
        Using \eqref{eq:rho cgv 0} for $\varphi \chi_R$ instead of $\varphi$, we decompose 
        \begin{align}\label{eq:rho cgv 2.0}
            \int_{\mathbb{R}^2}\varphi\prth{\rho_{\gamma_b} - \rho_{\gamma_b}^{sc, \le M}}
                =& \dfrac{1}{2\pi l_b^2} \int_{\R^2} \Tr{\prth{\varphi \chi_R - \varphi( \bz ) \chi_R( \bz )}\Pi_{z, \le M} \gamma_b}\dd \bz  + \Tr{\varphi \chi_R \Pi_{> M} \gamma_b} \nonumber \\
                +& \int_{\mathbb{R}^2}\varphi \eta_R\prth{\rho_{\gamma_b} - \rho_{\gamma_b}^{sc, \le M}} 
        \end{align}

        \medskip
        
    \noindent\textbf{Step 1, low Landau levels.} With the bounds ~\eqref{eq:lip varphi} and \eqref{eq:eta def} we find that, $\forall \abs{z} \le 2R$,
        \begin{align*}
            \abs{\nabla\prth{\varphi \chi_R}(\bz )} 
                \le \abs{\nabla \varphi( \bz )} + \dfrac{2}{R}\abs{\varphi( \bz )}
                \le \abs{\nabla \varphi( \bz )} + \dfrac{2}{R} \abs{z}\norm{\nabla \varphi}_{L^\infty}
                \le 5 \norm{\nabla \varphi}_{L^\infty}
        \end{align*}
        Hence
        \begin{align*}
            \norm{\nabla\prth{\varphi \chi_R}(\bz)}_{L^2} 
                \le 5 \norm{\nabla \varphi}_{L^\infty} \sqrt{\abs{B(0, 2R)}}
                \le C \norm{\nabla \varphi}_{L^\infty} R
        \end{align*}
        Using~\eqref{eq:rho cvg LowLL} we deduce
        \begin{equation}\label{eq:rho cvg 2 LLL}
            \abs{\dfrac{1}{2\pi l_b^2}\int_{\mathbb{R}^2} \Tr{\prth{\varphi\chi_R - \varphi( \bz )\chi_R( \bz )}\gamma_b \Pi_{z, \le M}}\dd \bz }
                \le C  R l_b \norm{\nabla \varphi}_{L^\infty} \sqrt{\Tr{\gamma_b \mathscr{L}_b^k}} \sqrt{p_{k-2}(M)} 
        \end{equation}
    
    \medskip
    
    \noindent\textbf{Step 2, higher Landau levels.} Since
        \begin{align*}
            \abs{\varphi( \bz ) \chi_R( \bz )} \le \norm{\nabla \varphi}_{L^\infty}\abs{z}
        \end{align*}
       it follows from~\eqref{eq:lip varphi} that
        \begin{align*}
            \abs{\Tr{\varphi\chi_R\Pi_{>M} \gamma_b}}
                \le \sqrt{\Tr{\gamma_b\abs{\varphi\chi_R}^2} \Tr{\gamma_b\Pi_{>M}}}
                \le \norm{\nabla \varphi}_{L^\infty} \sqrt{\Tr{\gamma_b\abs{X}^2}}\sqrt{\Tr{\gamma_b\Pi_{>M}}}.
        \end{align*}
        Inserting~\eqref{eq:end of LL series}, we deduce
        \begin{equation}\label{eq:rho cvg 2 HLL}
            \abs{\Tr{\varphi\chi_R\Pi_{>M} \gamma_b}}
                \le \norm{\nabla \varphi}_{L^\infty} \sqrt{\Tr{\gamma_b\abs{X}^2}} M^{-\frac{k}{2}} \sqrt{\Tr{\gamma_b\Pi_{>M}\mathscr{L}_b^k}} 
        \end{equation}

        \medskip
        
     \noindent\textbf{Step 3, tails of the densities.} We next estimate the third term in~\eqref{eq:rho cgv 2.0} . First, using ~\eqref{eq:lip varphi},
        \begin{align}\label{eq:rho cvg d tail}
            &\abs{\int_{\mathbb{R}^2}\varphi \eta_R \rho_{\gamma_b}} \nonumber\\
                \le& \norm{\nabla \varphi}_{L^\infty} \int_{\abs{z} \ge R}\abs{z} \rho_{\gamma_b}(\bz) \dd \bz 
                \le \dfrac{\norm{\nabla \varphi}_{L^\infty}}{R^{p-1}} \int_{\abs{z}\ge R} \abs{z}^p \rho_{\gamma_b}(\bz)\dd \bz 
                \le \dfrac{\norm{\nabla \varphi}_{L^\infty}}{R^{p-1}} \Tr{\gamma_b \abs{X}^p}.
        \end{align}
    On the other hand, using~\eqref{eq:lip varphi} again,
        \begin{align}\label{eq:rho cvg scd tail start}
            \abs{\int_{\mathbb{R}^2}\varphi \eta_R\rho_{\gamma_b}^{sc, \le M}} 
                &\le \norm{\nabla \varphi}_{L^\infty} \int_{\abs{z}\ge R} \abs{z} \rho_{\gamma_b}^{sc, \le M}(\bz )\dd \bz \nonumber\\
                &= \dfrac{\norm{\nabla \varphi}_{L^\infty}}{2\pi l_b^2} \sum_{n=0}^M \int_{\abs{z}\ge R} \abs{z} \Tr{\gamma_b \Pi_{z, n}}\dd \bz  
        \end{align}
        We write the spectral decomposition of $\gamma_b$ in the manner
        \begin{align*}
            \gamma_b \eqcolon \sum_{i \in \mathbb{N}} \lambda_i \ket{u_i}\bra{u_i}
        \end{align*}
        with $0\leq \lambda_i \leq (2\pi)^{-1} l_b^{-2}$ and $(u_i)_i$ an orthonormal basis of $L^2$. Let $\abs{z} \ge R, n\le M$, using the Cauchy-Schwarz inequality, we estimate
        \begin{align}\label{eq:tr gamma pinz0}
            \Tr{\gamma_b \Pi_{z, n}}
                =& \sum_{i\in\N} \lambda_i \abs{\bk{\psi_{z, n}}{u_i}}^2
                \le \sum_{i\in\N} \lambda_i \prth{\int_{\R^2} \abs{u_i}\sqrt{\abs{\psi_{z, n}}}\cdot \sqrt{\abs{\psi_{z, n}}}}^2 \nonumber\\
                \le& \norm{\psi_{z, n}}_{L^1} \sum_{i\in\mathbb{N}} \lambda_i  \int_{\R^2} \abs{u_i}^2 \abs{\psi_{z, n}}
                = \norm{\psi_{z, n}}_{L^1} \int_{\R^2} \rho_{\gamma_b} \abs{\psi_{z, n}} 
        \end{align}
        Hence
        \begin{align}\label{eq:another ineq}
            \int_{\R^2} \rho_{\gamma_b} \abs{\psi_{z, n}}
                =& \int_{\R^2} \rho_{\gamma_b}(\bx) \prth{1 + \abs{x}^p}\dfrac{\abs{\psi_{z, n}(\bx)}}{1+ \abs{x}^p}\dd \bx   
                \le \Tr{\gamma_b\prth{1+\abs{X}^p}} \norm{\dfrac{\psi_{z, n}}{1+\abs{\bullet}^p}}_{L^\infty} \nonumber\\
                =& \dfrac{1 + \Tr{\gamma_b \abs{X}^p}}{\sqrt{2\pi n!}l_b} \sup_{\bx\in\R^2}\dfrac{\abs{x}^n e^{-\frac{\abs{x}^2}{2}}}{1 + \abs{z + \sqrt{2}l_b x}^p} 
        \end{align}
        The function $t \mapsto t^n e^{-\frac{t^2}{2}}$
        attains its' global maximal on $\R^+$ at $t=\sqrt{n}$ with maximal value $\prth{\dfrac{n}{e}}^{\frac{n}{2}}$ and is decreasing for $t > \sqrt{n}$.
        
        Since we assume $4\sqrt{M}l_b \le R$, if $4 l_b\abs{x} \ge \abs{z}$, we have that
        \begin{align*}
            \abs{x} \ge \dfrac{\abs{z}}{4l_b} \ge \dfrac{R}{4 l_b} \ge \sqrt{M} \ge \sqrt{n}
        \end{align*}
        and hence
        \begin{equation}\label{eq:Linfty 1}
            \dfrac{\abs{x}^n e^{-\frac{\abs{x}^2}{2}}}{1 + \abs{z + \sqrt{2}l_b x}^p} 
                \le \prth{\dfrac{\abs{z}}{4l_b}}^n e^{-\frac{1}{2}\prth{\frac{\abs{z}}{4l_b}}^2} 
        \end{equation}
       If instead $4 l_b \abs{x} \le \abs{z}$, we have  
       $$\abs{z + \sqrt{2}l_b x} \ge \abs{z} - \sqrt{2}l_b \abs{x} \ge \abs{z}\prth{1 - \dfrac{\sqrt{2}}{4}} \ge \dfrac{\abs{z}}{2}$$
       and thus
        \begin{equation}\label{eq:Linfty 2}
            \dfrac{\abs{x}^n e^{-\frac{\abs{x}^2}{2}}}{1 + \abs{z + \sqrt{2}l_b x}^p} 
                \le \dfrac{\prth{\frac{n}{e}}^{\frac{n}{2}}}{1 + \abs{\frac{z}{2}}^p} 
        \end{equation}
        Putting \eqref{eq:Linfty 1} \eqref{eq:Linfty 2} and \eqref{eq:another ineq} together leads to
        \begin{align*}
            \int_{\R^2} \rho_{\gamma_b} \abs{\psi_{z, n}} \le \dfrac{\prth{1 + \Tr{\gamma_b \abs{X}^p}}}{\sqrt{2\pi n!}l_b} \prth{\prth{\dfrac{\abs{z}}{4l_b}}^n e^{-\frac{1}{2}\prth{\frac{\abs{z}}{4l_b}}^2} + \dfrac{\prth{\frac{n}{e}}^{\frac{n}{2}}}{1 + \abs{\frac{z}{2}}^p}}
        \end{align*}
        Inserting this and ~\eqref{eq:psinz L1} in \eqref{eq:tr gamma pinz0}, combining with Stirling's formula again gives 
        \begin{align*}
            \Tr{\gamma_b \Pi_{z, n}} 
                \le& C \prth{1 + \Tr{\gamma_b \abs{X}^p}} \dfrac{\prth{n+1}^{\frac{1}{4}}}{\sqrt{n!}} \prth{\prth{\dfrac{\abs{z}}{4l_b}}^n e^{-\frac{1}{2}\prth{\frac{\abs{z}}{4l_b}}^2} + \dfrac{\prth{\frac{n}{e}}^{\frac{n}{2}}}{1 + \abs{\frac{z}{2}}^p}} \\
                \le& C \prth{1 + \Tr{\gamma_b \abs{X}^p}} \prth{\dfrac{\prth{n+1}^{\frac{1}{4}}}{\sqrt{n!}}\prth{\dfrac{\abs{z}}{4l_b}}^n e^{-\frac{1}{2}\prth{\frac{\abs{z}}{4l_b}}^2} + \dfrac{1}{1 + \abs{\frac{z}{2}}^p}} 
        \end{align*}
        With the notation of Lemma~\ref{lem:integ results},
        \begin{align*}
            \int_{\abs{z}\ge R} \abs{z} \prth{\dfrac{\abs{z}}{4l_b}}^n e^{-\frac{1}{2}\prth{\frac{\abs{z}}{4l_b}}^2} \dd \bz 
                =& 128\pi l_b^3 \int_{\frac{R}{4l_b}}^\infty r^{n+2}e^{-\frac{1}{2}r^2} 
                = 128 \pi l_b^3 I_{n+2}\prth{\frac{R}{4l_b}}
        \end{align*}
        Since $p>3$, we can integrate
        \begin{align*}
            \int_{\abs{z}\ge R}\dfrac{\abs{z}}{1 + \abs{\frac{z}{2}}^p} \dd \bz  
                \le 2^p \int_{\abs{z}\ge R} \abs{z}^{1-p}\dd \bz 
                = 2^{p+1} \pi \int_R^{\infty} r^{2-p}dr
                = \dfrac{2^{p+1} \pi}{\prth{p-3}R^{p-3}} 
        \end{align*}
        Collecting the above considerations gives
        \begin{align*}
            \int_{\abs{z}\ge R} \abs{z} \Tr{\gamma_b \Pi_{z, n}}\dd \bz  
                \le C(p) \prth{1 + \Tr{\gamma_b \abs{X}^p}}\prth{ \dfrac{\prth{n+1}^{\frac{1}{4}}}{\sqrt{n!}} l_b^3 I_{n+2}\prth{\dfrac{R}{4l_b}} + \dfrac{1}{R^{p-3}}}
        \end{align*}
        that we may combine with~\eqref{eq:rho cvg scd tail start} to get
        \begin{align}\label{eq:rho cvg scd tail 1}
            &\abs{\int_{\mathbb{R}^2}\varphi \eta_R\rho_{\gamma_b}^{sc, \le M}}
                \le \dfrac{\norm{\nabla \varphi}_{L^\infty}}{2\pi l_b^2} \sum_{n=0}^M \int_{\abs{z}\ge R} \abs{z} \Tr{\gamma_b \Pi_{z, n}}\dd \bz  \\
                \le& C(p) \norm{\nabla \varphi}_{L^\infty}\prth{1 + \Tr{\gamma_b \abs{X}^p}} \prth{l_b \sum_{n=0}^M\dfrac{\prth{n+1}^{\frac{1}{4}}}{\sqrt{n!}} I_{n+2}\prth{\dfrac{R}{4l_b}} + \dfrac{M}{l_b^2R^{p-3}}}
        \end{align}
        But $R \ge 4 \sqrt{M}l_b \gg 4l_b$ so $\prth{I_n\prth{\frac{R}{4l_b}}}_{n\in\mathbb{N}}$ is increasing, hence it follows from Lemma~\ref{lem:integ results} that
        \begin{align*}
            \sum_{n=0}^M\dfrac{\prth{n+1}^{\frac{1}{4}}}{\sqrt{n!}} I_{n+2}\prth{\dfrac{R}{4l_b}}
                \le& \sum_{n=0}^{2\floor{\frac{M}{2}}+1}\dfrac{\prth{n+1}^{\frac{1}{4}}}{\sqrt{n!}} I_{n+2}\prth{\dfrac{R}{4l_b}} \\
                =& \sum_{n=0}^{\floor{\frac{M}{2}}}\dfrac{\prth{2n+1}^{\frac{1}{4}}}{\sqrt{(2n)!}} I_{2n+2}\prth{\dfrac{R}{4l_b}} + \sum_{n=0}^{\floor{\frac{M}{2}}}\dfrac{\prth{2n+2}^{\frac{1}{4}}}{\sqrt{(2n+1)!}} I_{2n+3}\prth{\dfrac{R}{4l_b}} \\
                \le& 2 \sum_{n=0}^{\floor{\frac{M}{2}}}\dfrac{\prth{2n+2}^{\frac{1}{4}}}{\sqrt{(2n)!}} I_{2n+3}\prth{\dfrac{R}{4l_b}} \\
                =& 2 e^{-\frac{R^2}{32 l_b^2}}\sum_{n=0}^{\floor{\frac{M}{2}}}\dfrac{\prth{2n+2}^{\frac{1}{4}}}{\sqrt{(2n)!}} 2^{n+1}\prth{n+1}!\sum_{i=0}^{n+1}\dfrac{1}{i!}\prth{\dfrac{R^2}{32 l_b^2}}^{i}
        \end{align*}
        Then using $(2n)! \ge (n!)^2, i! \ge 1$ and $\frac{R^2}{32 l_b^2} \ge \frac{M}{2} \ge 1$,
        \begin{align*}
            &\sum_{n=0}^M\dfrac{\prth{n+1}^{\frac{1}{4}}}{\sqrt{n!}} I_{n+2}\prth{\dfrac{R}{4l_b}} \\
                \le& 2 e^{-\frac{R^2}{32 l_b^2}}\sum_{n=0}^{\floor{\frac{M}{2}}}\prth{2n+2}^{\frac{1}{4}}(n+1)(n+2) 2^{n+1} \prth{\dfrac{R^2}{32l_b^2}}^{n+1} \\
                \le& 2 e^{-\frac{R^2}{32 l_b^2}}\prth{\floor{\frac{M}{2}} + 1}\prth{2\floor{\frac{M}{2}}+2}^{\frac{1}{4}}\prth{\floor{\frac{M}{2}}+1}\prth{\floor{\frac{M}{2}}+2} \prth{\dfrac{R^2}{16l_b^2}}^{\floor{\frac{M}{2}}+1} \\
                \le& C M^{\frac{13}{4}} \prth{\dfrac{R^2}{16l_b^2}}^{\frac{M}{2}+1}e^{-\frac{R^2}{32 l_b^2}}
                =  C M^{\frac{13}{4}} e^{\prth{M+2}\ln\prth{\frac{l_b^{-\prth{1+\beta}}}{16}} - \frac{l_b^{-2(1+\beta)}}{32}}
                \le  C M^{\frac{13}{4}} e^{(1+\beta)\prth{M+2}\ln\prth{l_b^{-1}} - \frac{l_b^{-2(1+\beta)}}{32}}
        \end{align*}
        Since $M \ll l_b^{-2}$ for large enough $b$,
        \begin{align*}
            \sum_{n=0}^M\dfrac{\prth{n+1}^{\frac{1}{4}}}{\sqrt{n!}} I_{n+2}\prth{\dfrac{R}{4l_b}}
                \le  C l_b^{-\frac{13}{2}} e^{l_b^{-2}\ln\prth{l_b^{-1}} - \frac{l_b^{-2(1+\beta)}}{32}}
                = C l_b^{-\frac{13}{2}} e^{- l_b^{-2(1+\beta)}\prth{\frac{1}{32} - l_b^{2\beta}\ln\prth{l_b^{-1}}}}.
        \end{align*}
        Inserting this in \eqref{eq:rho cvg scd tail 1}, we get 
        \begin{align}\label{eq:rho cvg scd tail}
            \abs{\int_{\mathbb{R}^2}\varphi \eta_R\rho_{\gamma_b}^{sc, \le M}}
                \le C(p) \norm{\nabla \varphi}_{L^\infty}\prth{1 + \Tr{\gamma_b \abs{X}^p}} \prth{l_b^{-\frac{11}{2}} e^{- l_b^{-2(1+\beta)}\prth{\frac{1}{32} - l_b^{2\beta}\ln\prth{l_b^{-1}}}} + Ml_b^{\beta(p-3) -2}}
        \end{align}
        
        \medskip

    \noindent\textbf{Step 4, normalization error.}
        We write, using complex notation,
        \begin{align*}
            \int_{\mathbb{R}^2}\varphi\prth{\rho_b - \rho_{\gamma_b}^{sc, \le M}}
                =&\ \ \int_{\mathbb{R}^2} \varphi\rho_b\prth{1 - \Tr{\gamma_b \Pi_{\le M}}}
                = \Tr{\gamma_b \Pi_{>M}}  \int_{\mathbb{R}^2} \int_0^1 \nabla \varphi(\tau \bz)\cdot \bz \rho_b(\bz) \dd \tau  \dd \bz   \\
                =&\ \frac{\Tr{\gamma_b \Pi_{>M}} }{\Tr{\gamma_b \Pi_{\le M}}}  \int_0^1 \int_{\mathbb{R}^2}\nabla \varphi(\tau \bz)\cdot \bz \Tr{\gamma_b \Pi_{z, \le M}} \frac{\dd \bz }{2\pi l_b^2} \dd \tau \\
                =&\ \frac{\Tr{\gamma_b \Pi_{>M}} }{\Tr{\gamma_b \Pi_{\le M}}}  \int_0^1 \text{Re}\int_{\mathbb{R}^2}2\partial_{{ z }}\varphi (\tau \bz) {z} \Tr{\gamma_b \Pi_{z, \le M}} \frac{\dd \bz }{2\pi l_b^2} \dd \tau  \\
                =&\ \frac{\Tr{\gamma_b \Pi_{>M}} }{\Tr{\gamma_b \Pi_{\le M}}}  \int_0^1 \text{Re}\int_{\mathbb{R}^2}2\partial_{\overline{ z }}\varphi (\tau \bz) \Tr{\gamma_b \Pi_{z, \le M} \sqrt{2}l_b c^\dagger} \frac{\dd \bz }{2\pi l_b^2} \dd \tau  
        \end{align*}
        where, in the last step, we used~\eqref{eq:coherent eigen}. 
        
        But~\eqref{eq:complex ops} implies that 
        $$ \sqrt{2}l_b c^\dagger = X - i\sqrt{2}l_b a$$
        and thus  using the Cauchy-Schwarz inequality, 
        \begin{align*}
            \abs{\Tr{\gamma_b \Pi_{z, \le M} \sqrt{2}l_b c^\dagger}}
                \le \Tr{\gamma_b \Pi_{z, \le M}} + \frac{1}{2} \Tr{\gamma_b \overline{X} \Pi_{z, \le M} X} + l_b^2 \Tr{\gamma_b a^\dagger \Pi_{z, \le M}a} \\
                = \prth{1 + (M+1)l_b^2}\Tr{\gamma_b \Pi_{z, \le M+1}} + \frac{1}{2} \Tr{\gamma_b \overline{X} \Pi_{z, \le M} X}
        \end{align*}
        Inserting in the above and recalling that $M l_b^2 \le 1$, we deduce
        \begin{align}
            \abs{\int_{\mathbb{R}^2}\varphi\prth{\rho_b - \rho_{\gamma_b}^{sc, \le M}}}
                \le&\ C\norm{\nabla \varphi}_{L^\infty}\frac{\Tr{\gamma_b \Pi_{>M}} }{\Tr{\gamma_b \Pi_{\le M}}} \prth{\Tr{\gamma_b\Pi_{\le M+1}} + \Tr{\gamma_b \overline{X} \Pi_{\le M} X}} \nonumber \\
                \le& C\prth{1 + \Tr{\gamma_b \abs{X}^2}} \norm{\nabla \varphi}_{L^\infty} \Tr{\gamma_b \Pi_{> M}}\nonumber \\
                \le& C\prth{1 + \Tr{\gamma_b \abs{X}^2}} \norm{\nabla \varphi}_{L^\infty} \Tr{\gamma_b \Pi_{> M}\mathscr{L}_b^k} M^{-k}\label{eq:norm error}
        \end{align}

        \medskip
        
    \noindent\textbf{Step 5, conclusion.} We insert \eqref{eq:rho cvg 2 LLL} \eqref{eq:rho cvg 2 HLL} \eqref{eq:rho cvg d tail} \eqref{eq:rho cvg scd tail} in ~\eqref{eq:rho cgv 2.0}:
        \begin{align*}
            &\abs{\int_{\mathbb{R}^2}\varphi\prth{\rho_{\gamma_b} - \rho_{\gamma_b}^{sc, \le M}}}
                \le C(p)  \norm{\nabla \varphi}_{L^\infty} \left( \sqrt{\Tr{\gamma_b \mathscr{L}_b^k}} l_b^{1-\beta}\sqrt{p_{k-2}(M)} \right.\\
                +&  \sqrt{\Tr{\gamma_b\abs{X}^2}} M^{-\frac{k}{2}} \sqrt{\Tr{\gamma_b\Pi_{>M}\mathscr{L}_b^k}}
                + \Tr{\gamma_b \abs{X}^p} l_b^{\beta(p-1)} \\
                +& \left. \prth{1 + \Tr{\gamma_b \abs{X}^p}} \prth{l_b^{-\frac{11}{2}} e^{- l_b^{-2(1+\beta)}\prth{\frac{1}{32} - l_b^{2\beta}\ln\prth{l_b^{-1}}}} + M l_b^{\beta(p-3) -2}} \right)
        \end{align*}
        We further notice that, since $p>3$, 
        \begin{align*}
                \sqrt{\Tr{\gamma_b\abs{X}^2}} \le \sqrt{1 + \Tr{\gamma_b\abs{X}^p}} \le 1 + \Tr{\gamma_b \abs{X}^p}.
        \end{align*}
        Hence, using the triangular inequality along with~\cref{eq:norm error} leads to 
        \begin{align*}
            &\abs{\int_{\mathbb{R}^2}\varphi\prth{\rho_{\gamma_b} - \rho_{\gamma_b}^{sc, \le M}}}
            \le C(p) \norm{\nabla \varphi}_{L^\infty}  \prth{1 + \Tr{\gamma_b \abs{X}^p}} \left( \Tr{\gamma_b \Pi_{> M} \mathscr{L}_b}M^{-k}\right.\\
                &\left. + \sqrt{\Tr{\gamma_b \mathscr{L}_b^k}} l_b^{1-\beta}   \syst{\matrx{
                        M^{1-\frac{k}{2}}  &\text{ if } k < 2 \\
                        \sqrt{\ln(M)}  &\text{ if } k = 2 \\
                        1 &\text{ if } k > 2
                    }} 
                +  M^{-\frac{k}{2}} \sqrt{\Tr{\gamma_b\Pi_{>M}\mathscr{L}_b^k}}
                +  l_b^{\beta(p-1)}
                +  M l_b^{\beta(p-3) -2} \right)
        \end{align*}
        and we obtain~\eqref{eq:W1withconf} using Kantorovitch-Rubinstein duality~\eqref{eq:MKW}. 
    \end{proof}

\section{Semi-classical dynamics of the densities}\label{sec:sc EDP}

This section contains the proof of Theorem~\ref{th:V of HF}. We first provide the second main ingredient mentioned after the statement, namely the study of the dynamics of the truncated semi-classical density. We next combine this with the bounds of Section~\ref{sec:densities} to conclude the proof.

    \subsection{Dynamics of the truncated semi-classical density}\label{sec:sc EDP 1}
    We now prove that $\rho_{\gamma_b}^{sc, \le M}$, as defined in Section~\ref{sec:densities}, almost satisfies the weak formulation of~\eqref{eq:drift} modulo a suitable choice of $1 \ll M\ll l_b^{-2}$. Here $W$ is a (possibly time-dependent) generic potential that we will replace with $V + w \star \rho_{\gamma_b (t)}$ later on.

    \begin{proposition}[\textbf{Drift equation for the truncated semi-classical density}]\label{prop:vorticity}\mbox{}\\
        Let $t\in \mathbb{R}_+, k \ge 0, l\ge1, \gamma_b(t) \in  \mathcal{L}^1\prth{L^2\prth{\mathbb{R}^2}}, W \in W^{l+1, \infty}\prth{\mathbb{R}^2}$ and assume
        \begin{align*}
            &\Tr{\gamma_b(t)} = 1, 0\le \gamma_b(t) \le 2\pi l_b^2\\
            & \Tr{\gamma_b(t)\mathscr{L}_b^k} < \infty \\
            &\partial_t \gamma_b(t) = \frac{1}{i l_b^2} \sbra{\mathscr{L}_b + W, \gamma_b(t)}.
        \end{align*}
        Then, $\forall \varphi \in L^1\cap W^{1, \infty}\prth{\R^2}$,
        \begin{align*}
            &\abs{\int_{\mathbb{R}^2}\varphi( \bz )\prth{\partial_t \rho_{\gamma_b}^{sc, \le M}(t, \bz) + \nabla^\perp W(\bz ) \cdot \nabla \rho_{\gamma_b}^{sc, \le M}(t, \bz)} \dd \bz } \\
                \le& C(l) \norm{\varphi}_{L^1 \cap W^{1, \infty}} \norm{W}_{W^{l+1, \infty}} \left(\sqrt{\Tr{\gamma_b \mathscr{L}_b^k}} l_b^{l-1} \syst{\matrx{
                        M^{1 + \frac{l-k}{2}} &\text{ if }   k < l+2   \\
                        \sqrt{\ln(M)} &\text{ if }  k= l+2\\
                        1 &\text{ if } k > l+2 \\
                    }}\right. \\
                    +& \dfrac{\Tr{\gamma_b\mathscr{L}_b^k \Pi_{M-l+1:M+l}}}{l_b M^{k - \frac{1}{2}}}
                    + \left. \Tr{\gamma_b\mathscr{L}_b^k} \sum_{p=2}^l l_b^{p-1} \syst{\matrx{
                        M^{\frac{p+1}{2} - k} &\text{ if }   k < \frac{p+1}{2}   \\
                        \ln(M) &\text{ if }  k= \frac{p+1}{2}\\
                        1 &\text{ if } k > \frac{p+1}{2} \\
                    }} \right)
        \end{align*}
    \end{proposition}
    
    The proof will proceed from a Taylor expansion of the potential $W$ at order $l$. Several bounds will be proved by induction on $l$ and are not particularly more complicated to write in the above generality. However our choice of $l$ will be set by the a priori bound available to us, as we explain first:
    
    \begin{remark}[\textbf{Choice of the expansion parameters}]\label{rem:choices}\mbox{}\\
        In the proof of Theorem~\ref{th:V of HF} we take $k=1$ to be able to use Lemma~\ref{lem:E conserv}. Then Proposition~\ref{prop:vorticity} gives
        \begin{align}\label{eq:dynamic rewritten}
            &\abs{\int_{\mathbb{R}^2}\varphi \prth{\partial_t \rho_{\gamma_b}^{sc, \le M}(t) + \nabla^\perp W \cdot \nabla \rho_{\gamma_b}^{sc, \le M}(t)}}
                \le C(l) \norm{\varphi}_{L^1 \cap W^{1, \infty}} \norm{W}_{W^{l+1, \infty}} \nonumber\\
                &\prth{
                    \sqrt{\Tr{\gamma_b(t) \mathscr{L}_b}}l_b^{l-1}M^{\frac{l+1}{2}}
                    + \dfrac{\Tr{\gamma_b(t)\mathscr{L}_b \Pi_{M-l+1:M+l}}}{l_b \sqrt{M}} 
                    + \Tr{\gamma_b(t)\mathscr{L}_b} l_b \sqrt{M}
                } 
        \end{align}
        The first error above is of order $l_b^{l-1}M^{\frac{l+1}{2}}$, the third is of order $l_b \sqrt{M}$. In Lemma~\ref{lem:N choice} we will see that $M\gg 1$ may be chosen so that the second error weighs $l_b^{-1} M^{-\frac{3}{2}}$. We are able to control all the errors if 
        \begin{align*}
             \prth{\dfrac{1}{l_b}}^{\frac{2}{3}} \ll M \ll \prth{\dfrac{1}{l_b}}^{2\frac{l-1}{l+1}}
        \end{align*}
        which is possible if and only if $l>2$. Hence the above will be applied to obtain Theorem~\ref{th:V of HF} with the choices $k=1,l=3$. This will yield error terms 
        $$ l_b ^2 M^2 +  l_b^{-1} M^{-\frac{3}{2}} +  l_b M^{1/2}. $$
        Since the second term must be $o(1)$, the first one dominates the third, and optimization in $M$ leads to a choice $M\sim l_b ^{-6/7}$ and a final error of order $O(l_b^{2/7})$.
        
        If more moments of the kinetic energy are bounded for positive times, we may use the above with a larger value of $k$ and hence get an efficient estimate for lower $l$, i.e. asking for less regularity of the external and interaction potentials.
        
        \hfill$\diamond$
    \end{remark}
    
    The first step in our proof consists in deriving an explicit equation satisfied by $\rho_{\gamma_b}^{sc, \le M}$. Most bounds are then obtained by writing quantum expectations (traces) using integrals of operator kernels. A general term will be an integral in $\bx, \by, \bz \in \R^2$ where $\bx,\by$ are the arguments of the operator kernel and $\bz$ the guiding-center coordinate of the coherent state at hand. As per the observations of Section~\ref{sec:coherent} such integrals are concentrated around $\bx\sim \by \sim \bz$. Principal terms corresponding to Equation~\eqref{eq:drift} are easily identified by replacing $W(\bx)-W(\by)$ (coming from the commutator with $W$) by $(\bx- \by) \cdot \nabla W (\bz)$ and then using~\eqref{eq:d1 pi le N z}. Some care in controlling the expansion of $W$ is needed:
    
    \medskip
    
     \noindent$\bullet$ Using~\eqref{eq:gradperp Pi 1} (or more precisely, the first term in~\eqref{eq:d1 pi le N Z bis}) a factor of $(\bx-\by)$ can make us gain\footnote{Note that this aspect is algebraic in nature: it is not clear that a factor of $|x-y|$ would gain us a factor of $l_b^2$.} a factor of $l_b^2$, at the price of integrations by parts (that we can perform using the regularity of the test function). 
     
     \medskip
    
     \noindent$\bullet$ A factor of $|x-z|$ or $|y-z|$ makes us gain at best a factor of $l_b \sqrt{n}$, the localization length of the coherent state wave-function in the $n$-th Landau level. Cf the discussion below Proposition~\ref{prop:rho sc cvg}.
     
     \medskip
    
     \noindent$\bullet$ There is a truncation error due to the second term in~\eqref{eq:gradperp Pi 1}, calling for some optimization in $M$.
     
     \medskip

    In~\eqref{eq:eps0} below the remainder is in a form allowing to leverage the first observation above. Indeed, the commutator naturally brings factors of $(\bx-\by)$ (think of the commutator with the position operator~$\bX $). The remainders in~\eqref{eq:epsI} are estimated using the second observation, which we formalize in Lemma~\ref{lem:generic term} below. Finally~\eqref{eq:epsII} is the truncation error, whose control will demand a proper choice of $M$ later in the proof, see Lemma~\ref{lem:N choice}.
    
    \begin{lemma}[\textbf{Equation for the semi-classical density and potential expansion}]\label{lem:pot expansion}\mbox{}\\
        Let $t\in \mathbb{R}_+, l\in \mathbb{N}, \gamma_b(t) \in  \mathcal{L}^1\prth{L^2\prth{\mathbb{R}^2}}, W \in W^{l, \infty}\prth{\mathbb{R}^2}$ and assume
        \begin{equation} \label{eq:Hartree potential exp}
            \partial_t \gamma_b(t) = \frac{1}{i l_b^2} \sbra{\mathscr{L}_b + W, \gamma_b(t)}.
        \end{equation}
        Denote $d^p W_{\bz}$  the $p^{th}$ differential of $W$ at $\bz\in \R^2$, meaning that $d^p W_{\bz}$ is a $p$-linear form on $\R^2$.
        
        Then, with $\rho_{\gamma_b}^{sc, \le M}$ as in Definition~\ref{def:Husimi},
        \begin{equation}\label{eq:error decomp}
            \partial_t \rho_{\gamma_b}^{sc, \le M} + \nabla^\perp W \cdot \nabla \rho_{\gamma_b}^{sc, \le M}
                = \mathcal{E}_0^l + \sum_{p=2}^l\sum_{q=1}^p \mathcal{E}_{\mathrm{I}}^{p, q} + \sum_{p=1}^l\sum_{q=1}^p \mathcal{E}_{\mathrm{II}}^{p, q} 
        \end{equation}
        where
        \begin{align}
        &\mathcal{V}_{z, l}(\bx) \coloneq W(\bx) - \sum_{p=0}^l \dfrac{1}{p!} d^p W_{\bz} \prth{\bx- \bz }^{\otimes p} \label{eq:expansion} \\
         &\mathcal{E}_0^l( \bz ) \coloneq \dfrac{1}{2i\pi l_b^4}\Tr{\gamma_b\sbra{\Pi_{z, \le M}, \mathcal{V}_{z, l}}}\label{eq:eps0}
        \end{align}
and
        \begin{align}
            & \mathcal{E}_{\mathrm{I}}^{p, q}\left( \bz \right) \coloneq \dfrac{1}{2\pi l_b^2 p!}  \Tr{\gamma_b d^p W_{\bz} \prth{\bX- \bz }^{\otimes(q-1)}\otimes \prth{\nabla_\bz ^\perp \Pi_{z, \le M}}\otimes\prth{\bX- \bz }^{\otimes (p-q)}} \label{eq:epsI}\\
            &\mathcal{E}_{\mathrm{II}}^{p, q}(\bz ) \coloneq \dfrac{\sqrt{M+1}}{2\sqrt{2}\pi l_b^3 p!} \nonumber \\
            &\Tr{\gamma_b d^p W_{\bz} \prth{\bX- \bz }^{\otimes(q-1)}\otimes \prth{\matrx{1 & 1 \\ -i & i}} \prth{\matrx{ \ket{\psi_{z, M}} \bra{\psi_{z, M+1}}\\ \ket{\psi_{z, M+1}} \bra{\psi_{z, M}}}} \otimes\prth{\bX- \bz }^{\otimes (p-q)}} \label{eq:epsII}
        \end{align}
    \end{lemma}

    \begin{proof}
    \noindent\textbf{Step 1 : Direct computation.} We start from \eqref{eq:LowLL sc density} and Equation~\eqref{eq:Hartree potential exp}:
        \begin{align}\label{eq:dym part}
            \partial_t \rho_{\gamma_b}^{sc, \le M}(\bz )
                =& \dfrac{1}{2\pi l_b^2} \Tr{\partial_t \gamma_b \Pi_{z, \le M}}
                = \dfrac{1}{2i\pi l_b^4} \Tr{\sbra{\mathscr{L}_b + W, \gamma_b} \Pi_{z, \le M}} \nonumber \\
                =& \dfrac{1}{2i\pi l_b^4} \Tr{ \gamma_b \sbra{\Pi_{z, \le M}, \mathscr{L}_b + W} }
                = \dfrac{1}{2i\pi l_b^4} \Tr{\gamma_b \sbra{\Pi_{z, \le M}, W}} 
        \end{align}
        On the other hand, using \eqref{eq:d1 pi le N z},
        \begin{align}\label{eq:pot part}
            \nabla^\perp W(\bz ) \cdot \nabla_\bz \rho_{\gamma_b}^{sc, \le M}(\bz )
                &= - \nabla W(\bz ) \cdot \nabla_\bz ^\perp \rho_{\gamma_b}^{sc, \le M} (\bz) 
                = \dfrac{-1}{2\pi l_b^2} \nabla W(\bz ) \cdot \Tr{\gamma_b \nabla^\perp_{\bz}  \Pi_{z, \le M}} \nonumber\\
                &= \dfrac{-1}{2i\pi l_b^4} \Tr{\gamma_b\sbra{\Pi_{z, \le M}, \bX \cdot\nabla W(\bz )}}\nonumber\\
                 &+ \dfrac{\sqrt{M+1}}{2\sqrt{2}\pi l_b^3} \nabla W(\bz ) \cdot \prth{\matrx{1 & 1 \\ -i & i}} \prth{\matrx{\Tr{\gamma_b\ket{\psi_{z, M}} \bra{\psi_{z, M+1}}} \\ \Tr{\gamma_b\ket{\psi_{z, M+1}} \bra{\psi_{z, M}}}}} 
        \end{align}
        Putting together Equation~\eqref{eq:dym part} and Equation~\eqref{eq:pot part} yields
        \begin{align}\label{eq:edp o(1) 1}
            &\partial_t \rho_{\gamma_b}^{sc, \le M}(\bz ) + \nabla^\perp W(\bz ) \cdot \nabla_\bz \rho_{\gamma_b}^{sc, \le M} (\bz)  
                = \dfrac{1}{2i\pi l_b^4} \Tr{\gamma_b \sbra{\Pi_{z, \le M}, W - \bX \cdot \nabla W(\bz )}} \nonumber \\
                    +& \dfrac{\sqrt{M+1}}{2\sqrt{2}\pi l_b^3} \nabla W(\bz ) \cdot \prth{\matrx{1 & 1 \\ -i & i}} \prth{\matrx{\Tr{\gamma_b\ket{\psi_{z, n}} \bra{\psi_{z, M+1}}} \\ \Tr{\gamma_b\ket{\psi_{z, M+1}} \bra{\psi_{z, n}}}}} 
        \end{align}

        \medskip
        
     \noindent\textbf{Step 2 : Taylor expansion.} We compute
        \begin{align*}
            \mathcal{V}_{z, l}(\by) - \mathcal{V}_{z, l}(\bx) 
                =&\ W(\by) - W(\bx) - \sum_{p=0}^l \dfrac{1}{p!} d^p W_{\bz} \prth{\prth{\by - \bz }^{\otimes p} - \prth{\bx- \bz }^{\otimes p}} \\
                =&\ W(\by) - W(\bx) - (\by - \bx)\cdot\nabla W(\bz ) - \sum_{p=2}^l \dfrac{1}{p!} d^p W_{\bz} \prth{\prth{\by - \bz }^{\otimes p} - \prth{\bx- \bz }^{\otimes p}} 
        \end{align*}
        and notice the telescopic expression
        \begin{align*}
        \prth{\by - \bz }^{\otimes p} - \prth{\bx - \bz }^{\otimes p} &= \sum_{q=0}^{p-1} \prth{\bx- \bz }^{\otimes q} \otimes\prth{\by - \bz }^{\otimes (p-q)} - \sum_{q=1}^{p} \prth{\bx- \bz }^{\otimes q}\otimes\prth{\by - \bz }^{\otimes (p-q)} \\
        &= \sum_{q=1}^p \prth{\bx- \bz }^{\otimes(q-1)} \otimes\prth{\by - \bz }^{\otimes (p+1-q)} - \sum_{q=1}^p \prth{\bx- \bz }^{\otimes q}\otimes\prth{\by - \bz }^{\otimes (p-q)}\\ 
        &= \sum_{q=1}^p \prth{\bx- \bz }^{\otimes(q-1)}\otimes(\by - \bz - (\bx-\bz))\otimes\prth{\by - \bz }^{\otimes (p-q)} \\
            &= \sum_{q=1}^p \prth{\bx- \bz }^{\otimes(q-1)}\otimes(\by - \bx)\otimes\prth{\by - \bz }^{\otimes (p-q)} 
        \end{align*}
       Combining the above we express the integral kernel
        \begin{align*}
            &\sbra{\Pi_{z, \le M},W - \bX  \cdot \nabla W(\bz )}(\bx,\by) = \Pi_{z, \le M}(\bx,\by)\prth{W(\by) - W(\bx) - (\by -\bx ) \cdot \nabla W(\bz )} \\
                =& \Pi_{z, \le M}(\bx,\by)\prth{\mathcal{V}_{z, l}(\by) - \mathcal{V}_{z, l}(\bx) + \sum_{p=2}^l \dfrac{1}{p!} d^p W_{\bz} \prth{\prth{\by - \bz }^{\otimes p} - \prth{\bx- \bz }^{\otimes p}}} \\
                =& \Pi_{z, \le M}(\bx,\by)\prth{\mathcal{V}_{z, l}(\by) - \mathcal{V}_{z, l}(\bx) + \sum_{p=2}^l \sum_{q=1}^p \dfrac{1}{p!} d^p W_{\bz} \prth{\bx- \bz }^{\otimes(q-1)}\otimes(\by - \bx)\otimes\prth{\by - \bz }^{\otimes (p-q)}}
        \end{align*}
        Inserting \eqref{eq:d1 pi le N z} we obtain the operator equality
        \begin{align}\label{eq:W expand com}
            &\sbra{\Pi_{z, \le M},W - \bX  \cdot \nabla W(\bz )} \nonumber \\
                =& \sbra{\Pi_{z, \le M}, \mathcal{V}_{z, l}} 
                    + \sum_{p=2}^l \sum_{q=1}^p \dfrac{1}{p!} d^p W_{\bz} \prth{\bX- \bz }^{\otimes(q-1)}\otimes\sbra{\Pi_{z, \le M}, \bX}\otimes\prth{\bX- \bz }^{\otimes (p-q)} \nonumber \\
                =& \sbra{\Pi_{z, \le M}, \mathcal{V}_{z, l}} 
                    + il_b^2 \sum_{p=2}^l \sum_{q=1}^p \dfrac{1}{p!} d^p W_{\bz} \prth{\bX- \bz }^{\otimes(q-1)}\otimes\prth{\nabla_\bz ^\perp \Pi_{z, \le M}}\otimes\prth{\bX- \bz }^{\otimes (p-q)} \nonumber \\
                    +& il_b \sqrt{\dfrac{M+1}{2}}
                    \sum_{p=2}^l \sum_{q=1}^p \dfrac{1}{p!} d^p W_{\bz} \prth{\bX- \bz }^{\otimes(q-1)}\otimes \prth{\matrx{1 & 1 \\ -i & i}} \prth{\matrx{ \ket{\psi_{z, M}} \bra{\psi_{z, M+1}}\\ \ket{\psi_{z, M+1}} \bra{\psi_{z, M}}}} \otimes\prth{\bX- \bz }^{\otimes (p-q)} 
        \end{align}
        
        \medskip
        
    \noindent\textbf{Step 3 : Conclusion.}
        Inserting~\eqref{eq:W expand com} in Equation~\eqref{eq:edp o(1) 1} we find
        \begin{align*}
            &\partial_t \rho_{\gamma_b}^{sc, \le M} + \nabla^\perp W \cdot \nabla \rho_{\gamma_b}^{sc, \le M}
                = \dfrac{1}{2i\pi l_b^4}\Tr{\gamma_b\sbra{\Pi_{z, \le M}, \mathcal{V}_{z, l}}} \\
                    &+ \sum_{p=2}^l\sum_{q=1}^p \dfrac{1}{2\pi l_b^2 p!} \Tr{\gamma_b d^p W_{\bz} \prth{\bX- \bz }^{\otimes(q-1)}\otimes \prth{\nabla_\bz ^\perp\Pi_{z, \le M}}\otimes\prth{\bX- \bz }^{\otimes (p-q)}} \\
                    &+ \sum_{p=1}^l\sum_{q=1}^p \dfrac{\sqrt{M+1}}{2\sqrt{2}\pi l_b^3 p!} \\
                    &\cdot \Tr{\gamma_b d^p W_{\bz} \prth{\bX- \bz }^{\otimes(q-1)}\otimes \prth{\matrx{1 & 1 \\ -i & i}} \prth{\matrx{ \ket{\psi_{z, M}} \bra{\psi_{z, M+1}}\\ \ket{\psi_{z, M+1}} \bra{\psi_{z, M}}}} \otimes\prth{\bX- \bz }^{\otimes (p-q)}} 
        \end{align*}
        Note that the high Landau level error term in Equation~\eqref{eq:edp o(1) 1} is exactly the term indexed by $p=1$ in the second sum above. Finally, defining the errors terms~\eqref{eq:eps0}, ~\eqref{eq:epsI}, \eqref{eq:epsII}, we get the decomposition Equation~\eqref{eq:error decomp}.
    \end{proof}

    In the proof of Proposition~\ref{prop:vorticity} we will use repeatedly the following bound to dispatch the error terms in~\eqref{eq:epsI} and~\eqref{eq:epsII}. 
    
    \begin{lemma}[\textbf{A general error term}]\label{lem:generic term}\mbox{}\\
    Let $\gamma_b$ be a non negative trace-class operator, $\bX = (X_1,X_2)$ the position operator, $r,M\in \N$, $\bz= (z_1,z_2)\in \R^2$ and $\Pi_{z,M}$ the coherent state projector~\eqref{eq:Pizn}.
    Define
    \begin{equation}\label{eq:generic term}
                \mathcal{E}_{r, M}(\bz) \coloneq 
                    \sum_{i_{1} \ldots i_r \in \sett{1, 2}^r} \Tr{\gamma_b \prod_{j=1}^r (X_{i_j} - z_{i_j}) \, \Pi_{z, M} \prod_{j=1}^r (X_{i_j} - z_{i_j})}.
            \end{equation}
     Then, we have the bound 
     \begin{equation}\label{eq:to use in esp II}
            \int_{\mathbb{R}^2} \abs{\mathcal{E}_{r, M}(\bz)} \frac{\dd \bz }{2\pi l_b^2}
                \le C(r) l_b^{2r} (M+1)^{r-k}\Tr{\gamma_b \mathscr{L}_b^k \Pi_{M-r:M+r}}
        \end{equation}
        with $\Pi_{M-r:M+r} $ as in Definition~\ref{def:coherentstates}.
    \end{lemma}
    \begin{proof}
        Following the proof of the binomial theorem, we get by induction on $r$
        \begin{align}
            &\sum_{i_{1} \ldots i_r \in \sett{1, 2}^r} \prod_{j=1}^r (X_{i_j} - z_{i_j}) \, \Pi_{z, M} \prod_{j=1}^r (X_{i_j} - z_{i_j}) \nonumber\\
                &= 2^{-r}\sum_{s=0}^r \prth{\matrx{r \\ s}}(X -  z )^s\overline{(X- z )}^{r-s} \Pi_{z, M} (X- z )^{r-s} \overline{(X -  z )}^s. \label{eq:no-com binom}
        \end{align}
        Indeed, for $r=0$, \cref{eq:no-com binom} reduces to $\Pi_{z, M} = \Pi_{z, M}$. Then, noticing that 
        \begin{align*}
            \sum_{i\in \sett{1, 2}}(X_i - z_i) \sbra{\dots} (X_i - z_i)
                = \frac{(X- z ) \sbra{\dots} \overline{(X- z )} + \overline{(X- z )} \sbra{\dots} (X- z )}{2}
        \end{align*}
        and using \cref{eq:no-com binom}, the induction argument is obtained through
        \begin{align*}
            &\sum_{i_{1} \ldots i_{r+1} \in \sett{1, 2}^{r+1}} \prod_{j=1}^{r+1} (X_{i_j} - z_{i_j}) \, \Pi_{z, M} \prod_{j=1}^{r+1} (X_{i_j} - z_{i_j}) \\
                =&\ 2^{-(r+1)}\sum_{s=0}^r \prth{\matrx{r \\ s}}(X -  z )^{s+1}\overline{(X- z )}^{r-s} \Pi_{z, M} (X- z )^{r-s} \overline{(X -  z )}^{s+1} \\
                    &+  2^{-(r+1)}\sum_{s=0}^r \prth{\matrx{r \\ s}}(X -  z )^s\overline{(X- z )}^{r+1-s} \Pi_{z, M} (X- z )^{r+1-s} \overline{(X -  z )}^s \\
                =&\ 2^{-(r+1)}\sum_{s=0}^{r+1} \prth{\matrx{r+1 \\ s}}(X -  z )^s\overline{(X- z )}^{r+1-s} \Pi_{z, M} (X- z )^{r+1-s} \overline{(X -  z )}^s
        \end{align*}
        with a change of variable $s + 1 \mapsto s$ in the first sum. But, in view of \cref{eq:complex ops},
        \begin{align*}
            X -  z  = \sqrt{2}l_b\prth{c^\dagger - \frac{ z }{\sqrt{2}l_b} + i a }, \quad
            \overline{X -  z } = \sqrt{2}l_b\prth{c  - \frac{\overline{ z }}{\sqrt{2}l_b} - i a^\dagger}.
        \end{align*}
       Combining this with \cref{eq:no-com binom}, \cref{eq:coherent eigen}, ~\eqref{eq:commutators} and~\eqref{eq:ladder coh} we have 
        \begin{align}
            &\sum_{i_{1} \ldots i_r \in \sett{1, 2}^r} \prod_{j=1}^r (X_{i_j} - z_{i_j}) \, \Pi_{z, M} \prod_{j=1}^r (X_{i_j} - z_{i_j}) \nonumber \\
                &= l_b^{2r} \sum_{s=0}^r \prth{\matrx{r \\ s}}\prth{c^\dagger - \frac{ z }{\sqrt{2}l_b} + i a }^s (-ia^\dagger)^{r-s} \Pi_{z, M} (ia)^{r-s} \prth{c  - \frac{\overline{ z }}{\sqrt{2}l_b} - i a^\dagger}^s \nonumber\\
                &= l_b^{2r} \sum_{s=0}^r \prth{\matrx{r \\ s}}\frac{(M+r-s)!}{M!}\prth{c^\dagger - \frac{ z }{\sqrt{2}l_b} + i a }^s \Pi_{z, M+r-s} \prth{c  - \frac{\overline{ z }}{\sqrt{2}l_b} - i a^\dagger}^s \label{eq:pos operator}
        \end{align}
        Then, by the Cauchy-Schwarz inequality,
        \begin{align}
            &\prth{c^\dagger - \frac{ z }{\sqrt{2}l_b} + i a }^s \Pi_{z, M+r-s} \prth{c  - \frac{\overline{ z }}{\sqrt{2}l_b} - i a^\dagger}^s \nonumber\\
                =& \sum_{n, m=0}^s \prth{\matrx{s \\ n}}\prth{\matrx{s \\ m}} \prth{c^\dagger - \frac{ z }{\sqrt{2}l_b}}^n (i a)^{s-n} \Pi_{z, M+r-s}(- i a^\dagger)^{s-m} \prth{c  - \frac{\overline{ z }}{\sqrt{2}l_b}}^m \nonumber\\
                \le& 2^s \sum_{n=0}^s \prth{\matrx{s \\ n}} \prth{c^\dagger - \frac{ z }{\sqrt{2}l_b}}^n (i a)^{s-n} \Pi_{z, M+r-s}(-i a^\dagger)^{s-n} \prth{c - \frac{\overline{ z }}{\sqrt{2}l_b}}^n \nonumber \\
                =& 2^s \sum_{n=0}^s \frac{(M+r-s)!}{(M+r+n-2s)!} \prth{c^\dagger - \frac{ z }{\sqrt{2}l_b}}^n \Pi_{z, M+r+n-2s} \prth{c - \frac{\overline{ z }}{\sqrt{2}l_b}}^n \label{eq:another CS}
        \end{align}
        As regards the combinatorial factors we have that, $\forall s\in \intint{0,r}, n\in\intint{0,s}$,
        \begin{align}
            \prth{\matrx{r \\ s}}\frac{(M+r-s)!}{M!} 2^s \frac{(M+r-s)!}{(M+r+n-2s)!} 
                &\le 2^r \ (M+r-s)^{r-s} \ 2^r \ (M+r-s)^{s-n}
                \nonumber \\
                &\le 4^r (M+r)^{r-n}
                \le 4^r (1+r)^{r-n}(M+1)^{r-n} \nonumber\\
                &\le 4^r (1+r)^r(M+1)^{r-n}.\label{eq:c estimates}
        \end{align}
        Hence, inserting \cref{eq:another CS} and \cref{eq:c estimates} in \cref{eq:pos operator}, we get
        \begin{align}
            &\sum_{i_{1} \ldots i_r \in \sett{1, 2}^r} \prod_{j=1}^r (X_{i_j} - z_{i_j}) \, \Pi_{z, M} \prod_{j=1}^r (X_{i_j} - z_{i_j}) \nonumber\\
                &\le C(r) l_b^{2r} \sum_{s=0}^r \sum_{n=0}^s (M+1)^{r-n}\prth{c^\dagger - \frac{ z }{\sqrt{2}l_b}}^n \Pi_{z, M+r+n-2s} \prth{c - \frac{\overline{ z }}{\sqrt{2}l_b}}^n \label{eq:eps proof}
        \end{align}
        Expending the powers using \cref{eq:coherent eigen}, $\forall j\in\mathbb{N}$,
        \begin{align*}
            0 \le \prth{c^\dagger - \frac{ z }{\sqrt{2}l_b}}^n \Pi_{z, j} \prth{c - \frac{\overline{ z }}{\sqrt{2}l_b}}^n
                =  \sum_{p, q = 0}^n \prth{\matrx{n \\p} }\prth{\matrx{n \\q}}{c^\dagger}^{n-p}(-c)^q \Pi_{z, j} {(-c^\dagger)}^p c^{n-q}
        \end{align*}
        Using $\sbra{\Pi_j, c} = \sbra{\Pi_j,c^\dagger} = 0$ together with ~\eqref{eq:res ID dz} and the following normal ordering formula (cf. for example~\cite[Lemma~4.7]{Rougerie-LMU})
        \begin{align*}
            c^q{c^\dagger}^p = \sum_{i=0}^{\min(p, q)} i!\prth{\matrx{p \\ i}}\prth{\matrx{q \\ i}}{c^\dagger}^{p-i} c^{q-i}
        \end{align*}
        we obtain
        \begin{align*}
            &\int_{\mathbb{R}^2} \abs{\Tr{\gamma_b \prth{c^\dagger - \frac{ z }{\sqrt{2}l_b}}^n \Pi_{z, j} \prth{c - \frac{\overline{ z }}{\sqrt{2}l_b}}^n}} \dd \bz  \\
                =& 2\pi l_b^2 \sum_{p, q = 0}^n (-1)^{p+q} \prth{\matrx{n \\p}}\prth{\matrx{n \\q}}  \Tr{\gamma_b  \Pi_j {c^\dagger}^{n-p}c^q{c^\dagger}^p c^{n-q}} \\
                =& 2\pi l_b^2 \sum_{i=0}^n\sum_{p, q = i}^n (-1)^{p+q} \prth{\matrx{n \\p}}\prth{\matrx{n \\q}} i!\prth{\matrx{p \\ i}}\prth{\matrx{q \\ i}}  \Tr{\gamma_b  \Pi_j {c^\dagger}^{n-i} c^{n-i}} \\
                =& 2\pi l_b^2 \sum_{i=0}^n i! \prth{\sum_{p=i}^n \prth{\matrx{n \\p}} \prth{\matrx{p \\i}} (-1)^p}^2  \Tr{\gamma_b  \Pi_j {c^\dagger}^{n-i} c^{n-i}}
                = 2\pi n! l_b^2 \Tr{\gamma_b  \Pi_j}
        \end{align*}
        noticing that 
        \begin{align*}
            \sum_{p=i}^n \prth{\matrx{n \\p}} \prth{\matrx{p \\i}} (-1)^p 
                = \sum_{p=i}^n \prth{\matrx{n - i \\p - i}} \prth{\matrx{n \\i}} (-1)^p 
                = (-1)^i\prth{\matrx{n \\i}}\sum_{p=0}^{n-i}  \prth{\matrx{n - i\\p}} (-1)^p 
                = \delta_{n, i}(-1)^{n}
        \end{align*}
        There remains to insert this in \cref{eq:eps proof} and recall \cref{eq:power to L}:
        \begin{align*}
            \int_{\mathbb{R}^2} \abs{ \mathcal{E}_{r, M}(\bz)} \dd \bz 
                &\le C(r) l_b^{2(r+1)} \sum_{s=0}^r \sum_{n=0}^s (M+1)^{r-n} n!\Tr{\gamma_b \Pi_{M+r+n-2s}} \\
                &\le C(r)l_b^{2(r+1)}(M+1)^r\Tr{\gamma_b \Pi_{M-r:M+r}} \\
                &\le C(r) l_b^{2(r+1)} (M+1)^{r-k}\Tr{\gamma_b \mathscr{L}_b^k \Pi_{M-r:M+r}}
        \end{align*}
    \end{proof}
    
    We next note a technical bound

    \begin{lemma}[\textbf{Some integrals}]\label{lem:equiv gamma}\mbox{}\\
        Let $\alpha \in \mathbb{R}_+$, then
        \begin{align*}
            I_n(\alpha) 
                \coloneq \dfrac{1}{\pi n!} \int_{\mathbb{R}^2} \abs{x}^{\alpha + 2n} e^{-\abs{x}^2} \dd \bx \eqvl\displaylimits_{n\to\infty}& n^{\frac{\alpha}{2}}
            \end{align*}
        and $\exists C > 0$ such that
        \begin{align*}
            \forall n \in \mathbb{N}, I_n(\alpha) \le C (n + 1)^{\frac{\alpha}{2}}
        \end{align*}
    \end{lemma}
    \begin{proof}
        With polar coordinates and the change of variable $u^2 \mapsto u$
        \begin{align*}
            \dfrac{1}{\pi}\int_{\mathbb{R}^2} \abs{x}^\alpha e^{-\abs{x}^2} \dd \bx   
                = 2 \int_{\mathbb{R}_+} u^{\alpha + 1} e^{-u^2} \dd u   
                = \int_{\mathbb{R}_+} u^{\frac{\alpha}{2}} e^{-u} du
                = \Gamma\prth{\dfrac{\alpha}{2} + 1}
        \end{align*}
        where $\Gamma$ is the Euler Gamma function,
        \begin{align*}
            \Gamma(z) \coloneq\int_{\mathbb{R}_+} t^{z - 1} e^{-t} \dd t
        \end{align*}
        We have the following equivalent for the Euler Gamma function (as a direct consequence of the Stirling formula)
        \begin{align*}
            \dfrac{\Gamma(n+x)}{\Gamma(n)} \eqvl\displaylimits_{n\to\infty} n^x. 
        \end{align*}
        Thus
        \begin{align*}
            I_n(\alpha) 
                =& \dfrac{1}{n!}\Gamma\prth{n + \dfrac{\alpha}{2} + 1}  
                = \dfrac{\Gamma\prth{n + \dfrac{\alpha}{2} + 1}}{\Gamma(n+1)}
        \end{align*}
        and then
        \begin{align*}
            I_n(\alpha) \eqvl\displaylimits_{n\to\infty} (n + 1)^{\frac{\alpha}{2}} \eqvl\displaylimits_{n\to\infty} n^{\frac{\alpha}{2}}.
        \end{align*}         
    \end{proof}
    
    We may now turn to the

    \begin{proof}[Proof of Proposition~\ref{prop:vorticity}]
    
    We start from~\eqref{eq:error decomp} and estimate the right-hand side term by term.
    
    \medskip
    
    \noindent\textbf{Step 1: Estimate of $\mathcal{E}_0^l$.} We introduce another projector:
        \begin{align*}
            \Tr{ \gamma_b \sbra{\Pi_{z, \le M}, \mathcal{V}_{z, l}}}
                &= \sum_{n=0}^M\Tr{ \gamma_b \sbra{\Pi_{z, n}, \mathcal{V}_{z, l}}} \\
                &= \sum_{n=0}^M \prth{ \Tr{\gamma_b \Pi_{z, n} \sbra{\Pi_{z, n}, \mathcal{V}_{z, l}}} + \Tr{\Pi_{z, n} \gamma_b \sbra{\Pi_{z, n}, \mathcal{V}_{z, l}}}}
        \end{align*}
        so with ~\eqref{eq:moment lemma 2} applied to $A _n\left( \bz \right) \coloneq \sbra{\Pi_{z, n}, \mathcal{V}_{z, l}}$, recalling \cref{eq:eps0},
        \begin{align}
            \abs{\int_{\mathbb{R}^2} \varphi( \bz )\mathcal{E}_0^l( \bz ) \dd \bz }
                =& \abs{\dfrac{1}{4i\pi l_b^4}\sum_{n=0}^M\int_{\mathbb{R}^2} \varphi( \bz ) \prth{ \Tr{\gamma_b \Pi_{z, n} \sbra{\Pi_{z, n}, \mathcal{V}_{z, l}}} + \Tr{\Pi_{z, n} \gamma_b \sbra{\Pi_{z, n}, \mathcal{V}_{z, l}}}}\dd \bz } \nonumber\\
                \le& \dfrac{2}{l_b^2} \abs{\dfrac{1}{2\pi l_b^2}\sum_{n=0}^M\int_{\mathbb{R}^2} \abs{ \varphi( \bz ) \Tr{\gamma_b \Pi_{z, n} \sbra{\Pi_{z, n}, \mathcal{V}_{z, l}}}}\dd \bz } \nonumber\\
                \le& \dfrac{\sqrt{\norm{\varphi}_{L^1}\norm{\varphi}_{L^\infty}}}{l_b^2} \sqrt{\Tr{\gamma_b \mathscr{L}_b^k}}
                \prth{\sum_{n=0}^M \dfrac{ \sup\limits_{z\in\mathbb{R}^2}\norm{\sbra{\Pi_{z, n}, \mathcal{V}_{z, l}}}_{\mathcal{L}^2}^2}{(n+1)^k}}^{\frac{1}{2}} \label{eq:eps1 starter}
        \end{align}
        We estimate the Hilbert-Schmidt norm with the changes of variables 
        $$\dfrac{\bx - \bz }{\sqrt{2}l_b} \to \bx, \quad \dfrac{\by - \bz}{\sqrt{2}l_b} \to \by.$$
        This gives 
        \begin{align*}
            \norm{\sbra{\Pi_{z, n}, \mathcal{V}_{z, l}}}_{\mathcal{L}^2}^2
                =& \iintr_{\mathbb{R}^2\times\mathbb{R}^2} \abs{\sbra{\Pi_{z, n}, \mathcal{V}_{z, l}}(\bx,\by)}^2 \dd \bx \dd \by  
                = \iintr_{\mathbb{R}^2\times\mathbb{R}^2}\prth{\mathcal{V}_{z, l} (\bx) - \mathcal{V}_{z, l}(\by)}^2 \abs{\Pi_{z, n}(\bx,\by)}^2 \dd \bx \dd \by\\
                =& \dfrac{1}{(2 \pi n! l_b^2)^2} \int_{\mathbb{R}^2\times \mathbb{R}^2} \prth{\mathcal{V}_{z, l}(\bx) - \mathcal{V}_{z, l}(\by)}^2\abs{\dfrac{\prth{x - z }\prth{y - z}}{2l_b^2}}^{2n} e^{-\frac{\abs{x-z}^2 + \abs{y - z}^2}{2l_b^2}}\dd \bx \dd \by\\
                =& \dfrac{1}{(\pi n!)^2} \int_{\mathbb{R}^2\times \mathbb{R}^2} \prth{\mathcal{V}_{z, l}\left( \bz + \sqrt{2}l_b \bx \right) - \mathcal{V}_{z, l}\left(\bz + \sqrt{2}l_b \by\right)}^2 \abs{xy}^{2n} e^{-\abs{x}^2 - \abs{y}^2}\dd \bx \dd \by  
        \end{align*}
        Using the expansion from Equation~\eqref{eq:expansion},
        \begin{align*}            
            \abs{\mathcal{V}_{z, l}(\bz + \sqrt{2}l_b \bx)} 
                \le \dfrac{1}{(l+1)!} \norm{d^{l+1}W}_{L^\infty} \abs{\sqrt{2}l_b x}^{l+1}
                \le e^{\sqrt{2}} \norm{d^{l+1}W}_{L^\infty} \abs{l_b x}^{l+1}
        \end{align*}
        and similarly with $y$ instead of $x$. With Lemma~\ref{lem:equiv gamma},
        \begin{align*}
            \norm{\sbra{\Pi_{z, n}, \mathcal{V}_{z, l}}}_{\mathcal{L}^2}^2
                \le& C \dfrac{\norm{d^{l+1}W}_{L^\infty}^2}{n!^2} \int_{\mathbb{R}^2 \times \mathbb{R}^2} l_b^{2(l+1)} \prth{\abs{x}^{2(l+1)} + \abs{y}^{2(l+1)}} \abs{xy}^{2n} e^{-\abs{x}^2 - \abs{y}^2} \dd \bx \dd \by\\
                \le& C \norm{d^{l+1}W}_{L^\infty}^2 \prth{(n + 1)l_b^2}^{l+1}
        \end{align*}
        Inserting this in~\eqref{eq:eps1 starter},
        \begin{align}
            \abs{\int_{\mathbb{R}^2} \varphi( \bz )\mathcal{E}_0^l( \bz ) \dd \bz }
                \le& C \sqrt{\norm{\varphi}_{L^1}\norm{\varphi}_{L^\infty}} \norm{d^{l+1}W}_{L^\infty}\dfrac{1}{l_b^2}\sqrt{\sum_{n=0}^M (n+1)^{l+1-k} l_b^{2(l+1)}} \sqrt{\Tr{\gamma_b \mathscr{L}_b^k}} \nonumber\\
                =& C \sqrt{\norm{\varphi}_{L^1}\norm{\varphi}_{L^\infty}} \norm{d^{l+1}W}_{L^\infty}\sqrt{\Tr{\gamma_b \mathscr{L}_b^k}} l_b^{l-1} \sqrt{p_{k-(l+2)}(M)} \label{eq:eps0 esti}
        \end{align}
        with $p_{k-(l+2)}(M)$ as in~\eqref{eq:power+log}.
        
        \bigskip

   \textbf{Step 2: Estimate of $\mathcal{E}_{\mathrm{II}}^{p, q}$.} From \eqref{eq:epsII},
            \begin{align*}
                &\mathcal{E}_{\mathrm{II}}^{p, q}(\bz ) \coloneq \dfrac{\sqrt{M+1}}{2\sqrt{2}\pi l_b^3 p!} \sum_{i_{1} \ldots i_p \in \sett{1, 2}^p} \partial_{i_1} \ldots \partial_{i_p} W(\bz ) i^{i_q - 1} \\
                &\Tr{\gamma_b \prth{\prod_{j=1}^{q-1} (X_{i_j} - z_{i_j})}\prth{(-1)^{i_q-1}\ket{\psi_{z, M}} \bra{\psi_{z, M+1}} + \ket{\psi_{z, M+1}} \bra{\psi_{z, M}}}\prth{\prod_{j=q+1}^p (X_{i_j} - z_{i_j})}}
            \end{align*}
            Using the Cauchy-Schwarz inequality, we get
            \begin{align}
                \abs{\int_{\mathbb{R}^2} \varphi( \bz ) \mathcal{E}_{\mathrm{II}}^{p, q}(\bz ) \dd \bz }
                    \le& C(p) \norm{d^p W}_{L^\infty} \dfrac{\sqrt{M}}{l_b^3}  \nonumber\\
                    &\int_{\mathbb{R}^2}\abs{\varphi( \bz )}\prth{\epsilon\mathcal{E}_{q-1, M}(\bz) + \epsilon \mathcal{E}_{q-1, M+1}(\bz) + \dfrac{1}{\epsilon}\mathcal{E}_{p-q, M}(\bz) + \dfrac{1}{\epsilon}\mathcal{E}_{p-q, M+1}(\bz)}\dd \bz  \label{eq:epsII decomp}
            \end{align}
    in the notation of Lemma~\ref{lem:generic term}. We insert the bound ~\eqref{eq:to use in esp II}, replacing $M+1$ by $M$, since $M\gg 1$:
        \begin{align*}
            \abs{\int_{\mathbb{R}^2} \varphi( \bz ) \mathcal{E}_{\mathrm{II}}^{p, q}(\bz ) \dd \bz }
                \le& C(p) \norm{d^p W}_{L^\infty} \norm{\varphi}_{L^\infty} \dfrac{M^{\frac{1}{2}-k}}{l_b} \left(\epsilon \prth{l_b^2 M}^{q-1} \Tr{\gamma_b \mathscr{L}_b^k \Pi_{M-q+1:M+q}} \right.\\
                &\left.+ \dfrac{1}{\epsilon} \prth{l_b^2 M}^{p-q}\Tr{\gamma_b \mathscr{L}_b^k \Pi_{M-p+q:M+p-q+1}}\right)
        \end{align*}
        Since $1\le q \le p$, choosing $\epsilon \coloneq \prth{l_b^2 M}^{\frac{p+1}{2}-q}$, we conclude
        \begin{equation}
            \abs{\int_{\mathbb{R}^2}\varphi( \bz ) \mathcal{E}_{\mathrm{II}}^{p, q}(\bz )\dd \bz }
                \le C(p) \norm{\nabla W}_{L^\infty} \norm{\varphi}_{L^\infty} M^{\frac{p}{2}-k} l_b^{p-2} \Tr{\gamma_b\mathscr{L}_b^k \Pi_{M-p+1:M+p}} \label{eq:epsII esti}
        \end{equation}

        \bigskip
        
    \noindent\textbf{Step 3: Estimate of $\mathcal{E}_{\mathrm{I}}^{p, q}$.} In order to estimate Equation~\eqref{eq:epsI}, we start with an integration by parts for which the following preparations will be helpful.
            
        Let $\odot$ denote the tensor contraction defined for $n, m \ge k$ by
        \begin{align*}
            u_1 \otimes \dots \otimes u_n \odot^k v_1 \otimes \dots \otimes v_m 
                \coloneq \bk{u_n}{v_1}\dots \bk{u_{n-k+1}}{v_k} u_1 \otimes \dots \otimes u_{n-k} \otimes v_{k+1} \otimes \dots \otimes v_m
        \end{align*}
        Identifying $d^p W(\bz )$ with the associated rank $p$ tensor, we notice that
        \begin{align*}
            &d^pW_{\bz} \left( \bx - \bz \right)^{\otimes(q-1)}\otimes \nabla_\bz ^\perp \Pi_{z, \le M}(\bx,\by) \otimes \left( \by - \bz \right)^{\otimes(p-q)} \\
                =& \nabla^{\otimes p}W(\bz ) \odot^p  \left( \bx - \bz \right)^{\otimes(q-1)}\otimes \nabla_\bz ^\perp \Pi_{z, \le M}(\bx,\by) \otimes \left( \by - \bz \right)^{\otimes(p-q)} \\
                =& \nabla_\bz ^\perp \Pi_{z, \le M}(\bx,\by) \cdot \nabla^{\otimes p}W(\bz ) \odot^{p-1}  \left( \bx - \bz \right)^{\otimes(q-1)} \otimes \left( \by - \bz \right)^{\otimes(p-q)}
        \end{align*}
        and
        \begin{align*}
            &\nabla_\bz ^\perp \cdot \nabla^{\otimes p}W(\bz ) \odot^{p-1}  \left( \bx - \bz \right)^{\otimes(q-1)} \otimes \left( \by - \bz \right)^{\otimes(p-q)} \\
                =& \prth{\nabla^\perp \cdot \nabla^{\otimes p}}W(\bz ) \odot^{p-1}  \left( \bx - \bz \right)^{\otimes(q-1)} \otimes \left( \by - \bz \right)^{\otimes(p-q)} \\
                    &+ \nabla^{\otimes p}W(\bz ) \odot^p \prth{\nabla^\perp_{\bz}  \otimes\left( \bx - \bz \right)} \otimes \left( \bx - \bz \right)^{\otimes(q-2)} \otimes \left( \by - \bz \right)^{\otimes(p-q)} \\
                    &+ \dots \\
                    &+ \nabla^{\otimes p}W(\bz ) \odot^p  \left( \bx - \bz \right)^{\otimes(q-1)} \otimes \left( \by - \bz \right)^{\otimes(p-q-1)} \otimes \prth{\nabla_\bz ^\perp \otimes \left( \by - \bz \right)} 
        \end{align*}
        but because $\nabla^\perp \cdot \nabla = 0, \nabla_\bz ^\perp \otimes \left( \bx - \bz \right) = \nabla_\bz ^\perp \otimes \left( \by - \bz \right) = \prth{\matrx{0&1\\-1&0}}$,
        \begin{align*}
            &\nabla_\bz ^\perp \cdot \nabla^{\otimes p}W(\bz ) \odot^{p-1}  \left( \bx - \bz \right)^{\otimes(q-1)} \otimes \left( \by - \bz \right)^{\otimes(p-q)} \\
                =& \prth{\nabla^\perp \cdot \nabla} \nabla^{\otimes(p-1)}W(\bz ) \odot^{p-1}  \left( \bx - \bz \right)^{\otimes(q-1)} \otimes \left( \by - \bz \right)^{\otimes(p-q)} \\
                    &+ \nabla^{\otimes p}W(\bz ) \odot^p \prth{\matrx{0&1\\-1&0}} \otimes \left( \bx - \bz \right)^{\otimes(q-2)} \otimes \left( \by - \bz \right)^{\otimes(p-q)} \\
                    &+ \dots \\
                    &+ \nabla^{\otimes p}W(\bz ) \odot^p  \left( \bx - \bz \right)^{\otimes(q-1)} \otimes \left( \by - \bz \right)^{\otimes(p-q-1)} \otimes \prth{\matrx{0&1\\-1&0}} \\
                =& \prth{\nabla^{\otimes p}W(\bz ) \odot^2 \prth{\matrx{0&1\\-1&0}}} \odot^{p-2} \left( \bx - \bz \right)^{\otimes(q-2)} \otimes \left( \by - \bz \right)^{\otimes(p-q)} \\
                    &+ \dots \\
                    &+ \prth{\nabla^{\otimes p}W(\bz ) \odot^{p-2}  \left( \bx - \bz \right)^{\otimes(q-1)} \otimes \left( \by - \bz \right)^{\otimes(p-q-1)}} \odot^{2} \prth{\matrx{0&1\\-1&0}} \\
                    =&0. 
        \end{align*}
        Indeed, since 
        \begin{align*}
            \nabla^{\otimes p}W(\bz ), \dots, \nabla^{\otimes p}W(\bz ) \odot^{p-2}  \left( \bx - \bz \right)^{\otimes(q-1)} \otimes \left( \by - \bz \right)^{\otimes(p-q-1)}
        \end{align*}
        are symmetric tensors and $\prth{\matrx{0&1\\-1&0}}$ is antisymmetric their contraction product is null.
            
        In view of these considerations, an integration by parts gives
        \begin{align}
            &\int_{\mathbb{R}^2}\varphi( \bz )  \mathcal{E}_{\mathrm{I}}^{p, q}\left( \bz \right)\dd \bz  \nonumber\\
                =& -\dfrac{1}{2\pi l_b^2 p!}  \int_{\mathbb{R}^2}\Tr{\gamma_b d^p W_{\bz} \prth{\bX- \bz}^{\otimes(q-1)}\otimes \nabla^\perp \varphi( \bz ) \Pi_{z, \le M}\otimes\prth{\bX- \bz}^{\otimes (p-q)}}\dd \bz  \nonumber\\
                =& -\dfrac{1}{2\pi l_b^2 p!}  \sum_{i_{1:p} \in \sett{1, 2}^p}\int_{\mathbb{R}^2}\partial_{i_{1:p}}W(\bz ) \nabla^\perp_{i_q} \varphi( \bz ) \nonumber\\
                    &\cdot \Tr{\gamma_b \prth{\prod_{j=1}^{q-1} \prth{X_{i_j} - z_{i_j}}} \Pi_{z, \le M} \prth{\prod_{j=q+1}^p \prth{X_{i_j} - z_{i_j}}}}\dd \bz  \label{eq:epsI decomp}
        \end{align}
        Using the Cauchy-Schwarz inequality and~\eqref{eq:to use in esp II} leads to
        \begin{align*}
            \abs{\int_{\mathbb{R}^2}\varphi( \bz )  \mathcal{E}_{\mathrm{I}}^{p, q}\left( \bz \right)\dd \bz }
                &\le \dfrac{C(p)}{2\pi l_b^2} \norm{d^p W}_{L^\infty} \norm{\nabla \varphi}_{L^\infty} \sum_{n=0}^M \int_{\R^2}  \prth{\epsilon\mathcal{E}_{q-1, n}(\bz) + \frac{1}{\epsilon}\mathcal{E}_{p-q, n}(\bz)}\dd \bz \\
                &\le C(p) \norm{d^p W}_{L^\infty} \norm{\nabla \varphi}_{L^\infty}  \\
                &\sum_{n=0}^M \prth{\epsilon_n l_b^{2(q-1)} (n+1)^{q-1-k} + \dfrac{1}{\epsilon_n} l_b^{2(p-q)} (n+1)^{p-q-k}}\Tr{\gamma_b \mathscr{L}_b^k \Pi_{\le M+p-1}}
        \end{align*}
        Choosing $\epsilon_n \coloneq \prth{l_b^2 n}^{\frac{p+1}{2}-q}$ and recalling \cref{eq:power+log}, we conclude
        \begin{equation}
            \abs{\int_{\mathbb{R}^2}\varphi( \bz )  \mathcal{E}_{\mathrm{I}}^{p, q}\left( \bz \right)\dd \bz }
                \le C(p) \norm{d^p W}_{L^\infty} \norm{\nabla \varphi}_{L^\infty} l_b^{p-1} p_{k- \frac{p+1}{2}}(M) \Tr{\gamma_b\mathscr{L}_b^k \Pi_{\le M+p-1}} \label{eq:epsI esti}
        \end{equation}

    \bigskip
        
    \noindent\textbf{Step 4: Conclusion.} Putting Equations~\eqref{eq:error decomp}, \eqref{eq:eps0 esti}, \eqref{eq:epsII esti} and \eqref{eq:epsI esti} together we obtain
        \begin{align*}
            &\abs{\int_{\mathbb{R}^2}\varphi( \bz )\prth{\partial_t \rho_{\gamma_b}^{sc, \le M}(t, \bz) + \nabla^\perp W(\bz ) \cdot \nabla \rho_{\gamma_b}^{sc, \le M}(t, \bz)} \dd \bz } \\
                \le& \left(C \sqrt{\norm{\varphi}_{L^1}\norm{\varphi}_{L^\infty}} \norm{d^{l+1}W}_{L^\infty}\sqrt{\Tr{\gamma_b \mathscr{L}_b^k}} l_b^{l-1} \sqrt{p_{k-(l+2)}(M)} \right. \\
                    +& \sum_{p=1}^l\sum_{q=1}^p C(p) \norm{\nabla W}_{L^\infty} \norm{\varphi}_{L^\infty} M^{\frac{p}{2}-k} l_b^{p-2} \Tr{\gamma_b\mathscr{L}_b^k \Pi_{M-p+1:M+p}} \\
                    +& \left. \sum_{p=2}^l\sum_{q=1}^p C(p) \norm{d^p W}_{L^\infty} \norm{\nabla \varphi}_{L^\infty} l_b^{p-1} p_{k- \frac{p+1}{2}}(M) \Tr{\gamma_b\mathscr{L}_b^k \Pi_{\le M+p-1}} \right) \\
                \le& C(l) \norm{\varphi}_{L^1 \cap W^{1, \infty}} \norm{W}_{W^{l+1, \infty}} \left(\sqrt{\Tr{\gamma_b \mathscr{L}_b^k}} l_b^{l-1} \sqrt{p_{k-(l+2)}(M)} \right. \\
                    +& \Tr{\gamma_b\mathscr{L}_b^k \Pi_{M-l+1:M+l}} \sum_{p=1}^l  M^{\frac{p}{2}-k} l_b^{p-2} 
                    + \left. \Tr{\gamma_b\mathscr{L}_b^k } \sum_{p=2}^l l_b^{p-1} p_{k- \frac{p+1}{2}}(M) \right) 
                    \end{align*}
                    and then
                    \begin{align*}
               \mathrm{[...]} \le& C(l) \norm{\varphi}_{L^1 \cap W^{1, \infty}} \norm{W}_{W^{l+1, \infty}} \left(\sqrt{\Tr{\gamma_b \mathscr{L}_b^k}} l_b^{l-1} \syst{\matrx{
                        M^{1 + \frac{l-k}{2}} &\text{ if }   k < l+2   \\
                        \sqrt{\ln(M)} &\text{ if }  k= l+2\\
                        1 &\text{ if } k > l+2 \\
                    }}\right. \\
                    +& \dfrac{\Tr{\gamma_b\mathscr{L}_b^k \Pi_{M-l+1:M+l}}}{l_b M^{k - \frac{1}{2}}}
                    + \left. \Tr{\gamma_b\mathscr{L}_b^k} \sum_{p=2}^l l_b^{p-1} \syst{\matrx{
                        M^{\frac{p+1}{2} - k} &\text{ if }   k < \frac{p+1}{2}   \\
                        \ln(M) &\text{ if }  k= \frac{p+1}{2}\\
                        1 &\text{ if } k > \frac{p+1}{2} \\
                    }} \right)
        \end{align*}
        where we used $l_b^2 M \le 1$ for the last inequality.
    \end{proof}

    In addition to the estimate of Proposition~\ref{prop:vorticity} we will need to control the time derivative of the normalization of $\rho^{sc, \le M}_{\gamma_b}$, with similar ideas as above:
    
    \begin{corollary}[\textbf{Controling the normalization}]\mbox{}\\
        Under the assumptions of Proposition~\ref{prop:vorticity} with in addition $l\ge 2, d^{l+1}W \in L^1\prth{\mathbb{R}^2}$, we have
        \begin{align}
            \abs{\partial_t  \Tr{\gamma_b \Pi_{\le M}}}  
                \le& C(l) \prth{\norm{W}_{W^{l+1, \infty}} + \norm{d^{l+1}W}_{L^1}}\left(l_b^{l-2} M^{\frac{l+3}{2}} 
                    + \dfrac{\Tr{\gamma_b\mathscr{L}_b^k \Pi_{M-l+1:M+l}}}{l_b M^{k - \frac{1}{2}}} \right. \nonumber\\
                    +& \left. \Tr{\gamma_b\mathscr{L}_b^k} \sum_{p=2}^l l_b^{p-1} \syst{\matrx{
                        M^{\frac{p+1}{2} - k} &\text{ if }   k < \frac{p+1}{2}   \\
                        \ln(M) &\text{ if }  k= \frac{p+1}{2}\\
                        1 &\text{ if } k > \frac{p+1}{2} \\
                    }} \right) \label{eq:dt norm}
        \end{align}
    \end{corollary}
    \begin{proof}
        An integration by parts gives
        \begin{align*}
            \int_{\mathbb{R}^2} \nabla^\perp W(\bz ) \cdot \nabla_\bz \rho_{\gamma_b}^{sc, \le M}(\bz )\dd \bz  = 0
        \end{align*}
        so that we may estimate
        \begin{align}
            \partial_t  \Tr{\gamma_b \Pi_{\le M}} 
                =&\ \partial_t \int_{\mathbb{R}^2} \Tr{\gamma_b \Pi_{\le{z, M}}} \frac{\dd \bz }{2\pi l_b^2}
                = \int_{\mathbb{R}^2} \partial_t \rho_{\gamma_b}^{sc, \le M}(\bz )\dd \bz  \nonumber \\
                =&\ \int_{\mathbb{R}^2} \prth{\partial_t \rho_{\gamma_b}^{sc, \le M}(\bz ) + \nabla^\perp W(\bz ) \cdot \nabla_\bz \rho_{\gamma_b}^{sc, \le M}(\bz )}\dd \bz  \label{eq:thetrick}
        \end{align}
        We will apply ideas similar to the proof of Proposition~\ref{prop:vorticity}, but this time with $\varphi = 1$. As in the previous proof we start from~\eqref{eq:error decomp}.  Our bounds on the $\mathcal{E}_{\rm I}^{p,q}$ and $\mathcal{E}_{\rm II}^{p,q}$ apply mutatis mutandis, for they were estimated in terms of $\norm{\varphi}_{L^\infty}$ only. Hence we only need to adapt the estimate of $\mathcal{E}_0^l$ to get rid of the dependence on $\norm{\varphi}_{L^1}$. 
        
        We write
        \begin{align}
            \int_{\mathbb{R}^2} \mathcal{E}_0^l( \bz ) \dd \bz  
                =\dfrac{1}{i l_b^2} \int_{\mathbb{R}^2} \Tr{\gamma_b\sbra{\Pi_{z, \le M}, \mathcal{V}_{z, l}}} \frac{\dd \bz }{2\pi l_b^2}
                = \frac{2}{l_b^2}\text{Im} \int_{\mathbb{R}^2}\Tr{\gamma_b \Pi_{z, \le M} \mathcal{V}_{z, l}}\frac{\dd \bz }{2\pi l_b^2} \label{eq:eps 0 bis}
        \end{align}
        using Taylor's formula
        \begin{align*}
            \mathcal{V}_{z, l}(\bx) \coloneq W(\bx) - \sum_{p=0}^l \dfrac{1}{p!} d^p W_{\bz} \prth{\bx - \bz }^{\otimes p} 
                = \frac{1}{l!}\int_0^1 (1-\tau)^ld^{l+1}W_{\bz + \tau(\bx-\bz)}(\bx-\bz)^{\otimes(l+1)}\dd \tau 
        \end{align*}
        Let $n\in \mathbb{N}$, we bound
        \begin{align}
            &\left|\int_{\mathbb{R}^2}\Tr{\gamma_b \Pi_{z, n} \mathcal{V}_{z, l}}\frac{\dd \bz }{2\pi l_b^2}\right|
                \le\ \int_{\mathbb{R}^2\times\mathbb{R}^2\times\mathbb{R}^2} \abs{\gamma_b(\bx,\by) \Pi_{z, n}(\by, \bx ) \mathcal{V}_{z, l}(\bx)} \dd \bx \dd \by\frac{\dd \bz }{2\pi l_b^2} \nonumber \\
                &\le\ \frac{1}{l!} \int_{\mathbb{R}^2\times\mathbb{R}^2\times\mathbb{R}^2} \abs{\gamma_b(\bx,\by)} \abs{\Pi_{z, n}(\by, \bx )} \int_0^1 (1-\tau)^l \abs{d^{l+1}W_{\bz+\tau(\bx-\bz)}}\dd \tau  \abs{x-z}^{l+1}\dd \bx \dd \by\frac{\dd \bz }{2\pi l_b^2} \nonumber\\
                &\le\ \frac{1}{l!} \prth{\ \int_{\mathbb{R}^2\times\mathbb{R}^2\times\mathbb{R}^2} \abs{\gamma_b(\bx,\by)}^2 \int_0^1 (1-\tau)^l \abs{d^{l+1}W_{\bz+\tau(\bx-\bz)}}\dd \tau  \dd \bx \dd \by\frac{\dd \bz }{2\pi l_b^2}}^\frac{1}{2} \nonumber\\
                    &\prth{\ \int_{\mathbb{R}^2\times\mathbb{R}^2\times\mathbb{R}^2} \abs{\Pi_{z, n}(\by, \bx )}^2 \int_0^1 (1-\tau)^l \abs{d^{l+1}W_{\bz+\tau(\bx-\bz)}}\dd \tau  \abs{x-z}^{2(l+1)}\dd \bx \dd \by\frac{\dd \bz }{2\pi l_b^2}}^\frac{1}{2} \label{eq:eps 0 bis step}
        \end{align}
        As $l\ge 2$, we can estimate the first factor by performing the following integral with the change of variable $(1-\tau)\bz \mapsto \bz$
        \begin{align*}
            \int_{\mathbb{R}^2} \int_0^1 (1-\tau)^l \abs{d^{l+1}W_{\bz+\tau(\bx-\bz)}}\dd \tau  \frac{\dd \bz }{2\pi l_b^2} 
                = \frac{\norm{d^{l+1}W}_{L^1}}{2\pi l_b^2}\int_0^1 (1-\tau)^{l-2}\dd \tau 
                = \frac{\norm{d^{l+1}W}_{L^1}}{2(l-1)\pi l_b^2}\
        \end{align*}
        so with the Pauli principle
        \begin{align}
            \int_{\mathbb{R}^2\times\mathbb{R}^2\times\mathbb{R}^2} \abs{\gamma_b(\bx,\by)}^2 \int_0^1 (1-\tau)^l \abs{d^{l+1}W_{\bz+\tau(\bx-\bz)}}\dd \tau  \dd \bx \dd \by\frac{\dd \bz }{2\pi l_b^2}
                \le \frac{\norm{d^{l+1}W}_{L^1}}{2(l-1)\pi l_b^2} \Tr{\gamma_b^2}
                \le \frac{\norm{d^{l+1}W}_{L^1}}{l-1} \label{eq:1st factor}
        \end{align}
        For the second factor in \cref{eq:eps 0 bis step}, after the $\dd \by  $ integration and the change of variable $\frac{x-z}{\sqrt{2}l_b} \mapsto x$,
        \begin{align}
            &\int_{\mathbb{R}^2\times\mathbb{R}^2\times\mathbb{R}^2} \abs{\Pi_{z, n}(\by, \bx )}^2 \int_0^1 (1-\tau)^l \abs{d^{l+1}W_{\bz+\tau(\bx-\bz)}}\dd \tau  \abs{x-z}^{2(l+1)}\dd \bx \dd \by\frac{\dd \bz }{2\pi l_b^2} \nonumber \\
                =&\ \frac{\prth{2 l_b^2}^{l+1}}{\pi} \int_{\mathbb{R}^2\times\mathbb{R}^2} \abs{\psi_{z, n}\prth{\bz + \sqrt{2}l_b \bx}}^2 \int_0^1 (1-\tau)^l \abs{d^{l+1}W_{\bz+\sqrt{2}l_b\tau \bx}}\dd \tau  \abs{x}^{2(l+1)}\dd \bx \dd \bz  \nonumber\\
                =&\ \frac{\prth{2l_b^2}^{l+1}}{\pi} \int_{\mathbb{R}^2\times\mathbb{R}^2} \frac{1}{2\pi n! l_b^2} \abs{x}^{2(n+l+1)}e^{-\abs{x}^2} \int_0^1 (1-\tau)^l \abs{d^{l+1}W_{\bz+\sqrt{2}l_b\tau \bx}}\dd \tau  \dd \bx \dd \bz  \nonumber \\
                =&\ \frac{2 \prth{2 l_b^2}^l \norm{d^{l+1} W}_{L^1}}{\pi(l+1)n!} \int_{\mathbb{R_+}} u^{2(n+l+1)+1}e^{-u^2}\dd u 
                =\frac{\prth{2l_b^2}^l \norm{d^{l+1} W}_{L^1}}{\pi(l+1)} \frac{(n+l+1)!}{n!} \nonumber\\
                &\le\ C(l) \norm{d^{l+1} W}_{L^1} l_b^{2l} (n+1)^{l+1}\label{eq:2nd factor}
        \end{align}
        Inserting \cref{eq:1st factor} and \cref{eq:2nd factor} in \cref{eq:eps 0 bis step} leads to
        \begin{align*}
            \abs{\int_{\ \mathbb{R}^2}\Tr{\gamma_b \Pi_{z, n} \mathcal{V}_{z, l}}\frac{\dd \bz }{2\pi l_b^2}}
                \le C(l) \norm{d^{l+1} W}_{L^1} l_b^{l} (n+1)^{\frac{l+1}{2}}
        \end{align*} 
        and thus, combining with \cref{eq:eps 0 bis} we get
        \begin{align*}
            \abs{\ \int_{\mathbb{R}^2} \mathcal{E}_0^l( \bz ) \dd \bz  }
                \le C(l) \norm{d^{l+1} W}_{L^1} l_b^{l-2} \sum_{n=0}^M (n+1)^{\frac{l+1}{2}}
                \le  C(l) \norm{d^{l+1} W}_{L^1} l_b^{l-2} M^{\frac{l+3}{2}}
        \end{align*}
        We conclude by using \cref{eq:thetrick} and proceeding as in the proof of Proposition~\ref{prop:vorticity}, picking $\varphi = 1$ and replacing by the above the estimate of the $\mathcal{E}_0^l( \bz )$ term therein.
    \end{proof}

\subsection{Dynamics of the first reduced density}\label{ssec:proof}

    In this part we conclude the proof of Theorem~\ref{th:V of HF}. What is left to do is to put together the estimates of Sections~\ref{sec:densities} and \ref{sec:sc EDP 1}. We start by explaining how to fix the Landau level cut-off $M$, as hinted at in Remark~\ref{rem:choices}.

\begin{lemma}[\textbf{Fixing the Landau level cut-off}]\label{lem:N choice}\mbox{}\\
        Let $\alpha >0$, $k \ge 0, \varphi \in L^1\prth{\R_+}, \gamma_b \in  L^\infty\prth{\R_+, \mathcal{L}^1\prth{L^2\prth{\mathbb{R}^2}}}$ and assume
        \begin{align*}
            & \forall t \in \R_+, \gamma_b(t) \ge 0, \Tr{\gamma_b(t)} = 1 \\
            &\int_{\R_+} \abs{\varphi(t)} \Tr{\gamma_b(t) \mathscr{L}_b^k} \dd t < \infty
        \end{align*}
        then $\exists M(\varphi) \in \intint{\floor{l_b^{-\alpha}}, 2\floor{l_b^{-\alpha}}}$ such that
        \begin{align*}
            \int_{\R_+} \abs{\varphi(t)} \Tr{\gamma_b(t) \mathscr{L}_b^k \Pi_{M(\varphi)-l+1:M(\varphi)+l}} \dd t 
                \le C(l) l_b^\alpha \int_{\R_+} \abs{\varphi(t)} \Tr{\gamma_b(t) \mathscr{L}_b^k}\dd t 
        \end{align*}
    \end{lemma}    
    \begin{proof}
        Assume for contradiction that $\forall  M \in \intint{\floor{l_b^{-\alpha}}, 2\floor{l_b^{-\alpha}}},$
        \begin{align*}
            \int_{\R_+} \abs{\varphi(t)} \Tr{\gamma_b(t) \mathscr{L}_b^k \Pi_{M-l+1:M+l}} \dd t 
                > \dfrac{4l}{M}\int_{\R_+} \abs{\varphi(t)} \Tr{\gamma_b(t) \mathscr{L}_b^k}\dd t
        \end{align*}
        then
        \begin{align*}
            \int_{\R_+} \abs{\varphi(t)} \Tr{\gamma_b(t) \mathscr{L}_b^k}\dd t
                &\ge \dfrac{1}{2l} \sum_{M=\floor{l_b^{-\alpha}}}^{2\floor{l_b^{-\alpha}}} \int_{\R_+} \abs{\varphi(t)} \Tr{\gamma_b(t) \mathscr{L}_b^k \Pi_{M-l+1:M+l}} \dd t \\
                &> \int_{\R_+} \abs{\varphi(t)} \Tr{\gamma_b(t) \mathscr{L}_b^k } \dd t \sum_{M=\floor{l_b^{-\alpha}}}^{2\floor{l_b^{-\alpha}}} \dfrac{2}{M} \\
                &\ge \dfrac{\floor{l_b^{-\alpha}}+1}{\floor{l_b^{-\alpha}}} \int_{\R_+} \abs{\varphi(t)} \Tr{\gamma_b(t) \mathscr{L}_b^k}\dd t
        \end{align*}
        which yields the desired contradiction.
    \end{proof}
    
    Let us now summarize some findings of the previous subsection:

    \begin{lemma}[\textbf{Summary of semi-classical approximation}]\label{lem:sc summary}\mbox{}\\
        Let $\frac{2}{3} < \alpha < 1$ and $\gamma_b \in L^\infty \prth{\mathbb{R}_+, \mathcal{L}^1\prth{L^2\prth{\mathbb{R}^2}}}$ be a solution of \eqref{eq:retime HF} with initial datum satisfying
        \begin{align*}
            \Tr{\gamma_b(0)} = 1, \quad 0\le \gamma_b(0) \le 2\pi l_b^2, \quad \Tr{\gamma_b(0) H_b(0)} < C
        \end{align*}
        with $ V, w \in  W^{4, \infty}\prth{\mathbb{R}^2}$, then $\forall \varphi \in L^1\prth{\R_+, L^1 \cap W^{1, \infty}\prth{\R^2}}, \exists M = \mathcal{O}\prth{l_b^{-\alpha}}$ such that
        \begin{itemize}
            \item $\forall t \geq 0, \forall \mu \in L^\infty \cap H^1\prth{\R^2}$,
                \begin{align}
                &\abs{\int_{\mathbb{R}^2} \mu \prth{\rho_{\gamma_b}(t) - \rho_{\gamma_b}^{sc, \le M}(t)}} 
                    \le C\prth{\norm{\mu}_{L^\infty} + \norm{\nabla\mu}_{L^2}} \prth{\abs{\Tr{\gamma_b(0) H_b(0)}} + \norm{V}_{L^\infty} + \norm{w}_{L^\infty}}l_b^{\frac{\alpha}{2}} \label{eq:weak density conv} \\
                &\abs{\int_{\mathbb{R}^2} \mu \prth{\rho_{\gamma_b}(t) - \rho_b(t)}} 
                    \le C\prth{\norm{\mu}_{L^\infty} + \norm{\nabla\mu}_{L^2}} \prth{\abs{\Tr{\gamma_b(0) H_b(0)}} + \norm{V}_{L^\infty} + \norm{w}_{L^\infty}}l_b^{\frac{\alpha}{2}}\label{eq:weak density conv bis}
                \end{align}
            \item moreover
                \begin{align}
                    &\abs{\int_{\R_+\times \R^2}\varphi \mathrm{DRIFT}_{\rho_{\gamma_b}}\prth{\rho_{\gamma_b}^{sc, \le M}}}
                    \le C \norm{\varphi}_{L^1\prth{\R_+, L^1\cap W^{1, \infty}\prth{\R^2}}} \prth{\norm{V}_{W^{4,\infty}} + \norm{w}_{W^{4,\infty}}} \nonumber \\
                    &\prth{\abs{\Tr{\gamma_b(0) H_b(0)}} + \norm{V}_{L^\infty} + \norm{w}_{L^\infty}} \prth{l_b^{2-2\alpha} + l_b^{\frac{3}{2}\alpha-1}} \label{eq:weak dynamics conv}
                \end{align}
            \item if $p > 3$ and $\Tr{\gamma_b(t) \abs{X}^p} < \infty$, then
                \begin{align}
                    W_1\left(\rho_{\gamma_b}(t) , \rho_b(t)\right)
                        \le C_1(t, p, V, w) \prth{l_b^{1-\frac{\alpha}{2}-\frac{6 + \alpha}{2p-4}} + l_b^{\frac{\alpha}{2}}} \label{eq:wassertein density conv}
                \end{align}
                with
                \begin{equation*}
                    C_1(t, p, V, w) \coloneq C(p) \prth{1 + \Tr{\gamma_b(t) \abs{X}^p}} \prth{\abs{\Tr{\gamma_b(0) H_b(0)}} + \norm{V}_{L^\infty} + \norm{w}_{L^\infty}}
                \end{equation*}
            \item if $V, w \in W^9$ and $d^9 V, d^9 W \in L^1\prth{\mathbb{R}^2}$,
                \begin{align}
                    &\abs{\int_{\R_+\times \R^2}\varphi \mathrm{DRIFT}_{\rho_{\gamma_b}}\prth{\rho_b}}
                    \le C \norm{\varphi}_{L^1\prth{\R_+, L^1\cap W^{1, \infty}\prth{\R^2}}}\prth{\abs{\Tr{\gamma_b(0) H_b(0)}} + \norm{V}_{L^\infty} + \norm{w}_{L^\infty}} \nonumber \\
                &\prth{\norm{V}_{W^{9,\infty}} + \norm{d^9V}_{L^1} + \norm{w}_{W^{9,\infty}} + \norm{d^9w}_{L^1}}  \prth{l_b^{2-2\alpha} + l_b^{\frac{3}{2}\alpha-1}}     \label{eq:normalized drift}
                \end{align}
        \end{itemize}
    \end{lemma}
    \begin{proof}  
        By Lemma~\ref{lem:fermions conserv},
        \begin{align*}
            \Tr{\gamma_b(t)} = 1 \text{ and } 0\le \gamma_b(t) \le 2\pi l_b^2
        \end{align*}
        then with Lemma~\ref{lem:E conserv} applied to $W \coloneq V + \dfrac{1}{2}w\star \rho_{\gamma_b(t)}$,
        \begin{align}
            \Tr{\gamma_b(t) \mathscr{L}_b} 
                \le& \abs{\Tr{\gamma_b(t) H_b(t)}} + \norm{V + \dfrac{1}{2}w\star\rho_{\gamma_b(t)}}_{L^\infty} \nonumber\\
                =& \abs{\Tr{\gamma_b(0) H_b(0)}} + \norm{V + \dfrac{1}{2}w\star\rho_{\gamma_b(t)}}_{L^\infty} \nonumber\\
                \le& \abs{\Tr{\gamma_b(0) H_b(0)}} + \norm{V}_{L^\infty} + \norm{w}_{L^\infty} 
                \label{eq:t indep E bound}
        \end{align}
        Moreover
        \begin{equation}
            \norm{V + w\star \rho_{\gamma_b}(t)}_{W^{4,\infty}} \le \norm{V}_{W^{4,\infty}} + \norm{w}_{W^{4,\infty}} 
            \label{eq:W choice}
        \end{equation}
        For this proof we choose $M \coloneq M\prth{t\mapsto \norm{\varphi(t)}_{L^1 \cap W^{1, \infty}\prth{\R^2}}}$ according to Lemma~\ref{lem:N choice}.
        
        \medskip
        
    \noindent\textbf{Proof of~\eqref{eq:weak density conv}:} using~\eqref{eq:GYRO esti 2} for $k=1$, \eqref{eq:t indep E bound}, $\sqrt{\Tr{\gamma_b \mathscr{L}_b}} \le \Tr{\gamma_b \mathscr{L}_b}$  and $M = \mathcal{O}\prth{l_b^{-\alpha}}$,
        \begin{align*}
            &\abs{\int_{\mathbb{R}^2} \mu \prth{\rho_{\gamma_b}(t) - \rho_{\gamma_b}^{sc, \le M}(t)}} \\
                \le& C\prth{\norm{\mu}_{L^\infty} + \norm{\nabla\mu}_{L^2}} \prth{M^{-\frac{1}{2}} + l_b \sqrt{M}} \Tr{\gamma_b(t) \mathscr{L}_b}  \\
                \le& C\prth{\norm{\mu}_{L^\infty} + \norm{\nabla\mu}_{L^2}} \prth{\abs{\Tr{\gamma_b(0) H_b(0)}} + \norm{V}_{L^\infty} + \norm{w}_{L^\infty}}
                \prth{l_b^{\frac{\alpha}{2}} + l_b^{1-\frac{\alpha}{2}}}
        \end{align*}
        we obtain ~\eqref{eq:weak density conv} by noticing that
        \begin{align*}
            \alpha < 1 \implies \frac{\alpha}{2} \le 1-\frac{\alpha}{2}
        \end{align*}
        To obtain \cref{eq:weak density conv bis} we only need to incorporate 
        \begin{align*}
            \abs{\int_{\mathbb{R}^2}\mu\prth{\rho_b - \rho_{\gamma_b}^{sc, \le M}}}
                =&\ \abs{\ \int_{\mathbb{R}^2} \mu\rho_b\prth{1 - \Tr{\gamma_b \Pi_{\le M}}}} 
                \le\norm{\mu}_{L^\infty} \Tr{\gamma_b \Pi_{>M}} \\
                \le&\ \norm{\mu}_{L^\infty}M^{-k}\Tr{\gamma_b \Pi_{>M}\mathscr{L}_b^k}
        \end{align*}
        in the argument and discard this extra error term since $M \ge \sqrt{M}$.
        
        \medskip
        
    \noindent\textbf{Proof of~\eqref{eq:weak dynamics conv}:} integrating Equation~\eqref{eq:dynamic rewritten} in time with $W \coloneq V + w\star\rho_{\gamma_b}, l \coloneq 3$ and using \eqref{eq:W choice}, we find
        \begin{align*}
            &\abs{\int_{\R_+\times\R^2}\varphi \mathrm{DRIFT}_{\rho_{\gamma_b}}\prth{\rho_{\gamma_b}^{sc, \le M}}}
                \le C  \prth{\norm{V}_{W^{4,\infty}} + \norm{w}_{W^{4,\infty}}} \\
                    &\int_{\R_+} \norm{\varphi(t)}_{L^1 \cap W^{1, \infty}}
                \prth{
                    \Tr{\gamma_b(t) \mathscr{L}_b}\prth{l_b^{2}M^{2} +  l_b \sqrt{M}}
                    + \dfrac{\Tr{\gamma_b(t)\mathscr{L}_b \Pi_{M-2:M+3}}}{l_b \sqrt{M}} 
                }\dd t
        \end{align*}
        Then, using Lemma~\ref{lem:N choice} and \eqref{eq:t indep E bound},
        \begin{align*}
            &\abs{\int_{\R_+\times\R^2}\varphi \mathrm{DRIFT}_{\rho_{\gamma_b}}\prth{\rho_{\gamma_b}^{sc, \le M}}}  \\
                \le& C \prth{\norm{V}_{W^{4,\infty}} + \norm{w}_{W^{4,\infty}}}\prth{l_b^{2-2\alpha} + l_b^{1-\frac{\alpha}{2}} + l_b^{\frac{3}{2}\alpha-1}} \int_{\R_+}  \norm{\varphi(t)}_{L^1 \cap W^{1, \infty}} \Tr{\gamma_b(t) \mathscr{L}_b} \dd t \\
                \le& C\norm{\varphi}_{L^1\prth{\R_+, L^1\cap W^{1, \infty}\prth{\R^2}}} \prth{\norm{V}_{W^{4,\infty}} + \norm{w}_{W^{4,\infty}}}\prth{\abs{\Tr{\gamma_b(0) H_b(0)}} + \norm{V}_{L^\infty} + \norm{w}_{L^\infty}} \\
                    &\prth{l_b^{2-2\alpha} + l_b^{1-\frac{\alpha}{2}} + l_b^{\frac{3}{2}\alpha-1}}
        \end{align*}
         This leads to~\eqref{eq:weak dynamics conv} after noticing that
        \begin{align*}
            \dfrac{2}{3} < \alpha \implies 2-2\alpha \le 1 - \dfrac{\alpha}{2}
        \end{align*}
        
        \medskip

    \noindent\textbf{Proof of~\eqref{eq:wassertein density conv}:} let $\beta > 0$, we have 
        \begin{align*}
            \dfrac{2}{3} < \alpha < 1 \implies 4\sqrt{M} l_b = \mathcal{O}\prth{l_b^{1-\frac{\alpha}{2}}} \ll l_b^{\frac{\alpha}{2}} \ll l_b^{-\beta}
        \end{align*}
        so we can apply Proposition~\eqref{prop:cvg t0} with $k=1$ and obtain
        \begin{align*}
        W_1\prth{\rho_{\gamma_b}(t), \rho_b(t)}
            \le&\ C(p) \prth{1 + \Tr{\gamma_b(t) \abs{X}^p}} \Tr{\gamma_b(t) \mathscr{L}_b} \\
                &\prth{M^{-1} + l_b^{1-\beta} \sqrt{M} + M^{-\frac{1}{2}} + l_b^{\beta(p-1)} + M l_b^{\beta(p-3) -2}} \\
            \le&\ C(p) \norm{\mu}_{W^{1, \infty}}\prth{1 + \Tr{\gamma_b(t) \abs{X}^p}} \prth{\abs{\Tr{\gamma_b(0) H_b(0)}} + \norm{V}_{L^\infty} + \norm{w}_{L^\infty}}\\
                &\prth{l_b^{\alpha} + l_b^{1-\frac{\alpha}{2}-\beta} + l_b^{\frac{\alpha}{2}} + l_b^{\beta(p-1)} + l_b^{\beta(p-3)-2 -\alpha}}
        \end{align*}
        Remark that $\alpha > \frac{\alpha}{2}$ and
        \begin{align*}
            \beta(p-1) \ge \beta(p-3)-2 -\alpha
        \end{align*}
        so that optimizing the above with respect to $\beta$ leads to
        \begin{align*}
            \beta(p-3)-2 -\alpha = 1 - \frac{\alpha}{2} - \beta \implies \beta =\dfrac{3 + \frac{\alpha}{2}}{p-2}
        \end{align*}
        which finally gives \eqref{eq:wassertein density conv}.

        \medskip
        
    \noindent\textbf{Proof of~\eqref{eq:normalized drift}:} with \cref{eq:power to L} and then \cref{eq:t indep E bound}, we note that $\rho_b$ is well defined for large enough $M$ as
        \begin{align}
            \Tr{\gamma_b(t) \Pi_{\le M}} 
                &= 1 - \Tr{\gamma_b(t) \Pi_{>M}}
                \ge 1 - \frac{\Tr{\gamma_b(t) \mathscr{L}_b\Pi_{>M}}}{M}
                \ge 1 - \frac{\Tr{\gamma_b(t) \mathscr{L}_b}}{M} \nonumber\\
                &\ge 1 - \frac{\abs{\Tr{\gamma_b(0) H_b(0)}} + \norm{V}_{L^\infty} + \norm{w}_{L^\infty}}{M} \label{eq:norm exist}
        \end{align}
        Recalling~\eqref{eq:final rho} we compute
        \begin{align}
            \mathrm{DRIFT}_{\rho_{\gamma_b}}\prth{\rho_b}
                = \frac{\mathrm{DRIFT}_{\rho_{\gamma_b}}\prth{\rho_{\gamma_b}^{sc, \le M}}}{\Tr{\gamma_b \Pi_{\le M}}} - \rho_b \frac{\partial_t  \Tr{\gamma_b \Pi_{\le M}}}{\Tr{\gamma_b \Pi_{\le M}}} \label{eq:dt drift norm}
        \end{align}
        With \cref{eq:dt norm} applied to $W \coloneq V + w\star\rho_{\gamma_b}(t), k=1, l\ge 3$, noticing that
        \begin{align*}
            \norm{d^{l+1}W}_{L^1} \le \norm{d^{l+1}V}_{L^1} + \norm{d^{l+1}w}_{L^1}
        \end{align*}
        and following the same computations as that leading to \cref{eq:weak dynamics conv} above, we find
        \begin{align*}
            &\abs{\ \int_{\R_+\times \R^2}\varphi \rho_b \partial_t  \Tr{\gamma_b \Pi_{\le M}}} \\
                \le&\ \int_{\mathbb{R}_+} \norm{\varphi(t)}_{L^\infty} \abs{\partial_t  \Tr{\gamma_b \Pi_{\le M}}} \dd t
                \le \int_{\mathbb{R}_+} \norm{\varphi(t)}_{L^1 \cap W^{1, \infty}} \abs{\partial_t  \Tr{\gamma_b \Pi_{\le M}}} \dd t \\
                \le& C \norm{\varphi}_{L^1\prth{\R_+, L^1\cap W^{1, \infty}\prth{\R^2}}} \prth{\norm{V}_{W^{l+1,\infty}} + \norm{d^{l+1}V}_{L^1} + \norm{w}_{W^{l+1,\infty}} + \norm{d^{l+1}w}_{L^1}} \nonumber \\
                    &\prth{\abs{\Tr{\gamma_b(0) H_b(0)}} + \norm{V}_{L^\infty} + \norm{w}_{L^\infty}} \prth{l_b^{l-2 - \alpha\frac{l+3}{2}} +  l_b^{1 - \frac{\alpha}{2}} +l_b^{\frac{3}{2}\alpha-1}}
        \end{align*}

        Combining this with \cref{eq:dt drift norm} and \cref{eq:weak dynamics conv}  
        \begin{align*}
            &\abs{\int_{\R_+\times \R^2}\varphi \mathrm{DRIFT}_{\rho_{\gamma_b}}\prth{\rho_b}}
            \le C \frac{\norm{\varphi}_{L^1\prth{\R_+, L^1\cap W^{1, \infty}\prth{\R^2}}}}{\Tr{\gamma_b \Pi_{> M}}}\prth{\abs{\Tr{\gamma_b(0) H_b(0)}} + \norm{V}_{L^\infty} + \norm{w}_{L^\infty}} \nonumber \\
                &\prth{\norm{V}_{W^{l+1,\infty}} + \norm{d^{l+1}V}_{L^1} + \norm{w}_{W^{l+1,\infty}} + \norm{d^{l+1}w}_{L^1}}  \prth{l_b^{l-2 - \alpha\frac{l+3}{2}} + l_b^{2-2\alpha} + l_b^{\frac{3}{2}\alpha-1}}     
        \end{align*}
        taking $l=8$,
        \begin{align*}
            \alpha < 1 \implies l-2 - \alpha\frac{l+3}{2} \ge \frac{l-7}{2} = \frac{1}{2} \ge \frac{3}{2}\alpha - 1
        \end{align*}
        we can discard an error term and conclude with \cref{eq:norm exist}.
    \end{proof}

Next we turn to the 

    \begin{proof}[Proof of Theorem~\ref{th:V of HF}]

        Let $\varphi \in  C^\infty_c\prth{\R_+ \times \R^2}$ and choose $M$ according to Lemma~\ref{lem:sc summary}.

        \bigskip
        
        \noindent\textbf{Step 1: decomposition.} With an integration by parts,
        \begin{align}
            &\int_{\mathbb{R}^2} \varphi(0) \rho_{\gamma_b}(0) - \int_{\mathbb{R}_+ \times \mathbb{R}^2}\rho_{\gamma_b}\mathrm{DRIFT}_{\rho_{\gamma_b}}(\varphi) \nonumber\\
                =& \int_{\mathbb{R}^2} \varphi(0) \prth{\rho_{\gamma_b}(0) - \rho_{\gamma_b}^{sc, \le M}(0)}
                    + \int_{\mathbb{R}^2} \varphi(0) \rho_{\gamma_b}^{sc, \le M}(0) 
                    - \int_{\mathbb{R}_+ \times \mathbb{R}^2}\prth{\rho_{\gamma_b} - \rho_{\gamma_b}^{sc, \le M}}\mathrm{DRIFT}_{\rho_{\gamma_b}}(\varphi) \nonumber\\
                    &- \int_{\mathbb{R}_+ \times \mathbb{R}^2}\rho_{\gamma_b}^{sc, \le M} \mathrm{DRIFT}_{\rho_{\gamma_b}}(\varphi)   \nonumber\\
                =& \int_{\mathbb{R}^2} \varphi(0) \prth{\rho_{\gamma_b}(0) - \rho_{\gamma_b}^{sc, \le M}(0)}  
                    + \int_{\mathbb{R}_+ \times \mathbb{R}^2} \varphi \mathrm{DRIFT}_{\rho_{\gamma_b}}(\rho_{\gamma_b}^{sc, \le M}) \nonumber\\
                    &- \int_{\mathbb{R}_+ \times \mathbb{R}^2}\prth{\rho_{\gamma_b} - \rho_{\gamma_b}^{sc, \le M}}\mathrm{DRIFT}_{\rho_{\gamma_b}}(\varphi) \label{eq:vor decomp}
        \end{align}

    \bigskip
        
        \noindent\textbf{Step 2: third term in~\eqref{eq:vor decomp}.} For $t\in\R_+$, \eqref{eq:weak density conv} implies
        \begin{align}
            &\abs{\int_{\mathbb{R}^2} \prth{\rho_{\gamma_b}(t) - \rho_{\gamma_b}^{sc, \le M}(t)}\mathrm{DRIFT}_{\rho_{\gamma_b}}(\varphi)(t)}
                \le C\prth{\abs{\Tr{\gamma_b(0) H_b(0)}} + \norm{V}_{L^\infty} + \norm{w}_{L^\infty}}  \nonumber\\
                & \prth{\norm{\mathrm{DRIFT}_{\rho_{\gamma_b}}(\varphi)(t)}_{L^\infty} + \norm{\nabla\mathrm{DRIFT}_{\rho_{\gamma_b}}(\varphi)(t)}_{L^2}} l_b^{\frac{\alpha}{2}} \label{eq:third term start}
        \end{align}
        Hence we estimate
        \begin{align*}
            \norm{\mathrm{DRIFT}_{\rho_{\gamma_b}}(\varphi)(t)}_{L^\infty}
                =& \norm{\partial_t \varphi(t) + \nabla^\perp(V + w\star \rho_{\gamma_b(t)})\cdot \nabla \varphi(t)}_{L^\infty} \\
                \le& \norm{\partial_t \varphi(t)}_{L^\infty} + \prth{\norm{\nabla V}_{L^\infty} + \norm{\nabla w}_{L^\infty}} \norm{\nabla \varphi(t)}_{L^\infty}
        \end{align*}
        and
        \begin{align*}
            &\norm{\nabla \mathrm{DRIFT}_{\rho_{\gamma_b}}(\varphi)(t)}_{L^2} \\
                =& \norm{\partial_t \nabla \varphi(t) + \nabla\prth{\nabla^\perp(V + w\star \rho_{\gamma_b(t)})\cdot \nabla \varphi(t)}}_{L^2} \\
                \le& \norm{\partial_t \nabla \varphi(t)}_{L^2} + \prth{\norm{V}_{W^{2,\infty}} + \norm{W}_{W^{2,\infty}}} \norm{\nabla \varphi(t)}_{L^2} + \prth{\norm{d V}_{L^\infty} + \norm{dw}_{L^\infty}}\norm{d^2 \varphi(t)}_{L^2}
        \end{align*}
        so
        \begin{align*}
            &\int_{\R_+}\prth{\norm{\mathrm{DRIFT}_{\rho_{\gamma_b}}(\varphi)(t)}_{L^\infty} + \norm{\nabla\mathrm{DRIFT}_{\rho_{\gamma_b}}(\varphi)(t)}_{L^2}} \dd t \\
                \le& \prth{1 + \norm{V}_{W^{2, \infty}}+ \norm{w}_{W^{2, \infty}}} \\
                    &\int_{\R_+}\prth{\norm{\partial_t \varphi(t)}_{L^\infty}  + \norm{\nabla \varphi(t)}_{L^\infty} + \norm{\partial_t \nabla \varphi(t)}_{L^2} + \norm{\nabla \varphi(t)}_{L^2} + \norm{d^2 \varphi(t)}_{L^2}}\dd t \\
                \le& \prth{1 + \norm{V}_{W^{2, \infty}}+ \norm{w}_{W^{2, \infty}}} \norm{\varphi}_{W^{1, 1}\prth{\R_+, W^{1, \infty}\cap H^2\prth{\R^2}}}
        \end{align*}
        thus \eqref{eq:third term start} implies
        \begin{align}
            &\abs{\int_{\mathbb{R}_+ \times \mathbb{R}^2}\prth{\rho_{\gamma_b} - \rho_{\gamma_b}^{sc, \le M}}\mathrm{DRIFT}_{\rho_{\gamma_b}}(\varphi)}
                \le C \norm{\varphi}_{W^{1, 1}\prth{\R_+, W^{1, \infty}\cap H^2\prth{\R^2}}} \prth{1+\norm{V}_{W^{2, \infty}} + \norm{w}_{W^{2, \infty}}} \nonumber\\
                    &\prth{\abs{\Tr{\gamma_b(0) H_b(0)}} + \norm{V}_{L^\infty} + \norm{w}_{L^\infty}} l_b^{\frac{\alpha}{2}} \label{eq:density time part}
        \end{align}

    \bigskip
        
        \noindent\textbf{Step 3: conclusion.} Inserting \eqref{eq:weak density conv} for $\mu \coloneq \varphi(0)$, \eqref{eq:weak dynamics conv} and \eqref{eq:density time part} in \eqref{eq:vor decomp}, we get
        \begin{align*}
            \abs{\int_{\mathbb{R}^2} \varphi(0) \rho_{\gamma_b}(0) - \int_{\mathbb{R}_+ \times \mathbb{R}^2}\rho_{\gamma_b}\mathrm{DRIFT}_{\rho_{\gamma_b}}(\varphi)}
                \le C(\varphi, V, w) \prth{l_b^\frac{\alpha}{2} +  l_b^{2-2\alpha} + l_b^{\frac{3}{2}\alpha-1}}
        \end{align*}
        Finally,
        \begin{align*}
            \alpha \mapsto \min\prth{\dfrac{\alpha}{2}, 2-2\alpha, \dfrac{3}{2}\alpha - 1}
        \end{align*}
         is maximal at $2-2\alpha = -1 + \frac{3}{2}\alpha$, so we conclude by taking
        \begin{align*}
            \alpha \coloneq \dfrac{6}{7}.
        \end{align*}
    \end{proof}

    \section{Stability for perturbed classical flows}\label{sec:stab class}

    Our limit model, the drift equation~\eqref{eq:drift}, enjoys stability estimates with respect to the initial data, by transport arguments \`a la Dobrushin~\cite{Dobrushin-79} (see e.g.~\cite[Section~1.4]{Golse-13} for exposition of this material). In this section we include a small source term in this formalism, with the goal of treating the error obtained in Theorem~\ref{th:V of HF} in this way. This will provide estimates on the difference between the density of the solution to the quantum evolution~\eqref{eq:HF} and the solution of the limit equation, leading to the proof of Theorem~\ref{th:conv}.
    
    \subsection{A Dobrushin-type Estimate}\label{sec:dobrushin}

    We aim at comparing a solution to the drift equation~\eqref{eq:drift} to a solution to a similar equation with a small source term and a possibly different initial datum.
    
    Let $S_b$ be function of time and space. Let $\rho_b$ be the solution to
        \begin{align}
            &\partial_t \rho_b + \nabla^\perp V \cdot \nabla \rho_b + \nabla^\perp w \star \rho_b  \cdot \nabla \rho_b  = S_b \nonumber\\
            &\rho_b(t_0, \bullet) \eqcolon \rho_{b, 0} \in L^1 \prth{\mathbb{R}^2, \mathbb{R}_+}, \norm{ \rho_{b, 0}}_{L^1} = 1 \label{eq:transport pert}
        \end{align}
        and $\rho$ the solution to
        \begin{align}
            &\partial_t \rho + \nabla^\perp V  \cdot \nabla \rho + \nabla^\perp w \star \rho \cdot \nabla \rho  = 0 \nonumber\\
            &\rho(t_0, \bullet) \eqcolon \rho_0 \in L^1 \prth{\mathbb{R}^2, \mathbb{R}_+}, \norm{\rho_0}_{L^1} = 1\label{eq:transport}
        \end{align}
        Let $Z_\rho(t, t_0, \bz_0)$ be the flow defined by
        \begin{align}
            &\partial_t Z_\rho(t, t_0, \bz_0) 
                = \nabla^\perp V\prth{Z_\rho(t, t_0, \bz_0)} + \nabla^\perp w \star \rho (t, \bullet) \prth{Z_\rho(t, t_0, \bz_0)} \label{eq:Zrho}\\
            &Z_\rho(t_0 , t_0, \bz_0) = \bz_0 \label{eq:Zrho0}
        \end{align}
        We claim the following, which is classical for $S_b(t, \bz) \equiv 0$  
        
    \begin{proposition}[\textbf{A Dobrushin-type Estimate}] \label{prop:Dobrushin}\mbox{}\\
        With the notation above,
        \begin{align*}
            W_1\left(\rho_b(t), \rho(t)\right) \le  e^{2\prth{\norm{ V}_{W^{2,\infty}} +  \norm{ w}_{W^{2,\infty}}} \abs{t - t_0}} \prth{W_1(\rho_{b, 0}, \rho_0) + \mathcal{E}_{S_b}(t)}
        \end{align*}
        with 
        \begin{equation}\label{eq:Dob error}
            \mathcal{E}_{S_b}(t) \coloneq \abs{\int_{t_0}^t \int_{t_0}^\tau \iint_{\mathbb{R}^2\times \mathbb{R}^2} \nabla^\perp w\prth{Z_{\rho_b}\left( \tau, t_0, \by \right) -  Z_{\rho_b}(\tau, \nu, \bx )} S_b\prth{\nu, \bx} \rho_{b, 0}(\by) \dd \tau  \dd \nu \dd \bx \dd \by  }
        \end{equation}
    \end{proposition}

    We will use characteristics for the above equations. For a general, time-dependent, potential $W \in L^\infty\prth{\mathbb{R}_+, W^{2, \infty}\prth{\mathbb{R}^2}}$ the PDE 
    \begin{align*}
        \partial_t \rho(t, \bz) + \nabla^\perp W(t, \bz) \cdot \nabla_\bz \rho(t, \bz) = 0
    \end{align*}
    is a transport equation with velocity field $ \nabla^\perp W$, thus we define the flow
    \begin{align*}
        Z: \matrx{&\mathbb{R}\times\mathbb{R}\times \mathbb{R}^2 &\to &\mathbb{R}^2 \\ &t, t_0, \bz_0 &\mapsto &Z(t, t_0, \bz_0)}
    \end{align*}
    as the unique (by the Cauchy-Lipschitz theorem) solution of the ODEs
    \begin{align}
        &\partial_t Z(t, t_0, \bz_0) = \nabla^\perp W(t, Z(t, t_0, \bz_0)) \nonumber \\
        & Z(t_0 , t_0, \bz_0) = \bz_0 \label{eq:flow def}
    \end{align}
    We denote $Z(t, t_0, \bullet)_* \sigma$ the push-forward of a measure $\sigma$ by the flow. We then have the classical
    
    \begin{lemma}[\textbf{Characteristics}]\label{lem:characteristics}\mbox{}\\
        Let 
        \begin{equation}
            \rho(t, \bullet) \coloneq Z(t, t_0, \bullet)_* \rho_0 + \int_{t_0}^t Z(t, \tau, \bullet)_* S(\tau, \bullet) \dd \tau  \label{eq:pushforward rho}
        \end{equation}
        then 
        \begin{align*}
            &\partial_t \rho(t, \bz) + \nabla^\perp W(t, \bz) \cdot \nabla_\bz \rho(t, \bz) = S(t, \bz) \\
            &\rho(t_0, \bullet) = \rho_0
        \end{align*}
    \end{lemma}
    
    \begin{proof}
        Let $t, \tau \in \mathbb{R}$ the flow satisfies
        \begin{equation}
            Z(t, \tau, Z(\tau, t, \bullet)) = \text{Id}_{\mathbb{R}^2} \label{eq:inverse flow}
        \end{equation}
        so
        \begin{align*}
            Z(t, \tau, \bullet)^{-1} = Z(\tau, t, \bullet)
        \end{align*}
        Thus ~\eqref{eq:pushforward rho} can we rewritten 
        \begin{align*}
            \rho(t, \bz) = \rho_0 \left(Z(t_0, t, \bz) \right) + \int_{t_0}^t S\left(\tau, Z\left(\tau, t, \bz\right)\right) \dd \tau 
        \end{align*}
        Now we compute
        \begin{align*}
            \partial_t \rho(t, \bz) 
                =& \partial_{t_0}Z(t_0, t, \bz) \cdot \nabla \rho_0\prth{Z(t_0, t, \bz)} + S\prth{t, Z\left(t, t, \bz \right)}  \\
                & + \int_{t_0}^t \partial_{t_0} Z(\tau, t, \bz) \cdot \nabla_\bz S\prth{\tau, Z(\tau, t, \bz)} \dd \tau  \\
                =& S(t, \bz) + \partial_{t_0}Z(t_0, t, \bz) \cdot \nabla \rho_0\prth{Z(t_0, t, \bz)} + \int_{t_0}^t \partial_{t_0} Z(\tau, t, \bz) \cdot \nabla_\bz S\prth{\tau, Z(\tau, t, \bz)} \dd \tau  
        \end{align*}
        and
        \begin{align*}
            \nabla^\perp W(t, \bz) \cdot \nabla_\bz \rho(t, \bz) 
                =& d\rho(t, \bullet)_\bz \nabla^\perp W(t, \bz) \\
                =& d {\rho_0}_{Z(t_0, t, \bz)} d Z(t_0, t, \bullet)_\bz \nabla^\perp W(t, \bz) \\
                &+ \int_{t_0}^t dS(\tau, \bullet)_{Z(\tau, t, \bz)} dZ(\tau, t, \bullet)_\bz \nabla^\perp W(t, \bz) \dd \tau  \\
                =& \nabla \rho_0(Z(t_0, t, \bz)) \cdot d Z(t_0, t, \bullet)_\bz \nabla^\perp W(t, \bz) \\
                +& \int_{t_0}^t \nabla_\bz S(\tau, Z(\tau, t, \bz)) \cdot dZ(\tau, t, \bullet)_\bz \nabla^\perp W(t, \bz) \dd \tau  
        \end{align*}
        Put together, the above give
        \begin{align}
            &\partial_t \rho(t, \bz) + \nabla^\perp W(t, \bz) \cdot \nabla_\bz \rho(t, \bz) 
                = S(t, \bz) \nonumber\\
                &+ \nabla \rho_0(Z(t_0, t, \bz)) \cdot \prth{ \partial_{t_0}Z(t_0, t, \bz) + d Z(t_0, t, \bullet)_\bz \nabla^\perp W(t, \bz)} \nonumber\\
                &+  \int_{t_0}^t \nabla_\bz S(\tau, Z(\tau, t, \bz)) \cdot \prth{\partial_{t_0} Z(\tau, t, \bz) + dZ(\tau, t, \bullet)_\bz \nabla^\perp W(t, \bz)} \dd \tau  
                \label{eq:EDP source and errors}
        \end{align}
        But ~\eqref{eq:inverse flow} also implies
        \begin{align*}
            &dZ(t, \tau, \bullet)_{Z(\tau, t, \bz)} dZ(\tau, t, \bullet)_\bz = \text{Id}_{\mathbb{R}^2} 
            \\
            & \dfrac{d}{dt} Z(t, \tau, Z(\tau, t, \bz)) = \partial_t Z(t, \tau, Z(\tau, t, \bz)) +  dZ(t, \tau, \bullet)_{Z(\tau, t, \bz)}\partial_{t_0} Z(\tau, t, \bz) = 0 
        \end{align*}
        therefore, combining with~\eqref{eq:flow def} and~\eqref{eq:inverse flow} leads to
        \begin{align*}
            \partial_{t_0} Z(\tau, t, \bz) 
                =& - \prth{dZ(t, \tau, \bullet)_{Z(\tau, t, \bz)}}^{-1} \partial_t Z(t, \tau, Z(\tau, t, \bz))
                = - dZ(\tau, t, \bullet)_\bz \partial_t Z(t, \tau, Z(\tau, t, \bz)) \\
                =&  - dZ(\tau, t, \bullet)_\bz \nabla^\perp W\prth{t, Z\prth{t, \tau, Z(\tau, t, \bz)}}
                = - dZ(\tau, t, \bullet)_\bz \nabla^\perp W(t, \bz)
        \end{align*}
        so we conclude by seeing that the last two terms in~\eqref{eq:EDP source and errors} are null.
    \end{proof}

    Next we introduce couplings as in~\eqref{eq:couplings} to state the 

    \begin{lemma}[\textbf{Coupling}]\label{lem:Dobrushin}\mbox{}\\
        Let $\rho_b$ and $\rho$ be as in~\eqref{eq:transport pert} and~\eqref{eq:transport}. Let $\pi_b$ we a coupling between $\rho_{b, 0}$ and $\rho_0$ and introduce the notation
        \begin{align*}
            D_{\pi_b}(\tau) \coloneq \iintr_{\mathbb{R}^2 \times \mathbb{R}^2} \abs{Z_{\rho_b}\left( \tau, t_0, \bx \right) - Z_\rho\left( \tau, t_0, \by \right)} \dd \pi_b(\bx,\by)
        \end{align*}
        then
        \begin{align*}
            D_{\pi_b}(t) \le \prth{D_{\pi_b}(t_0) + \mathcal{E}_{S_b}(t)}e^{2\prth{\norm{V}_{W{2,\infty}} +  \norm{w}_{W^{2,\infty}}} \abs{t - t_0}}
        \end{align*}
        with the same notation as in~\eqref{eq:Dob error}.        
    \end{lemma}
    
    \begin{proof}
        Starting from \eqref{eq:Zrho}, \eqref{eq:Zrho0},
        \begin{align*}
            &Z_{\rho_b}\left( t, t_0, \bz_{b, 0} \right) - Z_\rho(t, t_0, \bz_0) - \prth{\bz_{b, 0} - \bz_0 } \\
                =& \int_{t_0}^t \nabla^\perp \prth{V + w \star \rho_b (\tau, \bullet)} \prth{Z_{\rho_b}\left(\tau, t_0, \bz_{b, 0}\right)} \dd \tau  
                    - \int_{t_0}^t \nabla^\perp\prth{V + w \star \rho (\tau, \bullet)} \prth{Z_\rho\left( \tau, t_0, \bz_0 \right)} \dd \tau  \\
                =& \int_{t_0}^t \prth{\nabla^\perp V \prth{Z_{\rho_b}\left(\tau, t_0, \bz_{b, 0}\right)} - \nabla^\perp V \prth{Z_\rho\left( \tau, t_0, \bz_0 \right)}}\dd \tau  \\
                    +& \int_{t_0}^t \int_{\mathbb{R}^2} \nabla^\perp w\prth{Z_{\rho_b}\left(\tau, t_0, \bz_{b, 0}\right) - \bx}  \rho_b (\tau, \bx) \dd \tau   \dd \bx 
                    - \int_{t_0}^t \int_{\mathbb{R}^2} \nabla^\perp w \prth{Z_\rho\left( \tau, t_0, \bz_0 \right) - \by} \rho (\tau, \by)  \dd \tau  \dd \by
        \end{align*}
        Using Lemma~\ref{lem:characteristics}, 
        \begin{equation*}
            \rho_b(\tau, \bx) = \rho_{b, 0}\prth{Z_{\rho_b}(t_0, \tau, \bx)} + \int_{t_0}^\tau S_b\prth{\nu, Z_{\rho_b}(\nu, \tau, \bx)}\dd \nu, \quad \rho(\tau, \by) = \rho_0\prth{Z_\rho(t_0, \tau, \by)}
        \end{equation*}
        and inserting this in the above leads to
        \begin{align*}
            &Z_{\rho_b}\left( t, t_0, \bz_{b, 0} \right) - Z_\rho(t, t_0, \bz_0) - (\bz_{b, 0} - \bz_0) \\
                =& \int_{t_0}^t \prth{\nabla^\perp V \prth{Z_{\rho_b}\left(\tau, t_0, \bz_{b, 0}\right)} - \nabla^\perp V \prth{Z_\rho\left( \tau, t_0, \bz_0 \right)}}\dd \tau  \\
                    &+ \int_{t_0}^t \int_{\mathbb{R}^2} \nabla^\perp w\prth{Z_{\rho_b}\left(\tau, t_0, \bz_{b, 0}\right) - \bx}  \rho_{b, 0}\prth{Z_{\rho_b}(t_0, \tau, \bx)} \dd \tau   \dd \bx \\
                    &- \int_{t_0}^t \int_{\mathbb{R}^2} \nabla^\perp w \prth{Z_\rho\left( \tau, t_0, \bz_0 \right) - \by} \rho_0\prth{Z_\rho(t_0, \tau, \by)}  \dd \tau  \dd \by\\
                    &+ \int_{t_0}^t \int_{t_0}^\tau \int_{\mathbb{R}^2} \nabla^\perp w\prth{Z_{\rho_b}\left(\tau, t_0, \bz_{b, 0}\right) - \bx}   S_b\prth{\nu, Z_{\rho_b}(\nu, \tau, \bx)} \dd \tau  \dd \nu \dd \bx 
        \end{align*} 
        Since the flow is divergence free it preserves volume, thus with the changes of variable 
        $$Z_{\rho_b}(t_0, \tau, \bx) \mapsto \bx, \quad Z_\rho(t_0, \tau, \by) \mapsto \by$$ 
        and again $Z_{\rho_b}(\nu, \tau, \bx) \mapsto \bx$ for the last integral,
        \begin{align*}
            &Z_{\rho_b}\left( t, t_0, \bz_{b, 0} \right) - Z_\rho(t, t_0, \bz_0) - (\bz_{b, 0} - \bz_0) \\
                =& \int_{t_0}^t \prth{\nabla^\perp V \prth{Z_{\rho_b}\left(\tau, t_0, \bz_{b, 0}\right)} - \nabla^\perp V \prth{Z_\rho\left( \tau, t_0, \bz_0 \right)}}\dd \tau  \\
                    &+ \int_{t_0}^t \int_{\mathbb{R}^2} \nabla^\perp w\prth{Z_{\rho_b}\left(\tau, t_0, \bz_{b, 0}\right) - Z_{\rho_b}\left( \tau, t_0, \bx \right)}  \rho_{b, 0}(\bx) \dd \tau   \dd \bx \\
                    &- \int_{t_0}^t \int_{\mathbb{R}^2} \nabla^\perp w \prth{Z_\rho\left( \tau, t_0, \bz_0 \right) - Z_\rho\left( \tau, t_0, \by \right)} \rho_0(\by)  \dd \tau  \dd \by\\
                    &+ \int_{t_0}^t \int_{t_0}^\tau \int_{\mathbb{R}^2} \nabla^\perp w\prth{Z_{\rho_b}\left(\tau, t_0, \bz_{b, 0}\right) -  Z_{\rho_b}(\tau, \nu, \bx )}   S_b\prth{\nu, \bx} \dd \tau  \dd \nu \dd \bx 
        \end{align*}
        So
        \begin{align*}
            &Z_{\rho_b}\left( t, t_0, \bz_{b, 0} \right) - Z_\rho(t, t_0, \bz_0) - (\bz_{b, 0} - \bz_0) \\
                =& \int_{t_0}^t \prth{\nabla^\perp V \prth{Z_{\rho_b}\left(\tau, t_0, \bz_{b, 0}\right)} - \nabla^\perp V \prth{Z_\rho\left( \tau, t_0, \bz_0 \right)}}\dd \tau  \\
                    & \int_{t_0}^t \iintr_{\mathbb{R}^2\times \mathbb{R}^2} \prth{\nabla^\perp w\prth{Z_{\rho_b}\left(\tau, t_0, \bz_{b, 0}\right) - Z_{\rho_b}\left( \tau, t_0, \bx \right)} -  \nabla^\perp w \prth{Z_\rho\left( \tau, t_0, \bz_0 \right) - Z_\rho\left( \tau, t_0, \by \right)}}  \\
                    &\dd \tau  \dd \pi_b(\bx,\by) + \int_{t_0}^t \int_{t_0}^\tau \int_{\mathbb{R}^2} \nabla^\perp w\prth{Z_{\rho_b}\left(\tau, t_0, \bz_{b, 0}\right) -  Z_{\rho_b}(\tau, \nu, \bx )}   S_b\prth{\nu, \bx} \dd \tau  \dd \nu \dd \bx 
        \end{align*}
        Next, using
        \begin{align*}
            &\abs{\nabla^\perp w\prth{Z_{\rho_b}\left(\tau, t_0, \bz_{b, 0}\right) - Z_{\rho_b}\left( \tau, t_0, \bx \right)} - \nabla^\perp w \prth{Z_\rho\left( \tau, t_0, \bz_0 \right) - Z_\rho\left( \tau, t_0, \by \right)}} \\
                \le& \abs{\nabla^\perp w\prth{Z_{\rho_b}\left(\tau, t_0, \bz_{b, 0}\right) - Z_{\rho_b}\left( \tau, t_0, \bx \right)} - \nabla^\perp w \prth{Z_{\rho_b}\left(\tau, t_0, \bz_{b, 0}\right) -  Z_\rho\left( \tau, t_0, \by \right)}} \\
                    &+ \abs{\nabla^\perp w \prth{Z_{\rho_b}\left(\tau, t_0, \bz_{b, 0}\right) -  Z_\rho\left( \tau, t_0, \by \right)} - \nabla^\perp w \prth{Z_\rho\left( \tau, t_0, \bz_0 \right) - Z_\rho\left( \tau, t_0, \by \right)}} \\
                \le& \norm{w}_{W^{2,\infty}} \prth{\abs{Z_{\rho_b}\left( \tau, t_0, \bx \right) - Z_\rho\left( \tau, t_0, \by \right)} + \abs{Z_{\rho_b}\left(\tau, t_0, \bz_{b, 0}\right) - Z_\rho\left( \tau, t_0, \bz_0 \right)}}
        \end{align*}
        we obtain
        \begin{align*}
            &\abs{Z_{\rho_b}\left( t, t_0, \bz_{b, 0} \right) - Z_\rho(t, t_0, \bz_0)}
                \le \abs{\bz_{b, 0} - \bz_0} + \norm{V}_{W^{2,\infty}} \int_{t_0}^t \abs{Z_\rho\left( \tau, t_0, \bz_0 \right) - Z_{\rho_b}\left(\tau, t_0, \bz_{b, 0}\right)}\dd \tau \\
                    &+ \norm{w}_{W^{2,\infty}} \int_{t_0}^t \iintr_{\mathbb{R}^2\times \mathbb{R}^2} \abs{Z_{\rho_b}\left( \tau, t_0, \bx \right) - Z_\rho\left( \tau, t_0, \by \right)} \dd \pi_b(\bx,\by) \dd \tau  \\
                    &+ \norm{w}_{W^{2,\infty}} \int_{t_0}^t \abs{Z_{\rho_b}\left(\tau, t_0, \bz_{b, 0}\right) - Z_\rho\left( \tau, t_0, \bz_0 \right)} \dd \tau   \\
                    &+ \int_{t_0}^t \int_{t_0}^\tau \int_{\mathbb{R}^2} \abs{\nabla^\perp w\prth{Z_{\rho_b}\left(\tau, t_0, \bz_{b, 0}\right) -  Z_{\rho_b}(\tau, \nu, \bx )} S_b\prth{\nu, \bx}} \dd \tau  \dd \nu \dd \bx 
        \end{align*}
        and integrating against $\dd \pi_b(\bz_{b, 0}, \bz_0)$ we obtain
        \begin{align*}
            D_{\pi_b}(t)
                \le& D_{\pi_b}(0) + \prth{\norm{V}_{W^{2,\infty}} + 2 \norm{w}_{W^{2,\infty}}} \int_{t_0}^t D_{\pi_b}(\tau) \dd \tau  \\
                    &+ \abs{\int_{t_0}^t \int_{t_0}^\tau \iintr_{\mathbb{R}^2\times \mathbb{R}^2} \nabla^\perp w\prth{Z_{\rho_b}\left(\tau, t_0, \bz_{b, 0}\right) -  Z_{\rho_b}(\tau, \nu, \bx )} S_b\prth{\nu, \bx} \rho_{b, 0}\left( \bz_{b, 0} \right) \dd \tau  \dd \nu \dd \bx \dd \bz _{b, 0}} \\
                \le& D_{\pi_b}(0) + 2\prth{\norm{V}_{W^{2,\infty}} +  \norm{w}_{W^{2,\infty}}} \int_{t_0}^t D_{\pi_b}(\tau) \dd \tau  \\
                    &+  \abs{\int_{t_0}^t \int_{t_0}^\tau \iintr_{\mathbb{R}^2\times \mathbb{R}^2} \nabla^\perp w\prth{Z_{\rho_b}\left( \tau, t_0, \by \right) -  Z_{\rho_b}(\tau, \nu, \bx )} S_b\prth{\nu, \bx} \rho_{b, 0}(\by) \dd \tau  \dd \nu \dd \bx \dd \by  }
        \end{align*}
    \end{proof}

    We conclude this subsection by giving the 
    
    \begin{proof}[Proof of Proposition~\ref{prop:Dobrushin}]
        Let $\pi_b \in \Gamma(\rho_{b, 0}, \rho_0)$ and define
        \begin{align*}
            &\phi_t(\bx,\by) \coloneq \prth{Z_{\rho_b}(t, t_0, \bx), Z_\rho(t, t_0, \by)} \\
            &\pi_b(t) \coloneq {\phi_t}_* \pi_b
        \end{align*}
        Then $\forall A \subseteq \mathbb{R}^2$ Borel set,
        \begin{align*}
            \pi_b(t)(\mathbb{R}^2, A) 
                =& \pi_b\prth{Z_{\rho_b}(t_0, t, \mathbb{R}^2), Z_\rho(t_0, t, A)}
                = \pi_b(\mathbb{R}^2, Z_\rho(t_0, t, A))
                = \rho_0(Z_\rho(t_0, t, A)) \\
                =& Z_\rho(t, t_0, \bullet)_* \rho_0 (A)
                = \rho(t, A)
        \end{align*}
        and similarly for the second variable thus $\pi_b(t) $ is a coupling for $\rho_b(t)$ and $\rho(t)$. Exchanging $t$ and $t_0$ and using \cref{lem:Dobrushin} we get
        \begin{align*}
            W_1(\rho_b(t), \rho(t)) 
                =& \inf_{\pi_b \in \Gamma(\rho_b(t), \rho(t))} \iintr_{\mathbb{R}^2 \times \mathbb{R}^2} \abs{x-y} \dd \pi_b (\bx,\by) \\
                =& \inf_{\pi_b \in \Gamma(\rho_{b, 0}, \rho_0)} \iintr_{\mathbb{R}^2 \times \mathbb{R}^2} \abs{x-y} d{\phi_t}_* \pi_b (\bx,\by) \\
                =& \inf_{\pi_b \in \Gamma(\rho_{b, 0}, \rho_0)} \iintr_{\mathbb{R}^2 \times \mathbb{R}^2} \abs{Z_{\rho_b}(t, t_0, \bx) -  Z_\rho(t, t_0, \by)} d \pi_b (\bx,\by) \\
                =& \inf_{\pi_b \in \Gamma(\rho_{b, 0}, \rho_0)} D_{\pi_b}(t)
                \le e^{2\prth{\norm{w}_{W^{2,\infty}} +  \norm{w}_{W^{2,\infty}}}\abs{t - t_0}} \prth{\inf_{\pi_b \in \Gamma(\rho_{b, 0}, \rho_0)} D_{\pi_b}(0) + \mathcal{E}_{S_b}(t)}\\
                =& e^{2\prth{\norm{V}_{W^{2,\infty}} +  \norm{w}_{W^{2,\infty}}}\abs{t - t_0}} \prth{W_1(\rho_{b, 0}, \rho_0) + \mathcal{E}_{S_b}(t)}
        \end{align*}
    \end{proof}

    \subsection{Application: proof of Theorem~\ref{th:conv}}

    As in Section~\ref{sec:sc EDP} it is convenient to first consider the dynamics of the truncated semi-classical density:
    
    \begin{proposition}[\textbf{Convergence of the semi-classical density in Wasserstein metric}]\label{prop:sc theorem}\mbox{}\\
        Let $\frac{2}{3} < \alpha < 1$. Under the same assumptions as in Theorem~\ref{th:V of HF}, with in addition
        \begin{align*}
            \nabla w \in L^1\prth{\R^2}, \quad w\in H^2\prth{\R^2}, \quad V, w \in W^{9, \infty}\prth{\R^2}, \quad d^9V, d^9w \in L^1\prth{\R^2}
        \end{align*}
        if $\rho \in L^\infty\prth{\R_+, L^1\prth{\R^2}}$ solves the drift equation \eqref{eq:drift} then $\exists M = \mathcal{O}\prth{l_b^{-\alpha}}$  such that
        \begin{align*}
            W_1\prth{\rho_b(t), \rho(t)}
                \le\ e^{2\prth{\norm{w}_{W^{2,\infty}} +  \norm{V}_{W^{2,\infty}}}t}\prth{W_1\prth{\rho_b(0), \rho(0)} + C_2(t, V, w)\prth{l_b^{2-2\alpha} + l_b^{\frac{3}{2}\alpha-1}}}
        \end{align*}
        with $\rho_b$ as defined in \cref{eq:final rho} and
        \begin{align*}
            &C_2(t, V, w) 
                \coloneq Ct^2\prth{\norm{w}_{H^2}+\norm{V}_{W^{9,\infty}} + \norm{d^9V}_{L^1} + \norm{w}_{W^{9,\infty}} + \norm{d^9w}_{L^1}} \\
                &\prth{\norm{\nabla w}_{L^1} + \norm{w}_{W^{2,\infty}}e^{\prth{\norm{V}_{W^{2,\infty}} +  \norm{w}_{W^{2,\infty}}} t}}
            \prth{\abs{\Tr{\gamma_b(0) H_b(0)}} + \norm{V}_{L^\infty} + \norm{w}_{L^\infty}}
        \end{align*}
    \end{proposition}

    \begin{proof}
        Our goal is to apply Proposition~~\ref{prop:Dobrushin} with $t_0 = 0$ and
        \begin{align*}
            S_b \coloneq \mathrm{DRIFT}_{\rho_b}\prth{\rho_b}
        \end{align*}
        To this end, we need to estimate the error term $\mathcal{E}_{S_b}(t)$ coming from~\eqref{eq:Dob error}. We define
        \begin{align*}
            \varphi(t, \bz) \coloneq \int_{0}^T  \int_{\mathbb{R}^2} \nabla^\perp w\prth{Z_{\rho_b}(\tau, 0, \bx) -  Z_{\rho_b}(\tau,t, \bz)}\rho_b(0, \bx) \dd \tau  \dd \bx 
        \end{align*}
        so that
        \begin{align}
            \mathcal{E}_{S_b}(T) 
                =& \abs{\int_0^T \int_0^\tau  \iintr_{\mathbb{R}^2\times \mathbb{R}^2}\nabla^\perp w\prth{Z_{\rho_b}(\tau, 0, \bx) - Z_{\rho_b}(\tau, t, \bz)}S_b(t, \bz)\rho_b(0, \bx)\dd \tau  \dd t \dd \bz  \dd \bx } \\
                \le& \abs{\ \int_{\R_+\times\R^2}\varphi \mathbb{1}_{[0, T]} S_b}\nonumber\\
                \leq& \abs{\ \int_{\R_+\times\R^2}\varphi \mathbb{1}_{[0, T]} \mathrm{DRIFT}_{\rho_{\gamma_b}}\prth{\rho_b}}
                    + \abs{\ \int_{\R_+\times\R^2}\varphi \mathbb{1}_{[0, T]} \nabla^\perp w \star\prth{\rho_{\gamma_b} - \rho_b} \cdot \nabla \rho_b }\nonumber\\
                \eqcolon& \ \mathcal{E}_{S_b} ^1 (T) + \mathcal{E}_{S_b} ^2 (T).\label{eq:Dob error split}
        \end{align}
        
        \medskip
        
    \noindent\textbf{Step 1: we estimate $\varphi \in L^1\prth{[0, T], L^1\prth{\mathbb{R}^2} \cap W^{1, \infty}\prth{\R^2}}$.} Let $t\in[0, T]$. With the changes of variable $Z_{\rho_b}(\tau,t, \bz)\mapsto \bz$,
        \begin{align}
            \norm{\varphi(t, \bullet)}_{L^1} 
                \le& \int_{0}^T  \iintr_{\mathbb{R}^2 \times \R^2}\abs{\nabla^\perp w\prth{Z_{\rho_b}(\tau, 0, \bx) -  Z_{\rho_b}(\tau,t, \bz)} \rho_b(0, \bx)} \dd \tau  \dd \bx \dd \bz  \nonumber \\
                =& \int_{0}^T  \iintr_{\mathbb{R}^2 \times \R^2}\abs{\nabla^\perp w\prth{Z_{\rho_b}(\tau, 0, \bx) -  \bz} \rho_b(0, \bx)} \dd \tau  \dd \bx  \dd \bz  
                \le T \norm{\nabla w}_{L^1} \label{eq:phit L1}
        \end{align}
        by performing the $\bz$ integral first. Moreover
        \begin{align}
            \abs{\varphi(t, \bz)} \le T \norm{\nabla w}_{L^\infty} \label{eq:phit Linf}
        \end{align}
        and
        \begin{equation}
            \abs{\nabla_\bz \varphi(t, \bz)} \le \norm{w}_{W^{2,\infty}} \int_{0}^T \abs{d Z_{\rho_b}(\tau,t, \bullet)_\bz} \dd \tau  .
            \label{eq:phit nabla Linfty start}
        \end{equation}
        But, using 
        \begin{align*}
            Z_{\rho_b}(\tau,t, \bz) = \bz + \int_t^\tau \nabla^\perp\prth{V + w\star \rho_b(s, \bullet)}\prth{Z_{\rho_b}(s,t, \bz)} \dd s
        \end{align*}
        we get
        \begin{align*}
            dZ_{\rho_b}(\tau,t, \bullet)_\bz = \text{Id}_{\mathbb{R}^2} + \int_t^\tau d \nabla^\perp\prth{V + w\star \rho_b(s, \bullet)}_{Z_{\rho_b}(s,t, \bz)}d Z_{\rho_b}(s,t, \bullet)_\bz \dd s
        \end{align*}
        so
        \begin{align*}
            \abs{dZ_{\rho_b}(\tau,t, \bullet)_\bz} \le \abs{\text{Id}_{\mathbb{R}^2}} + \prth{\norm{V}_{W^{2,\infty}} +  \norm{w}_{W^{2,\infty}}}\int_t^\tau  \abs{dZ_{\rho_b}(s,t, \bullet)_\bz} \dd s
        \end{align*}
        Hence, applying Gr\"onwall's lemma,
        \begin{align*}
            \abs{dZ_{\rho_b}(\tau,t, \bullet)_\bz} \le \sqrt{2}e^{\prth{\norm{V}_{W^{2,\infty}} +  \norm{w}_{W^{2,\infty}}} \abs{\tau - t}}
        \end{align*}
        With~\eqref{eq:phit nabla Linfty start}, we conclude that
        \begin{align}
            \abs{\nabla_\bz \varphi(t, \bz)} 
                \le& \sqrt{2} \norm{w}_{W^{2,\infty}} \int_{0}^T e^{\prth{\norm{V}_{W^{2,\infty}} +  \norm{w}_{W^{2,\infty}}} \abs{\tau - t}} \dd \tau  \nonumber\\
                \le& \sqrt{2} \norm{w}_{W^{2,\infty}}T e^{\prth{\norm{V}_{W^{2,\infty}} +  \norm{w}_{W^{2,\infty}}} T} \label{eq:phit dlinf}
        \end{align}
        Collecting \cref{eq:phit L1}, \cref{eq:phit Linf} and \cref{eq:phit dlinf} we find
        \begin{align*}
            \norm{\varphi(t, \bullet)}_{L^1\cap W^{1, \infty}\prth{\R^2}} 
                \le C T\prth{\norm{\nabla w}_{L^1} + \norm{w}_{W^{2,\infty}}e^{\prth{\norm{V}_{W^{2,\infty}} +  \norm{w}_{W^{2,\infty}}} T}}
        \end{align*}
        and thus
        \begin{align}\label{eq:phi bound}
            \norm{\varphi \mathbb{1}_{[0, T]}}_{L^1\prth{\R_+, L^1\cap W^{1, \infty}\prth{\R^2}}} 
                =& \norm{\varphi}_{L^1\prth{[0, T], L^1\cap W^{1, \infty}\prth{\R^2}}} \nonumber\\
                \le& C T^2\prth{\norm{\nabla w}_{L^1} + \norm{w}_{W^{2,\infty}}e^{\prth{\norm{V}_{W^{2,\infty}} +  \norm{w}_{W^{2,\infty}}} T}}
        \end{align}
                
                \medskip
                
    \noindent\textbf{Step 2: bound on $\mathcal{E}_{S_b}^2(T)$.} We choose $M$ according to Lemma~\ref{lem:sc summary}. Integrating by parts and using the symmetry of $w$, 
        \begin{align*}
            \mathcal{E}_{S_b} ^2 (T) 
                = \abs{\ \int_{\R_+\times\R^2} \mathbb{1}_{[0, T]} \nabla \varphi \cdot  \nabla^\perp w \star\left( \rho_{\gamma_b} - \rho_b \right) \rho_b}
                = \abs{\ \int_{\R_+\times\R^2} \mathbb{1}_{[0, T]} \left( \rho_{\gamma_b} - \rho_b \right) \left(\nabla \varphi \rho_b\right) \overset{.}{\star} \nabla^\perp w}.
        \end{align*}
        Then, using~\eqref{eq:weak density conv bis} with 
        $$ \mu (t,x) \coloneq \left(\nabla \varphi \rho_b\right)(t, \bullet) \overset{.}{\star} \nabla^\perp w (\bx) \coloneq \int \rho_b(t, \by) \nabla \varphi (t, \by) \cdot \nabla^\perp w (\bx - \by)\dd \by  $$
        leads to 
        \begin{align*}
         \mathcal{E}_{S_b} ^2 (T)
            \leq C \int_{0}^T \prth{\norm{\mu(t)}_{L^\infty} + \norm{\nabla \mu(t)}_{L^2}} \dd t \prth{\abs{\Tr{\gamma_b(0) H_b(0)}} + \norm{V}_{L^\infty} + \norm{w}_{L^\infty}}l_b^{\frac{\alpha}{2}}.
        \end{align*} 
        With Young's convolution inequality,
        \begin{align*}
            &\norm{\mu(t)}_{L^\infty} \le \norm{\nabla \varphi(t)}_{L^\infty} \norm{\nabla w}_{L^\infty} \\
            & \norm{\nabla \mu(t)}_{L^2} \le \norm{\left(\nabla \varphi(t) \rho^{sc, \le M}_{\gamma_b(t)}\right)}_{L^1} \norm{ w}_{H^2}
            \le \norm{\nabla\varphi(t)}_{L^\infty}\norm{w}_{H^2}
        \end{align*}
        so
        \begin{align}\label{eq:Dob error 2}
         \mathcal{E}_{S_b} ^2 (T) \leq& C \norm{\varphi\mathbb{1}_{[0, T]}}_{L^1\prth{\R_+, W^{1,\infty}\prth{\R^2}}} \prth{\norm{\nabla w}_{L^\infty}+\norm{w}_{H^2}} \nonumber \\ &\prth{\abs{\Tr{\gamma_b(0) H_b(0)}} + \norm{V}_{L^\infty} + \norm{w}_{L^\infty}}l_b^{\frac{\alpha}{2}}.
        \end{align}
        
        \medskip
        
    \noindent\textbf{Step 3: conclusion.} There remains to estimate the error term $\mathcal{E}_{S_b}^1(T)$ from \cref{eq:Dob error split}. With \eqref{eq:normalized drift}, 
    \begin{align*}
        \mathcal{E}_{S_b}^1(T) \le&\ C \norm{\mathbb{1}_{\sbra{0, T}}\varphi}_{L^1\prth{\R_+, L^1\cap W^{1, \infty}\prth{\R^2}}}\prth{\abs{\Tr{\gamma_b(0) H_b(0)}} + \norm{V}_{L^\infty} + \norm{w}_{L^\infty}} \nonumber \\
                &\prth{\norm{V}_{W^{9,\infty}} + \norm{d^9V}_{L^1} + \norm{w}_{W^{9,\infty}} + \norm{d^9w}_{L^1}}  \prth{l_b^{2-2\alpha} + l_b^{\frac{3}{2}\alpha-1}}  
    \end{align*}    
    We combine this with~\eqref{eq:Dob error 2} and insert the resulting bound in~\eqref{eq:Dob error split}:
    \begin{align*}
        &\mathcal{E}_{S_b}(T) \le C \norm{\mathbb{1}_{\sbra{0, T}}\varphi}_{L^1\prth{\R_+, L^1\cap W^{1, \infty}\prth{\R^2}}}\prth{\abs{\Tr{\gamma_b(0) H_b(0)}} + \norm{V}_{L^\infty} + \norm{w}_{L^\infty}} \nonumber \\                &\prth{\norm{w}_{H^2}+\norm{V}_{W^{9,\infty}} + \norm{d^9V}_{L^1} + \norm{w}_{W^{9,\infty}} + \norm{d^9w}_{L^1}}  \prth{ l_b^{2-2\alpha} + l_b^{\frac{3}{2}\alpha-1}}
    \end{align*}
    noticing that 
        \begin{align}\label{eq:alpha}
            \alpha < 1 &\implies \dfrac{3}{2}\alpha - 1 \le \dfrac{\alpha}{2},
        \end{align}
    we obtain the desired conclusion by using Proposition~\ref{prop:Dobrushin} and \eqref{eq:phi bound}. 
    \end{proof}

    Finally we turn to the 

    \begin{proof}[Proof of Theorem~\ref{th:conv}] Let $\dfrac{2}{3}< \alpha < 1, l \ge 3$ and $M$ be chosen as in Proposition~\ref{prop:sc theorem}. Using Proposition~\ref{prop:sc theorem} and then \eqref{eq:wassertein density conv} for $t=0$,
        \begin{align*}
            &\abs{\ \int_{\mathbb{R}^2} \varphi \prth{\rho_b(t) - \rho(t)}} \\
                \le& \norm{\nabla \varphi}_{L^\infty} W_1\prth{\rho_b(t), \rho(t)} \\
                \le& \norm{\nabla \varphi}_{L^\infty}  e^{2\prth{\norm{w}_{W^{2,\infty}} +  \norm{V}_{W^{2,\infty}}}t} \prth{W_1\prth{\rho_b(0), \rho(0)} + C_2(t, V, w)\prth{ l_b^{2-2\alpha} + l_b^{\frac{3}{2}\alpha-1}}} \\
                \le& \norm{\nabla \varphi}_{L^\infty}  e^{2\prth{\norm{w}_{W^{2,\infty}} +  \norm{V}_{W^{2,\infty}}}t} \\
                    &\prth{W_1\prth{\rho_{\gamma_b}(0), \rho(0)}+\ C_1(0, p, V, w)  \prth{l_b^{1-\frac{\alpha}{2}-\frac{6 + \alpha}{2p-4}} + l_b^{\frac{\alpha}{2}}} + C_2(t, V, w)\prth{ l_b^{2-2\alpha} + l_b^{\frac{3}{2}\alpha-1}}} \\
                \le& \norm{\nabla \varphi}_{L^\infty}  e^{2\prth{\norm{w}_{W^{2,\infty}} +  \norm{V}_{W^{2,\infty}}}t}\prth{1 + C_1(0, p, V, w) + C_2(t, v, w)} \\&
                \left(W_1\prth{\rho_{\gamma_b} (0), \rho(0)} + l_b^{1-\frac{\alpha}{2}-\frac{6 + \alpha}{2p-4}} + l_b^{\frac{\alpha}{2}} + l_b^{2-2\alpha} + l_b^{\frac{3}{2}\alpha-1} \right)
        \end{align*}
        On the other hand, from \eqref{eq:weak density conv bis},
        \begin{align*}
            \abs{\int_{\mathbb{R}^2} \varphi \prth{\rho_{\gamma_b}(t) - \rho_b(t)}} 
                \le C\prth{\norm{\varphi}_{L^\infty} + \norm{\nabla\varphi}_{L^2}} \prth{\abs{\Tr{\gamma_b(0) H_b(0)}} + \norm{V}_{L^\infty} + \norm{w}_{L^\infty}}l_b^{\frac{\alpha}{2}}
        \end{align*}
        Recalling~\eqref{eq:alpha} and using the triangle inequality,
        \begin{align*}
            &\abs{\ \int_{\R^2}\varphi\prth{\rho_{\gamma_b(t)} - \rho(t)}}  \\
                \le& \widetilde{C}(p, t, V, w)\prth{\norm{\varphi}_{W^{1, \infty}} + \norm{\nabla\varphi}_{L^2}}\prth{W_1\prth{\rho_{\gamma_b}(0), \rho(0)} + l_b^{1-\frac{\alpha}{2}-\frac{6 + \alpha}{2p-4}} + l_b^{2-2\alpha} + l_b^{\frac{3}{2}\alpha-1}}
        \end{align*}
        With 
        \begin{align*}
            \widetilde{C}(p, t, V, w) =& \abs{\Tr{\gamma_b(0) H_b(0)}} + \norm{V}_{L^\infty} + \norm{w}_{L^\infty} \\ &+ e^{2\prth{\norm{w}_{W^{2,\infty}} +  \norm{V}_{W^{2,\infty}}}t}\prth{1 + C_1(0, p, V, w) + C_2(t, v, w)}
        \end{align*}
        We conclude with the following optimisation:
        \begin{align*}
            \alpha \mapsto \min\prth{1-\frac{\alpha}{2}-\frac{6 + \alpha}{2p-4}, 2-2\alpha, \frac{3}{2}\alpha-1}
        \end{align*}
        is maximal at
        \begin{align*}
            \alpha \coloneq \min\prth{2\dfrac{2p-7}{4p-7}, \dfrac{6}{7}}
        \end{align*}
        with maximal value
        \begin{align*}
            \min\prth{2\dfrac{p-7}{4p-7}, \dfrac{2}{7}}.
        \end{align*}
    \end{proof}

\newpage
    
    \bibliographystyle{siam}

\end{document}